\documentclass[12pt,a4paper,draft]{amsart}
\usepackage{latexsym}
\usepackage{ifthen}
\usepackage[leqno]{amsmath}
\usepackage{enumerate}
\usepackage{calc}
\usepackage{amstext,amsbsy,amsopn,amsthm,amsgen,amsfonts,amscd,amsxtra,upref}
\usepackage[mathscr]{euscript}
\usepackage{amssymb}

\swapnumbers
\theoremstyle{plain}
\newtheorem{Thm}{Theorem}[section]
\newtheorem{Lem}[Thm]{Lemma}
\newtheorem{Cor}[Thm]{Corollary}
\newtheorem{Pro}[Thm]{Proposition}
\newtheorem{Prp}[Thm]{Properties}
\newtheorem{Sub}[Thm]{Sublemma}

\theoremstyle{definition}
\newtheorem{Def}[Thm]{Definition}
\newtheorem{Exm}[Thm]{Example}
\newtheorem{Exs}[Thm]{Examples}

\theoremstyle{remark}
\newtheorem{Rem}[Thm]{Remark}
\newtheorem{Rms}[Thm]{Remarks}
\newtheorem*{Com}{Commentary}

\newcommand{\myEmail}{piotr.niemiec@uj.edu.pl}
\newcommand{\myAddress}{\noindent{}Piotr Niemiec\\{}Jagiellonian University\\{}Institute of Mathematics\\{}%%
   ul. \L{}ojasiewicza 6\\{}30-348 Krak\'{o}w\\{}Poland}
\newcommand{\myData}{\author[P. Niemiec]{Piotr Niemiec}\address{\myAddress}\email{\myEmail}}

\newcommand{\FFF}{{\mathbb F}}
\newcommand{\GGG}{{\mathbb G}}
\newcommand{\KKK}{{\mathbb K}}

\newcommand{\PPP}{{\mathbb P}}\newcommand{\QQQ}{{\mathbb Q}}\newcommand{\RRR}{{\mathbb R}}
\newcommand{\UUU}{{\mathbb U}}
\newcommand{\VVV}{{\mathbb V}}
\newcommand{\ZZZ}{{\mathbb Z}}%% the same as '\mathbb k'
\newcommand{\CCc}{{\mathcal C}}
\newcommand{\FFf}{{\mathcal F}}

\newcommand{\OOo}{{\mathcal O}}

\newcommand{\CcC}{{\mathscr C}}
\newcommand{\EeE}{{\mathscr E}}

\newcommand{\LlL}{{\mathscr L}}

\newcommand{\PpP}{{\mathscr P}}

\newcommand{\VvV}{{\mathscr V}}

\newcommand{\Ff}{{\mathfrak F}}
\newcommand{\Gg}{{\mathfrak G}}\newcommand{\Hh}{{\mathfrak H}}
\newcommand{\Ll}{{\mathfrak L}}

\newcommand{\eE}{{\mathfrak e}}

\newcommand{\mM}{{\mathfrak m}}

\newcommand{\SECT}[1]{\section{#1}\renewcommand{\theequation}{\thesection-\arabic{equation}}\setcounter{equation}{0}}

\newcounter{help}
\newcommand{\ITE}[3]{\ifthenelse{#1}{#2}{#3}}\newcommand{\ITEE}[3]{\ITE{\equal{#1}{#2}}{#3}{}}

\newcommand{\diam}{\operatorname{diam}}
\newcommand{\dist}{\operatorname{dist}}\newcommand{\supp}{\operatorname{supp}}

\newcommand{\id}{\operatorname{id}}
\newcommand{\lin}{\operatorname{lin}}
\newcommand{\card}{\operatorname{card}}

\newcommand{\leqsl}{\leqslant}\newcommand{\geqsl}{\geqslant}

\newcommand{\epsi}{\varepsilon}\newcommand{\varempty}{\varnothing}\newcommand{\dd}{\colon}
\newcommand{\dint}[1]{\,\textup{d} #1}

\newcommand{\tfcae}{the following conditions are equivalent:}

\newcommand{\THM}[1]{Theorem~\textup{\ref{thm:#1}}}
\newcommand{\DEF}[1]{Definition~\textup{\ref{def:#1}}}\newcommand{\COR}[1]{Corollary~\textup{\ref{cor:#1}}}
\newcommand{\LEM}[1]{Lemma~\textup{\ref{lem:#1}}}\newcommand{\PRO}[1]{Proposition~\textup{\ref{pro:#1}}}
\newcommand{\REM}[1]{Remark~\textup{\ref{rem:#1}}}
\newcommand{\EXM}[1]{Example~\textup{\ref{exm:#1}}}

\newenvironment{thm}[1]{\begin{Thm}\label{thm:#1}}{\end{Thm}}\newenvironment{lem}[1]{\begin{Lem}\label{lem:#1}}{\end{Lem}}
\newenvironment{cor}[1]{\begin{Cor}\label{cor:#1}}{\end{Cor}}\newenvironment{pro}[1]{\begin{Pro}\label{pro:#1}}{\end{Pro}}
\newenvironment{dfn}[1]{\begin{Def}\label{def:#1}}{\end{Def}}
\newenvironment{exm}[1]{\begin{Exm}\label{exm:#1}}{\end{Exm}}
\newenvironment{rem}[1]{\begin{Rem}\label{rem:#1}}{\end{Rem}}

\newcommand{\bibITEM}[2]{\ITE{\equal{#2}{}}{\bibitem{#1} }{\bibitem[#2]{#1} }}
\newcommand{\BIB}[8]{%% article: hidden label, author, title, journal, number, year, pages, shown label
   \bibITEM{#1}{#8} #2, \textit{#3}, #4{} \textbf{#5} (#6), #7.}
\newcommand{\myBIB}[6][P. Niemiec]{#1, \textit{#2}, #3{}\ITE{\equal{#4}{}}{}{ \textbf{#4}} (#5), #6.}
\newcommand{\BIb}[6]{%% book: hidden label, author, title, publisher + city, year, shown label
   \bibITEM{#1}{#6} #2, \textit{#3}, #4, #5.}
\newcommand{\BiB}[9]{%% article in book: hid.label, author, title, in:, book title, publisher etc, year, pages, shown label
   \bibITEM{#1}{#9} #2, \textit{#3}, #4{} \textit{#5}, #6, #7, #8.}

\newcommand{\myBAPP}[3][P. Niemiec]{%% my paper to appear: [authors], title, comment
   #1, \textit{#2}, #3}

\newcommand{\jRN}[2][]{%%
   \ITEE{#2}{ActaM}{\ITE{\equal{#1}{+}}%%
      {Acta Mathematica}{Acta Math.}}%%
   \ITEE{#2}{AdvM}{\ITE{\equal{#1}{+}}%%
      {Advances in Mathematics}{Adv. in Math.}}%%
   \ITEE{#2}{ACS}{\ITE{\equal{#1}{+}}%%
      {Applied Categorical Structures}{Appl. Categor. Struct.}}%%
   \ITEE{#2}{ActaSM}{\ITE{\equal{#1}{+}}%%
      {Acta Scientiarum Mathematicarum}{Acta Sci. Math.}}%%
   \ITEE{#2}{AmJM}{\ITE{\equal{#1}{+}}%%
      {American Journal of Mathematics}{Amer. J. Math.}}%%
   \ITEE{#2}{AmMMon}{\ITE{\equal{#1}{+}}%%
      {American Mathematical Monthly}{Amer. Math. Mon.}}%%
   \ITEE{#2}{AnnSciEcNormSupT}{\ITE{\equal{#1}{+}}%%
      {Annales Scientifiques de l'\'{E}cole Normale Sup\'{e}rieure (3)}{Ann. Sci. \'{E}c. Norm. Sup\'{e}r. (3)}}%%
   \ITEE{#2}{AnnM}{\ITE{\equal{#1}{+}}%%
      {Annals of Mathematics}{Ann. Math.}}%%
   \ITEE{#2}{AnnProb}{\ITE{\equal{#1}{+}}%%
      {The Annals of Probability}{Ann. Probab.}}%%
   \ITEE{#2}{AnnPALog}{\ITE{\equal{#1}{+}}%%
      {Annals of Pure and Applied Logic}{Ann. Pure Appl. Logic}}%%
   \ITEE{#2}{ArchM}{\ITE{\equal{#1}{+}}%%
      {Archiv der Mathematik}{Arch. Math.}}%%
   \ITEE{#2}{AttiAccLincRendNat}{\ITE{\equal{#1}{+}}%%
      {Atti della Accademia Nazionale dei Lincei. Rendiconti. Classe di Scienze Fisiche, Matematiche e Naturali}%%
      {Atti Accad. Naz. Lincei Rend. Cl. Sci. Fis. Mat. Nat.}}%%
   \ITEE{#2}{BAMS}{\ITE{\equal{#1}{+}}%%
      {Bulletin of the American Mathematical Society}{Bull. Amer. Math. Soc.}}%%
   \ITEE{#2}{BLondMS}{\ITE{\equal{#1}{+}}%%
      {Bulletin of the London Mathematical Sociecy}{Bull. Lond. Math. Soc.}}%%
   \ITEE{#2}{BAPolSSSM}{\ITE{\equal{#1}{+}}%%
      {Bulletin de l'Acad\'{e}mie Polonaise des Sciences. S\'{e}rie des Sciences Math\'{e}matiques}%%
      {Bull. Acad. Pol. Sci. S\'{e}r. Sci. Math.}}%%
   \ITEE{#2}{BullSM}{\ITE{\equal{#1}{+}}%%
      {Bulletin des Sciences Math\'{e}matiques}{Bull. Sci. Math.}}%%
   \ITEE{#2}{BullPol}{\ITE{\equal{#1}{+}}%%
      {Bulletin of the Polish Academy of Sciences: Mathematics}{Bull. Polish Acad. Sci. Math.}}%%
   \ITEE{#2}{CanadJM}{\ITE{\equal{#1}{+}}%%
      {Canadian Journal Mathematics}{Canad. J. Math.}}%%
   \ITEE{#2}{CollectM}{\ITE{\equal{#1}{+}}%%
      {Collectanea Mathematica}{Collect. Math.}}%%
   \ITEE{#2}{CMUC}{\ITE{\equal{#1}{+}}%%
      {Commentationes Mathematicae Universitatis Carolinae}{Comment. Math. Univ. Carolin.}}%%
   \ITEE{#2}{CRASParis}{\ITE{\equal{#1}{+}}%%
      {Comptes Rendus de l'Acad\'{e}mie des Sciences. Paris}{C. R. Acad. Sci. Paris}}%%
   \ITEE{#2}{CEurJM}{\ITE{\equal{#1}{+}}%%
      {Central European Journal of Mathematics}{Cent. Eur. J. Math.}}%%
   \ITEE{#2}{CMHelv}{\ITE{\equal{#1}{+}}%%
      {Commentarii Mathematici Helvetici}{Comment. Math. Helv.}}%%
   \ITEE{#2}{CollM}{\ITE{\equal{#1}{+}}%%
      {Colloquium Mathematicum}{Coll. Math.}}%%
   \ITEE{#2}{ComposM}{\ITE{\equal{#1}{+}}%%
      {Compositio Mathematica}{Compos. Math.}}%%
   \ITEE{#2}{CzMJ}{\ITE{\equal{#1}{+}}%%
      {Czechoslovak Mathematical Journal}{Czech. Math. J.}}%%
   \ITEE{#2}{DissM}{\ITE{\equal{#1}{+}}%%
      {Dissertationes Mathematicae}{Diss. Math.}}%%
   \ITEE{#2}{DANSSSR}{\ITE{\equal{#1}{+}}%%
      {Doklady Akademii Nauk SSSR}{Dokl. Akad. Nauk SSSR}}%%
   \ITEE{#2}{DukeMJ}{\ITE{\equal{#1}{+}}%%
      {Duke Mathematical Journal}{Duke Math. J.}}%%
   \ITEE{#2}{ELA}{\ITE{\equal{#1}{+}}%%
      {The Electronic Journal of Linear Algebra}{Electron. J. Linear Algebra}}%%
   \ITEE{#2}{ExtrM}{\ITE{\equal{#1}{+}}%%
      {Extracta Mathematicae}{Extracta Math.}}%%
   \ITEE{#2}{FM}{\ITE{\equal{#1}{+}}%%
      {Fundamenta Mathematicae}{Fund. Math.}}%%
   \ITEE{#2}{FAA}{\ITE{\equal{#1}{+}}%%
      {Functional Analysis and its Applications}{Funct. Anal. Appl.}}%%
   \ITEE{#2}{FunkAnalPril}{\ITE{\equal{#1}{+}}%%
      {Funktsional'ny\u{\i} Analiz i Ego Prilozheniya}{Funkts. Anal. Prilozh.}}%%
   \ITEE{#2}{GTopA}{\ITE{\equal{#1}{+}}%%
      {General Topology and its Applications}{General Topol. Appl.}}%%
   \ITEE{#2}{IllinoisJM}{\ITE{\equal{#1}{+}}%%
      {Illinois Journal of Mathematics}{Illinois J. Math.}}%%
   \ITEE{#2}{IndagMP}{\ITE{\equal{#1}{+}}%%
      {Indagationes Mathematicae (Proceedings)}{Indagationes Math. Proc.}}%%
   \ITEE{#2}{InHauEtSPM}{\ITE{\equal{#1}{+}}%%
      {Inst. Hautes \'{E}tudes Sci. Publ. Math.}{Inst. Hautes \'{E}tudes Sci. Publ. Math.}}%%
   \ITEE{#2}{IEOT}{\ITE{\equal{#1}{+}}%%
      {Integral Equations and Operator Theory}{Integr. Equ. Oper. Theory}}%%
   \ITEE{#2}{IsraelJM}{\ITE{\equal{#1}{+}}%%
      {Israel Journal of Mathematics}{Israel J. Math.}}%%
   \ITEE{#2}{JAusMSA}{\ITE{\equal{#1}{+}}%%
      {Journal of the Australian Mathematical Society. Series A}{J. Aust. Math. Soc. Ser. A}}%%
   \ITEE{#2}{JCA}{\ITE{\equal{#1}{+}}%%
      {Journal of Convex Analysis}{J. Convex Anal.}}%%
   \ITEE{#2}{JChinUST}{\ITE{\equal{#1}{+}}%%
      {J. China Univ. Sci. Tech.}{J. China Univ. Sci. Tech.}}%%
   \ITEE{#2}{JFA}{\ITE{\equal{#1}{+}}%%
      {Journal of Functional Analysis}{J. Funct. Anal.}}%%
   \ITEE{#2}{JMAnApp}{\ITE{\equal{#1}{+}}%%
      {J. Math. Anal. Appl.}{J. Math. Anal. Appl.}}%%
   \ITEE{#2}{JOT}{\ITE{\equal{#1}{+}}%%
      {Journal of Operator Theory}{J. Operator Theory}}%%
   \ITEE{#2}{KodaiMSemRep}{\ITE{\equal{#1}{+}}%%
      {Kodai Math. Sem. Rep.}{Kodai Math. Sem. Rep.}}%%
   \ITEE{#2}{LAA}{\ITE{\equal{#1}{+}}%%
      {Linear Algebra and its Applications}{Linear Algebra Appl.}}%%
   \ITEE{#2}{LMLA}{\ITE{\equal{#1}{+}}%%
      {Linear and Multilinear Algebra}{Linear Multilinear Algebra}}%%
   \ITEE{#2}{MLQ}{\ITE{\equal{#1}{+}}%%
      {Mathematical Logic Quarterly}{Math. Log. Q.}}%%
   \ITEE{#2}{MProcCambPhS}{\ITE{\equal{#1}{+}}%%
      {Mathematical Proceedings of the Cambridge Philosophical Society}{Math. Proc. Cambridge Phil. Soc.}}%%
   \ITEE{#2}{MMag}{\ITE{\equal{#1}{+}}%%
      {Mathematics Magazine}{Math. Mag.}}%%
   \ITEE{#2}{MSb}{\ITE{\equal{#1}{+}}%%
      {Matematicheski\u{\i} Sbornik}{Mat. Sb.}}%%
   \ITEE{#2}{MStud}{\ITE{\equal{#1}{+}}%%
      {Matematychni Studi\"{\i}}{Mat. Stud.}}%%
   \ITEE{#2}{MAnn}{\ITE{\equal{#1}{+}}%%
      {Mathematische Annalen}{Math. Ann.}}%%
   \ITEE{#2}{MAMS}{\ITE{\equal{#1}{+}}%%
      {Memoirs of the American Mathematical Society}{Mem. Amer. Math. Soc.}}%%
   \ITEE{#2}{MichMJ}{\ITE{\equal{#1}{+}}%%
      {Michigan Mathematical Journal}{Mich. Math. J.}}%%
   \ITEE{#2}{MonatM}{\ITE{\equal{#1}{+}}%%
      {Monatshefte f\"{u}r Mathematik}{Mh. Math.}}%%
   \ITEE{#2}{NonlinA}{\ITE{\equal{#1}{+}}%%
      {Nonlinear Analysis: Theory, Methods \& Applications}{Nonlinear Anal.}}%%
   \ITEE{#2}{OpusM}{\ITE{\equal{#1}{+}}%%
      {Opuscula Mathematica}{Opuscula Math.}}%%
   \ITEE{#2}{PacJM}{\ITE{\equal{#1}{+}}%%
      {Pacific Journal of Mathematics}{Pacific J. Math.}}%%
   \ITEE{#2}{PeriodMHung}{\ITE{\equal{#1}{+}}%%
      {Periodica Mathematica Hungarica}{Period. Math. Hungarica}}%%
   \ITEE{#2}{PAMS}{\ITE{\equal{#1}{+}}%%
      {Proceedings of the American Mathematical Society}{Proc. Amer. Math. Soc.}}%%
   \ITEE{#2}{ProcCambPhS}{\ITE{\equal{#1}{+}}%%
      {Proceedings of the Cambridge Philosophical Society}{Proc. Cambridge Phil. Soc.}}%%
   \ITEE{#2}{ProcImpAcadTokyo}{\ITE{\equal{#1}{+}}%%
      {Proc. Imp. Acad. Tokyo}{Proc. Imp. Acad. Tokyo}}%%
   \ITEE{#2}{PLondMS}{\ITE{\equal{#1}{+}}%%
      {Proceedings of the London Mathematical Society}{Proc. London Math. Soc.}}%%
   \ITEE{#2}{PNatlUSA}{\ITE{\equal{#1}{+}}%%
      {Proceedings of the National Academy of Sciences of the United States of America}{Proc. Natl. Acad. Sci. USA}}%%
   \ITEE{#2}{PublRIMSKyoto}{\ITE{\equal{#1}{+}}%%
      {Publ. Res. Inst. Math. Sci. Kyoto Univ.}{Publ. Res. Inst. Math. Sci.}}%%
   \ITEE{#2}{PWN}{\ITE{\equal{#1}{+}}%%
      {PWN -- Polish Scientific Publishers, Warszawa}{PWN -- Polish Scientific Publishers, Warszawa}}%%
   \ITEE{#2}{RCMP}{\ITE{\equal{#1}{+}}%%
      {Rendiconti del Circolo Matematico di Palermo}{Rend. Circ. Mat. Palermo}}%%
   \ITEE{#2}{RussMS}{\ITE{\equal{#1}{+}}%%
      {Russian Mathematical Surveys}{Russian Math. Surveys}}%%
   \ITEE{#2}{SbM}{\ITE{\equal{#1}{+}}%%
      {Sbornik: Mathematics}{Sb. Math.}}%%
   \ITEE{#2}{SciRepTokyoA}{\ITE{\equal{#1}{+}}%%
      {Science Reports of Tokyo Kyoiku Daigaku, Section A}{Sci. Rep. Tokyo Kyoiku Daigaku Sect. A}}%%
   \ITEE{#2}{SeminProbStras}{\ITE{\equal{#1}{+}}%%
      {S\'{e}minaire de probabilit\'{e}s de Strasbourg}{S\'{e}min. Prob. Strasbourg}}%%
   \ITEE{#2}{SIAMJMAA}{\ITE{\equal{#1}{+}}%%
      {SIAM Journal on Matrix Analysis and Applications}{SIAM J. Matrix Anal. Appl.}}%%
   \ITEE{#2}{SibirMZ}{\ITE{\equal{#1}{+}}%%
      {Sibirski\v{\i} Mat. \v{Z}hurnal}{Sibirsk. Mat. \v{Z}.}}%%
   \ITEE{#2}{SM}{\ITE{\equal{#1}{+}}%%
      {Studia Mathematica}{Studia Math.}}%%
   \ITEE{#2}{TAMS}{\ITE{\equal{#1}{+}}%%
      {Transactions of the American Mathematical Society}{Trans. Amer. Math. Soc.}}%%
   \ITEE{#2}{TomskUnivRev}{\ITE{\equal{#1}{+}}%%
      {Tomsk Universitet Review}{Tomsk. Univ. Rev.}}%%
   \ITEE{#2}{TopA}{\ITE{\equal{#1}{+}}%%
      {Topology and its Applications}{Topology Appl.}}%%
   \ITEE{#2}{TsukubaJM}{\ITE{\equal{#1}{+}}%%
      {Tsukuba Journal of Mathematics}{Tsukuba J. Math.}}%%
   \ITEE{#2}{UspekhiMN}{\ITE{\equal{#1}{+}}%%
      {Uspekhi Matem. Nauk}{Uspekhi Mat. Nauk}}%%
   }

\newcommand{\paplist}[3][]{%%
   \ITEE{#3}{NIAkhiezer,IMGlazman1993}{%%
      \BIb{#2}{N.I. Akhiezer and I.M. Glazman}%%
         {Theory of Linear Operators in Hilbert Space}%%
         {Dover Publications, Inc., New York}{1993}{#1}}%%
   \ITEE{#3}{RDAnderson1967}{%%
      \BIB{#2}{R.D. Anderson}%%
         {On topological infinite deficiency}%%
         {\jRN{MichMJ}}{14}{1967}{365--383}{#1}}%%
   \ITEE{#3}{RDAnderson,JMcCharen1970}{%%
      \BIB{#2}{R.D. Anderson and J. McCharen}%%
         {On extending homeomorphisms to Fr\'{e}chet manifolds}%%
         {\jRN{PAMS}}{25}{1970}{283--289}{#1}}%%
   \ITEE{#3}{RDAnderson,DWCurtis,JVanMill1982}{%%
      \BIB{#2}{R.D. Anderson, D.W. Curtis, J. van Mill}%%
         {A fake topological Hilbert space}%%
         {\jRN{TAMS}}{272}{1982}{311--321}{#1}}%%
   \ITEE{#3}{RArens,JEells1956}{%%
      \BIB{#2}{R. Arens and J. Eells}%%
         {On embedding uniform and topological spaces}%%
         {\jRN{PacJM}}{6}{1956}{397--403}{#1}}%%
   \ITEE{#3}{NAronszajn,PPanitchpakdi1956}{%%
      \BIB{#2}{N. Aronszajn and P. Panitchpakdi}%%
         {Extension of uniformly continuous transformations and hyperconvex metric spaces}%%
         {\jRN{PacJM}}{6}{1956}{405--439}{#1}}%%
   \ITEE{#3}{KJBabenko1948}{%%
      \BIB{#2}{K.J. Babenko}%%
         {On conjugate functions}%%
         {\jRN{DANSSSR}}{62}{1948}{157--160}{#1}}%%
   \ITEE{#3}{TBanakh1995}{%%
      \BIB{#2}{T.O. Banakh}%%
         {Topology of spaces of probability measures, I}%%
         {\jRN{MStud}}{5}{1995}{65--87 (Russian)}{#1}}%%
   \ITEE{#3}{TBanakh1995a}{%%
      \BIB{#2}{T.O. Banakh}%%
         {Topology of spaces of probability measures, II}%%
         {\jRN{MStud}}{5}{1995}{88--106 (Russian)}{#1}}%%
   \ITEE{#3}{TBanakh1998}{%%
      \BIB{#2}{T. Banakh}%%
         {Characterization of spaces admitting a homotopy dense embedding into a Hilbert manifold}%%
         {\jRN{TopA}}{86}{1998}{123--131}{#1}}%%
   \ITEE{#3}{TBanakh,TNRadul1997}{%%
      \BIB{#2}{T.O. Banakh and T.N. Radul}%%
         {Topology of spaces of probability measures}%%
         {\jRN{SbM}}{188}{1997}{973--995}{#1}}%%
   \ITEE{#3}{TBanakh,TRadul,MZarichnyi1996}{%%
      \BIb{#2}{T. Banakh, T. Radul, M. Zarichnyi}%%
         {Absorbing sets in infinite-dimensional manifolds}%%
         {VNTL Publishers, Lviv}{1996}{#1}}%%
   \ITEE{#3}{TBanakh,IZarichnyy2008}{%%
      \BIB{#2}{T. Banakh and I. Zarichnyy}%%
         {Topological groups and convex sets homeomorphic to non-separable Hilbert spaces}%%
         {\jRN{CEurJM}}{6}{2008}{77--86}{#1}}%%
   \ITEE{#3}{HBecker,ASKechris1996}{%%
      \BIb{#2}{H. Becker and A.S. Kechris}%%
         {The Descriptive Set Theory of Polish Group Actions \textup{(London Math. Soc. Lecture Note Series, vol. 232)}}%%
         {University Press, Cambridge}{1996}{#1}}%%
   \ITEE{#3}{NEBenamara,NNikolski1999}{%%
      \BIB{#2}{N.E. Benamara and N. Nikolski}%%
         {Resolvent tests for similarity to a normal operator}%%
         {\jRN{PLondMS}}{78}{1999}{585--626}{#1}}%%
   \ITEE{#3}{YBenyamini,JLindenstrauss2000}{%%
      \BIb{#2}{Y. Benyamini and J. Lindenstrauss}%%
         {Geometric nonlinear functional analysis I}%%
         {AMS Colloquium Publications 48}{2000}{#1}}%%
   \ITEE{#3}{SKBerberian1974}{%%
      \BIb{#2}{S.K. Berberian}%%
         {Lectures in Functional Analysis and Operator Theory}%%
         {Graduate Texts in Mathematics 15, Springer-Verlag, New York}{1974}{#1}}%%
   \ITEE{#3}{SNBernstein1954}{%%
      \BIb{#2}{S.N. Bernstein}%%
         {Collected Works II}%%
         {Akad. Nauk SSSR, Moscow}{1954 (Russian)}{#1}}%%
   \ITEE{#3}{CzBessaga,APelczynski1972}{%%
      \BIB{#2}{Cz. Bessaga and A. Pe\l{}czy\'{n}ski}%%
         {On spaces of measurable functions}%%
         {\jRN{SM}}{44}{1972}{597--615}{#1}}%%
   \ITEE{#3}{CzBessaga,APelczynski1975}{%%
      \BIb{#2}{Cz. Bessaga and A. Pe\l{}czy\'{n}ski}%%
         {Selected topics in infinite-dimensional topology}%%
         {\jRN{PWN}}{1975}{#1}}%%
   \ITEE{#3}{MBestvina,JMogilski1986}{%%
      \BIB{#2}{M. Bestvina and J. Mogilski}%%
         {Characterizing certain incomplete infinite-dimensional absolute retracts}%%
         {\jRN{MichMJ}}{33}{1986}{291--313}{#1}}%%
   \ITEE{#3}{MBestvina,PBowers,JMogilsky,JWalsh1986}{%%
      \BIB{#2}{M. Bestvina, P. Bowers, J. Mogilsky, J. Walsh}%%
         {Characterization of Hilbert space manifolds revisited}%%
         {\jRN{TopA}}{24}{1986}{53--69}{#1}}%%
   \ITEE{#3}{RBhatia1997}{%%
      \BIb{#2}{R. Bhatia}%%
         {Matrix Analysis}%%
         {Springer, New York}{1997}{#1}}%%
   \ITEE{#3}{GBirkhoff1936}{%%
      \BIB{#2}{G. Birkhoff}%%
         {A note on topological groups}%%
         {\jRN{ComposM}}{3}{1936}{427--430}{#1}}%%
   \ITEE{#3}{MSBirman,MZSolomjak1987}{%%
      \BIb{#2}{M.S. Birman and M.Z. Solomjak}%%
         {Spectral Theory of Self-Adjoint Operators in Hilbert Space}%%
         {D. Reidel Publishing Co., Dordrecht}{1987}{#1}}%%
   \ITEE{#3}{EBishop1961}{%%
      \BIB{#2}{E. Bishop}%%
         {A generalization of the Stone-Weierstrass theorem}%%
         {\jRN{PacJM}}{11}{1961}{777--783}{#1}}%%
   \ITEE{#3}{JBlass,WHolsztynski1972}{%%
      \BIB{#2}{J. Blass and W. Holszty\'{n}ski}%%
         {Cubical polyhedra and homotopy III}%%
         {\jRN{AttiAccLincRendNat}}{53}{1972}{275--279}{#1}}%%
   \ITEE{#3}{FFBonsall,NJDuncan1973}{%%
      \BIb{#2}{F.F. Bonsall and N.J. Duncan}%%
         {Complete Normed Algebras}%%
         {Springer Verlag, Berlin}{1973}{#1}}%%
   \ITEE{#3}{NBourbaki2002}{%%
      \BIb{#2}{N. Bourbaki}%%
         {Lie Groups and Lie Algebras, Chapters 4--6}%%
         {Springer, New York}{2002}{#1}}%%
   \ITEE{#3}{PLBowers1989}{%%
      \BIB{#2}{P.L. Bowers}%%
         {Limitation topologies on function spaces}%%
         {\jRN{TAMS}}{314}{1989}{421--431}{#1}}%%
   \ITEE{#3}{AMBruckner,JBBruckner,BSThomson1997}{%%
      \BIb{#2}{A.M. Bruckner, J.B. Bruckner, B.S. Thomson}%%
         {Real Analysis}%%
         {Prentice-Hall, New Jersey}{1997}{#1}}%%
   \ITEE{#3}{PJCameron,AMVershik2006}{%%
      \BIB{#2}{P.J. Cameron and A.M. Vershik}%%
         {Some isometry groups of Urysohn space}%%
         {\jRN{AnnPALog}}{143}{2006}{70--78}{#1}}%%
   \ITEE{#3}{JAVanCasteren1980}{%%
      \BIB{#2}{J.A. van Casteren}%%
         {A problem of Sz.-Nagy}%%
         {\jRN{ActaSM}}{42}{1980}{189--194}{#1}}%%
   \ITEE{#3}{JAVanCasteren1983}{%%
      \BIB{#2}{J.A. van Casteren}%%
         {Operators similar to unitary or selfadjoint ones}%%
         {\jRN{PacJM}}{104}{1983}{241--255}{#1}}%%
   \ITEE{#3}{RCauty1994}{%%
      \BIB{#2}{R. Cauty}%%
         {Un espace m\'{e}trique lin\'{e}aire qui n'est pas un r\'{e}tracte absolu}%%
         {\jRN{FM}}{146}{1994}{85--99, (French)}{#1}}%%
   \ITEE{#3}{TAChapman1971}{%%
      \BIB{#2}{T.A. Chapman}%%
         {Deficiency in infinite-dimensional manifolds}%%
         {\jRN{GTopA}}{1}{1971}{263--272}{#1}}%%
   \ITEE{#3}{TAChapman1976}{%%
      \BIb{#2}{T.A. Chapman}%%
         {Lectures on Hilbert cube manifolds}%%
         {C.B.M.S. Regional Conference Series in Math. No 28, Amer. Math. Soc.}{1976}{#1}}%%
   \ITEE{#3}{RBChuaqui1977}{%%
      \BIB{#2}{R.B. Chuaqui}%%
         {Measures invariant under a group of transformations}%%
         {\jRN{PacJM}}{68}{1977}{313--329}{#1}}%%
   \ITEE{#3}{JBConway1985}{%%
      \BIb{#2}{J.B. Conway}%%
         {A Course in Functional Analysis}%%
         {Springer-Verlag, New York}{1985}{#1}}%%
   \ITEE{#3}{JBConway2000}{%%
      \BIb{#2}{J.B. Conway}%%
         {A Course in Operator Theory}%%
         {(Graduate Studies in Mathematics, vol. 21) Amer. Math. Soc., Providence}{2000}{#1}}%%
   \ITEE{#3}{GCorach,AMaestripieri,MMbekhta2009}{%%
      \BIB{#2}{G. Corach, A. Maestripieri, M. Mbekhta}%%
         {Metric and homogeneous structure of closed range operators}%%
         {\jRN{JOT}}{61}{2009}{171--190}{#1}}%%
   \ITEE{#3}{MJCowen,RGDouglas1978}{%%
      \BIB{#2}{M.J. Cowen and R.G. Douglas}%%
         {Complex geometry and operator theory}%%
         {\jRN{ActaM}}{141}{1978}{187--261}{#1}}%%
   \ITEE{#3}{DWCurtis1985}{%%
      \BIB{#2}{D.W. Curtis}%%
         {Boundary sets in the Hilbert cube}%%
         {\jRN{TopA}}{20}{1985}{201--221}{#1}}%%
   \ITEE{#3}{MMDay1958}{%%
      \BIb{#2}{M.M. Day}%%
         {Normed Linear Spaces}%%
         {Springer Verlag, Berlin}{1958}{#1}}%%
   \ITEE{#3}{CDellacherie1967}{%%
      \BIB{#2}{C. Dellacherie}%%
         {Un compl\'{e}ment au th\'{e}or\`{e}me de Weierstrass-Stone}%%
         {\jRN{SeminProbStras}}{1}{1967}{52--53}{#1}}%%
   \ITEE{#3}{JJDijkstra1987}{%%
      \BIB{#2}{J.J. Dijkstra}%%
         {Strong negligibility of $\sigma$-compacta does not characterize Hilbert space}%%
         {\jRN{PacJM}}{127}{1987}{19--30}{#1}}%%
   \ITEE{#3}{JJDijkstra1990}{%%
      \BIB{#2}{J.J. Dijkstra}%%
         {Characterizing Hilbert space topology in terms of strong negligibility}%%
         {\jRN{ComposM}}{75}{1990}{299--306}{#1}}%%
   \ITEE{#3}{TDobrowolski,WMarciszewski2002}{%%
      \BIB{#2}{T. Dobrowolski and W. Marciszewski}%%
         {Failure of the Factor Theorem for Borel pre-Hilbert spaces}%%
         {\jRN{FM}}{175}{2002}{53--68}{#1}}%%
   \ITEE{#3}{TDobrowolski,JMogilski1990}{%%
      \BiB{#2}{T. Dobrowolski and J. Mogilski}%%
         {Problems on Topological Classification of Incomplete Metric Spaces}{Chapter 25 in:}%%
         {Open Problems in Topology}{J. van Mill and G.M. Reed (eds.), North-Holland Amsterdam}{1990}{411--429}{#1}}%%
   \ITEE{#3}{TDobrowolski,HTorunczyk1981}{%%
      \BIB{#2}{T. Dobrowolski and H. Toru\'{n}czyk}%%
         {Separable complete ANR's admitting a group structure are Hilbert manifolds}%%
         {\jRN{TopA}}{12}{1981}{229--235}{#1}}%%
   \ITEE{#3}{RGDouglas1966}{%%
      \BIB{#2}{R.G. Douglas}%%
         {On majorization, factorization and range inclusion of operators in Hilbert space}%%
         {\jRN{PAMS}}{17}{1966}{413--416}{#1}}%%
   \ITEE{#3}{CHDowker1947}{%%
      \BIB{#2}{C.H. Dowker}%%
         {Mapping theorems for non-compact spaces}%%
         {\jRN{AmJM}}{69}{1947}{200--242}{#1}}%%
   \ITEE{#3}{CHDowker1952}{%%
      \BIB{#2}{C.H. Dowker}%%
         {Topology of metric complexes}%%
         {\jRN{AmJM}}{74}{1952}{555--577}{#1}}%%
   \ITEE{#3}{JDugundji1951}{%%
      \BIB{#2}{J. Dugundji}%%
         {An extension of Tietze's theorem}%%
         {\jRN{PacJM}}{1}{1951}{353--367}{#1}}%%
   \ITEE{#3}{JDugundji1958}{%%
      \BIB{#2}{J. Dugundji}%%
         {Absolute neighborhood retracts and local connectedness for arbitrary metric spaces}%%
         {\jRN{ComposM}}{13}{1958}{229--246}{#1}}%%
   \ITEE{#3}{JDugundji1965}{%%
      \BIB{#2}{J. Dugundji}%%
         {Locally equiconnected spaces and absolute neighborhood retracts}%%
         {\jRN{FM}}{57}{1965}{187--193}{#1}}%%
   \ITEE{#3}{NDunford,JTSchwartz1958}{%%
      \BIb{#2}{N. Dunford and J.T. Schwartz}%%
         {Linear Operators, part I}%%
         {Interscience Publishers, New York}{1958}{#1}}%%
   \ITEE{#3}{NDunford,JTSchwartz1963}{%%
      \BIb{#2}{N. Dunford and J.T. Schwartz}%%
         {Linear Operators, part II}%%
         {Interscience Publishers, New York}{1963}{#1}}%%
   \ITEE{#3}{NDunford,JTSchwartz1971}{%%
      \BIb{#2}{N. Dunford and J.T. Schwartz}%%
         {Linear Operators, part III}%%
         {Wiley-Interscience, New York}{1971}{#1}}%%
   \ITEE{#3}{MLEaton,MDPerlman1977}{%%
      \BIB{#2}{M.L. Eaton and M.D. Perlman}%%
         {Reflection groups, generalized Schur functions and the geometry of majorization}%%
         {\jRN{AnnProb}}{5}{1977}{829--860}{#1}}%%
   \ITEE{#3}{EGEffros1966}{%%
      \BIB{#2}{E.G. Effros}%%
         {Global structure in von Neumann algebras}%%
         {\jRN{TAMS}}{121}{1966}{434--454}{#1}}%%
   \ITEE{#3}{REspinola,MAKhamsi2001}{%%
      \BiB{#2}{R. Espinola and M.A. Khamsi}%%
         {Introduction to hyperconvex spaces}{Chapter XIII in:}{Handbook of Metric Fixed Point Theory}%%
         {W.A. Kirk and B. Sims (editors), Kluwer Academic Publishers}{2001}{391--435}{#1}}%%
   \ITEE{#3}{PAFillmore,JPWilliams1971}{%%
      \BIB{#2}{P.A. Fillmore and J.P. Williams}%%
         {On operator ranges}%%
         {\jRN{AdvM}}{7}{1971}{254--281}{#1}}%%
   \ITEE{#3}{JEells,NHKuiper1969}{%%
      \BIB{#2}{J. Eells and N.H. Kuiper}%%
         {Homotopy negligible subsets in infinite-dimensional manifolds}%%
         {\jRN{ComposM}}{21}{1969}{151--161}{#1}}%%
   \ITEE{#3}{REngelking1977}{%%
      \BIb{#2}{R. Engelking}%%
         {General Topology}%%
         {\jRN{PWN}}{1977}{#1}}%%
   \ITEE{#3}{REngelking1978}{%%
      \BIb{#2}{R. Engelking}%%
         {Dimension Theory}%%
         {\jRN{PWN}}{1978}{#1}}%%
   \ITEE{#3}{PErdos,RDMauldin1976}{%%
      \BIB{#2}{P. Erd\"{o}s and R.D. Mauldin}%%
         {The nonexistence of certain invariant measures}%%
         {\jRN{PAMS}}{59}{1976}{321--322}{#1}}%%
   \ITEE{#3}{RHFox1943}{%%
      \BIB{#2}{R.H. Fox}%%
         {On fiber spaces, II}%%
         {\jRN{BAMS}}{49}{1943}{733--735}{#1}}%%
   \ITEE{#3}{NAFriedman1970}{%%
      \BIb{#2}{N.A. Friedman}%%
         {Introduction to ergodic theory}%%
         {Van Nostrand Reinhold Company}{1970}{#1}}%%
   \ITEE{#3}{SGao,ASKechris2003}{%%
      \BIB{#2}{S. Gao and A.S. Kechris}%%
         {On the classification of Polish metric spaces up to isometry}%%
         {\jRN{MAMS}}{161}{2003}{viii+78}{#1}}%%
   \ITEE{#3}{MIGarrido,FMontalvo1991}{%%
      \BIB{#2}{M.I. Garrido and F. Montalvo}%%
         {On some generalizations of the Kakutani-Stone and Stone-Weierstrass theorems}%%
         {\jRN{ExtrM}}{6}{1991}{156--159}{#1}}%%
   \ITEE{#3}{IMGelfand,MANaimark1943}{%%
      \BIB{#2}{I.M. Gelfand and M.A. Naimark}%%
         {On the embedding of normed rings into the ring of operators in Hilbert space}%%
         {\jRN{MSb}}{12}{1943}{197--213}{#1}}%%
   \ITEE{#3}{FGesztesy,MMalamud,MMitrea,SNaboko2009}{%%
      \BIB{#2}{F. Gesztesy, M. Malamud, M. Mitrea, S. Naboko}%%
         {Generalized polar decompositions for closed operators in Hilbert spaces and some applications}%%
         {\jRN{IEOT}}{64}{2009}{83--113}{#1}}%%
   \ITEE{#3}{JGlimm1960}{%%
      \BIB{#2}{J. Glimm}%%
         {A Stone-Weierstrass theorem for $\CCc^*$-algebras}%%
         {\jRN{AnnM}}{72}{1960}{216--244}{#1}}%%
   \ITEE{#3}{GGodefroy,NJKalton2003}{%%
      \BIB{#2}{G. Godefroy and N.J. Kalton}%%
         {Lipschitz-free Banach spaces}%%
         {\jRN{SM}}{159}{2003}{121--141}{#1}}%%
   \ITEE{#3}{ICGohberg,MGKrein1967}{%%
      \BIB{#2}{I.C. Gohberg and M.G. Krein}%%
         {On a description of contraction operators similar to unitary ones}%%
         {\jRN{FunkAnalPril}}{1}{1967}{38--60}{#1}}%%
   \ITEE{#3}{MGromov1981}{%%
      \BIB{#2}{M. Gromov}%%
         {Groups of polynomial growth and expanding maps}%%
         {\jRN{InHauEtSPM}}{53}{1981}{53--73}{#1}}%%
   \ITEE{#3}{MGromov1999}{%%
      \BIb{#2}{M. Gromov}%%
         {Metric Structures for Riemannian and Non-Riemannian Spaces}%%
         {Progress in Math. \textbf{152}, Birkh\"{a}user}{1999}{#1}}%%
   \ITEE{#3}{JDeGroot1956}{%%
      \BIB{#2}{J. de Groot}%%
         {Non-archimedean metrics in topology}%%
         {\jRN{PAMS}}{7}{1956}{948--953}{#1}}%%
   \ITEE{#3}{LCGrove,CTBenson1985}{%%
      \BIb{#2}{L.C. Grove and C.T. Benson}%%
         {Finite Reflection Group}%%
         {2nd ed., Springer-Verlag}{1985}{#1}}%%
   \ITEE{#3}{VIGurarii1966}{%%
      \BIB{#2}{V.I. Gurari\v{\i}}{Spaces of universal placement, isotropic spaces and a problem of Mazur %%
         on rotations of Banach spaces \textup{(Russian)}}%%
         {\jRN{SibirMZ}}{7}{1966}{1002--1013}{#1}}%%
   \ITEE{#3}{HHahn1932}{%%
      \BIb{#2}{H. Hahn}%%
         {Reelle Funktionen I}%%
         {Leipzig}{1932}{#1}}%%
   \ITEE{#3}{PRHalmos1950}{%%
      \BIb{#2}{P.R. Halmos}%%
         {Measure theory}%%
         {Van Nostrand, New York}{1950}{#1}}%%
   \ITEE{#3}{PRHalmos1951}{%%
      \BIb{#2}{P.R. Halmos}%%
         {Introduction to Hilbert Space and the Theory of Spectral Multiplicity}%%
         {Chelsea Publishing Company, New York}{1951}{#1}}%%
   \ITEE{#3}{PRHalmos1956}{%%
      \BIb{#2}{P.R. Halmos}%%
         {Lectures on Ergodic Theory}%%
         {Publ. Math. Soc. Japan, Tokyo}{1956}{#1}}%%
   \ITEE{#3}{PRHalmos1982}{%%
      \BIb{#2}{P.R. Halmos}%%
         {A Hilbert Space Problem Book}%%
         {Springer-Verlag New York Inc.}{1982}{#1}}%%
   \ITEE{#3}{RWHansell1972}{%%
      \BIB{#2}{R.W. Hansell}%%
         {On the nonseparable theory of Borel and Souslin sets}%%
         {\jRN{BAMS}}{78}{1972}{236--241}{#1}}%%
   \ITEE{#3}{FHausdorff1930}{%%
      \BIB{#2}{F. Hausdorff}%%
         {Erweiterung einer Hom\"{o}omorphie}%%
         {\jRN{FM}}{16}{1930}{353--360}{#1}}%%
   \ITEE{#3}{FHausdorff1934}{%%
      \BIB{#2}{F. Hausdorff}%%
         {\"{U}ber innere Abbildungen}%%
         {\jRN{FM}}{23}{1934}{279--291}{#1}}%%
   \ITEE{#3}{FHausdorff1938}{%%
      \BIB{#2}{F. Hausdorff}%%
         {Erweiterung einer stetigen Abbildung}%%
         {\jRN{FM}}{30}{1938}{40--47}{#1}}%%
   \ITEE{#3}{DWHenderson1971}{%%
      \BIB{#2}{D.W. Henderson}%%
         {Corrections and extensions of two papers about infinite-dimensional manifolds}%%
         {\jRN{GTopA}}{1}{1971}{321--327}{#1}}%%
   \ITEE{#3}{DWHenderson1975}{%%
      \BIB{#2}{D.W. Henderson}%%
         {$Z$-sets in ANR's}%%
         {\jRN{TAMS}}{213}{1975}{205--216}{#1}}%%
   \ITEE{#3}{DWHenderson,RMSchori1970}{%%
      \BIB{#2}{D.W. Henderson and R.M. Schori}%%
         {Topological classification of infinite-dimensional manifolds by homotopy type}%%
         {\jRN{BAMS}}{76}{1970}{121--124}{#1}}%%
   \ITEE{#3}{DWHenderson,JEWest1970}{%%
      \BIB{#2}{D.W. Henderson and J.E. West}%%
         {Triangulated infinite-dimensional manifolds}%%
         {\jRN{BAMS}}{76}{1970}{655--660}{#1}}%%
   \ITEE{#3}{BHoffmann1979}{%%
      \BIB{#2}{B. Hoffmann}%%
         {A compact contractible topological group is trivial}%%
         {\jRN{ArchM}}{32}{1979}{585--587}{#1}}%%
   \ITEE{#3}{DHofmann2002}{%%
      \BIB{#2}{D. Hofmann}%%
         {On a generalization of the Stone-Weierstrass theorem}%%
         {\jRN{ACS}}{10}{2002}{569--592}{#1}}%%
   \ITEE{#3}{GHognas,AMukherjea1995}{%%
      \BIb{#2}{G. H\"ogn\"as and A. Mukherjea}%%
         {Probability Measures on Semigroups. Convolution Products, Random Walks, and Random Matrices}%%
         {Plenum Press, New York}{1995}{#1}}%%
   \ITEE{#3}{MRHolmes1992}{%%
      \BIB{#2}{M.R. Holmes}%%
         {The universal separable metric space of Urysohn and isometric embeddings thereof in Banach spaces}%%
         {\jRN{FM}}{140}{1992}{199--223}{#1}}%%
   \ITEE{#3}{MRHolmes2008}{%%
      \BIB{#2}{M.R. Holmes}%%
         {The Urysohn space embeds in Banach spaces in just one way}%%
         {\jRN{TopA}}{155}{2008}{1479--1482}{#1}}%%
   \ITEE{#3}{RRHolmes,TYTam1999}{%%
      \BIB{#2}{R.R. Holmes and T.Y. Tam}%%
         {Distance to the convex hull of an orbit under the action of a compact group}%%
         {\jRN{JAusMSA}}{66}{1999}{331--357}{#1}}%%
   \ITEE{#3}{RHorn,RMathias1990}{%%
      \BIB{#2}{R. Horn and R. Mathias}%%
         {Cauchy-Schwartz inequalities associated with positive semidefinite matrices}%%
         {\jRN{LAA}}{142}{1990}{63--82}{#1}}%%
   \ITEE{#3}{GEHuhunaisvili1955}{%%
      \BIB{#2}{G.E. Huhunai\v{s}vili}%%
         {On a property of Urysohn's universal metric space}%%
         {\jRN{DANSSSR}}{101}{1955}{607--610 (Russian)}{#1}}%%
   \ITEE{#3}{JEHumphreys1990}{%%
      \BIb{#2}{J.E. Humphreys}%%
         {Reflection Groups and Coxeter Groups}%%
         {Cambridge University Press}{1990}{#1}}%%
   \ITEE{#3}{JRIsbell1964}{%%
      \BIB{#2}{J.R. Isbell}%%
         {Six theorems about injective metric spaces}%%
         {\jRN{CMHelv}}{39}{1964}{65--76}{#1}}%%
   \ITEE{#3}{SIzumino,YKato1985}{%%
      \BIB{#2}{S. Izumino and Y. Kato}%%
         {The closure of invertible operators on Hilbert space}%%
         {\jRN{ActaSM}}{49}{1985}{321--327}{#1}}%%
   \ITEE{#3}{CJiang2004}{%%
      \BIB{#2}{C. Jiang}%%
         {Similarity classification of Cowen-Douglas operators}%%
         {\jRN{CanadJM}}{56}{2004}{742--775}{#1}}%%
   \ITEE{#3}{WBJohnson,JLindenstrauss2001}{%%
      \BiB{#2}{W.B. Johnson and J. Lindenstrauss}{Basic Concepts in the Geometry of Banach Spaces}%%
         {Chapter 1 in:}{Handbook of the Geometry of Banach Spaces, Vol. 1}%%
         {W.B. Johnson and J. Lindenstrauss (editors), Elsevier Science B.V., Amsterdam}{2001}{1--84}{#1}}%%
   \ITEE{#3}{IBJung,JStochel2008}{%%
      \BIB{#2}{I.B. Jung and J. Stochel}%%
         {Subnormal operators whose adjoints have rich point spectrum}%%
         {\jRN{JFA}}{255}{2008}{1797--1816}{#1}}%%
   \ITEE{#3}{SKakutani1936}{%%
      \BIB{#2}{S. Kakutani}%%
         {\"{U}ber die Metrisation der topologischen Gruppen}%%
         {\jRN{ProcImpAcadTokyo}}{12}{1936}{82--84}{#1}}%%
   \ITEE{#3}{SKakutani1938}{%%
      \BIB{#2}{S. Kakutani}%%
         {Two fixed-point theorems concerning bicompact convex sets}%%
         {\jRN{ProcImpAcadTokyo}}{14}{1938}{242--245}{#1}}%%
   \ITEE{#3}{SKakutani1941}{%%
      \BIB{#2}{S. Kakutani}%%
         {Concrete representation of abstract L-spaces}%%
         {\jRN{AnnM}}{42}{1941}{523--537}{#1}}%%
   \ITEE{#3}{SKakutani1941a}{%%
      \BIB{#2}{S. Kakutani}%%
         {Concrete representation of abstract M-spaces}%%
         {\jRN{AnnM}}{42}{1941}{994--1024}{#1}}%%
   \ITEE{#3}{NKalton2007}{%%
      \BIB{#2}{N. Kalton}%%
         {Extending Lipschitz maps into $\CCc(K)$-spaces}%%
         {\jRN{IsraelJM}}{162}{2007}{275--315}{#1}}%%
   \ITEE{#3}{RKane2001}{%%
      \BIb{#2}{R. Kane}%%
         {Reflection Groups and Invariant Theory}%%
         {Canadian Mathematical Society, Springer}{2001}{#1}}%%
   \ITEE{#3}{VKannan,SRRaju1980}{%%
      \BIB{#2}{V. Kannan and S.R. Raju}%%
         {The nonexistence of invariant universal measures on semigroups}%%
         {\jRN{PAMS}}{78}{1980}{482--484}{#1}}%%
   \ITEE{#3}{MKatetov1988}{%%
      \BiB{#2}{M. Kat\v{e}tov}{On universal metric spaces}{in: Frolik (ed.),}%%
         {General Topology and its Relations to Modern Analysis and Algebra VI. Proceedings of the Sixth Prague %%
         Topological Symposium 1986}{Heldermann Verlag Berlin}{1988}{323--330}{#1}}%%
   \ITEE{#3}{YKatznelson1960}{%%
      \BIB{#2}{Y. Katznelson}%%
         {Sur les alg\'{e}bres dont les \'{e}l\'{e}ments non n\'{e}gatifs admettent des racines carr\'{e}es}%%
         {\jRN{AnnSciEcNormSupT}}{77}{1960}{167--174}{#1}}%%
   \ITEE{#3}{OHKeller1931}{%%
      \BIB{#2}{O.H. Keller}%%
         {Die Homoiomorphie der kompakten konvexen Mengen in Hilbertschen Raum}%%
         {\jRN{MAnn}}{105}{1931}{748--758}{#1}}%%
   \ITEE{#3}{MAKhamsi,WAKirk,CMartinez2000}{%%
      \BIB{#2}{M.A. Khamsi, W.A. Kirk, C. Martinez}%%
         {Fixed point and selection theorems in hyperconvex spaces}%%
         {\jRN{PAMS}}{128}{2000}{3275--3283}{#1}}%%
   \ITEE{#3}{ABKhararazishvili1998}{%%
      \BIb{#2}{A.B. Khararazishvili}%%
         {Transformation groups and invariant measures. Set-theoretic aspects}%%
         {World Scientific Publishing Co., Inc., River Edge, NJ}{1998}{#1}}%%
   \ITEE{#3}{YKijima1987}{%%
      \BIB{#2}{Y. Kijima}%%
         {Fixed points of nonexpansive self-maps of a compact metric space}%%
         {\jRN{JMAnApp}}{123}{1987}{114--116}{#1}}%%
   \ITEE{#3}{JKindler1995}{%%
      \BIB{#2}{J. Kindler}%%
         {Minimax theorems with applications to convex metric spaces}%%
         {\jRN{CollM}}{68}{1995}{179--186}{#1}}%%
   \ITEE{#3}{WAKirk1998}{%%
      \BIB{#2}{W.A. Kirk}%%
         {Hyperconvexity of $\RRR$-trees}%%
         {\jRN{FM}}{156}{1998}{67--72}{#1}}%%
   \ITEE{#3}{VLKleeJr1952}{%%
      \BIB{#2}{V.L. Klee Jr.}%%
         {Invariant metrics in groups (solution of a problem of Banach)}%%
         {\jRN{PAMS}}{3}{1952}{484--487}{#1}}%%
   \ITEE{#3}{HJKowalsky1957}{%%
      \BIB{#2}{H.J. Kowalsky}%%
         {Einbettung metrischer R\"{a}ume}%%
         {\jRN{ArchM}}{8}{1957}{336--339}{#1}}%%
   \ITEE{#3}{WKubis,MRubin2010}{%%
      \BIB{#2}{W. Kubi\'{s} and M. Rubin}%%
         {Extension and reconstruction theorems for the Urysohn universal metric space}%%
         {\jRN{CzMJ}}{60}{2010}{1--29}{#1}}%%
   \ITEE{#3}{KKuratowski1966}{%%
      \BIb{#2}{K. Kuratowski}%%
         {Topology. \textup{Vol. I}}%%
         {\jRN{PWN}}{1966}{#1}}%%
   \ITEE{#3}{KKuratowski,BKnaster1927}{%%
      \BIB{#2}{K. Kuratowski and B. Knaster}%%
         {A connected and connected im kleinen point set which contains no perfect subset}%%
         {\jRN{BAMS}}{33}{1927}{106--109}{#1}}%%
   \ITEE{#3}{KKuratowski,AMostowski1976}{%%
      \BIb{#2}{K. Kuratowski and A. Mostowski}%%
         {Set Theory with an Introduction to Descriptive Set Theory}%%
         {\jRN{PWN}}{1976}{#1}}%%
   \ITEE{#3}{GLewicki1992}{%%
      \BIB{#2}{G. Lewicki}%%
         {Bernstein's ``lethargy'' theorem in metrizable topological linear spaces}%%
         {\jRN{MonatM}}{113}{1992}{213--226}{#1}}%%
   \ITEE{#3}{ASLewis1996}{%%
      \BIB{#2}{A.S. Lewis}%%
         {Group invariance and convex matrix analysis}%%
         {\jRN{SIAMJMAA}}{17}{1996}{927--949}{#1}}%%
   \ITEE{#3}{C-KLi,N-KTsing1991}{%%
      \BIB{#2}{C.-K. Li and N.-K. Tsing}%%
         {$G$-invariant norms and $G(c)$-radii}%%
         {\jRN{LAA}}{150}{1991}{179--194}{#1}}%%
   \ITEE{#3}{AJLazar,JLindenstrauss1971}{%%
      \BIB{#2}{A.J. Lazar and J. Lindenstrauss}%%
         {Banach spaces whose duals are $L_1$ spaces and their representing matrices}%%
         {\jRN{ActaM}}{126}{1971}{165--193}{#1}}%%
   \ITEE{#3}{ALindenbaum1926}{%%
      \BIB{#2}{A. Lindenbaum}%%
         {Contributions \`{a} l'\'{e}tude de l'espace m\'{e}trique I}%%
         {\jRN{FM}}{8}{1926}{209--222}{#1}}%%
   \ITEE{#3}{DLindenstrauss,LTzafriri1971}{%%
      \BIB{#2}{D. Lindenstrauss and L. Tzafriri}%%
         {On the complemented subspaces problem}%%
         {\jRN{IsraelJM}}{9}{1971}{263--269}{#1}}%%
   \ITEE{#3}{LHLoomis1945}{%%
      \BIB{#2}{L.H. Loomis}%%
         {Abstract congruence and the uniqueness of Haar measure}%%
         {\jRN{AnnM}}{46}{1945}{348--355}{#1}}%%
   \ITEE{#3}{LHLoomis1949}{%%
      \BIB{#2}{L.H. Loomis}%%
         {Haar measure in uniform structures}%%
         {\jRN{DukeMJ}}{16}{1949}{193--208}{#1}}%%
   \ITEE{#3}{ERLorch1939}{%%
      \BIB{#2}{E.R. Lorch}%%
         {Bicontinuous linear transformation in certain vector spaces}%%
         {\jRN{BAMS}}{45}{1939}{564--569}{#1}}%%
   \ITEE{#3}{ATLundell,SWeingram1969}{%%
      \BIb{#2}{A.T. Lundell and S. Weingram}%%
         {The topology of CW-complexes}%%
         {Litton Educ. Publ.}{1969}{#1}}%%
   \ITEE{#3}{WLusky1976}{%%
      \BIB{#2}{W. Lusky}%%
         {The Gurarij spaces are unique}%%
         {\jRN{ArchM}}{27}{1976}{627--635}{#1}}%%
   \ITEE{#3}{WLusky1977}{%%
      \BIB{#2}{W. Lusky}%%
         {On separable Lindenstrauss spaces}%%
         {\jRN{JFA}}{26}{1977}{103--120}{#1}}%%
   \ITEE{#3}{DMaharam1942}{%%
      \BIB{#2}{D. Maharam}%%
         {On homogeneous measure algebras}%%
         {\jRN{PNatlUSA}}{28}{1942}{108--111}{#1}}%%
   \ITEE{#3}{MMalicki,SSolecki2009}{%%
      \BIB{#2}{M. Malicki and S. Solecki}%%
         {Isometry groups of separable metric spaces}%%
         {\jRN{MProcCambPhS}}{146}{2009}{67--81}{#1}}%%
   \ITEE{#3}{PMankiewicz1972}{%%
      \BIB{#2}{P. Mankiewicz}%%
         {On extension of isometries in normed linear spaces}%%
         {\jRN{BAPolSSSM}}{20}{1972}{367--371}{#1}}%%
   \ITEE{#3}{JMartinezMaurica,MTPellon1987}{%%
      \BIB{#2}{J. Martinez-Maurica and M.T. Pell\'{o}n}%%
         {Non-archimedean Chebyshev centers}%%
         {\jRN{IndagMP}}{90}{1987}{417--421}{#1}}%%
   \ITEE{#3}{KMaurin1980}{%%
      \BIb{#2}{K. Maurin}%%
         {Analysis, Part II}%%
         {D. Reidel, Dordrecht-Boston-London}{1980}{#1}}%%
   \ITEE{#3}{SMazur,SUlam1932}{%%
      \BIB{#2}{S. Mazur and S. Ulam}%%
         {Sur les transformationes isom\'{e}triques d'espaces vectoriels norm\'{e}s}%%
         {\jRN{CRASParis}}{194}{1932}{946--948}{#1}}%%
   \ITEE{#3}{SMazurkiewicz1920}{%%
      \BIB{#2}{S. Mazurkiewicz}%%
         {Sur les lignes de Jordan}%%
         {\jRN{FM}}{1}{1920}{166--209}{#1}}%%
   \ITEE{#3}{SMazurkiewicz,WSierpinski1920}{%%
      \BIB{#2}{S. Mazurkiewicz and W. Sierpi\'{n}ski}%%
         {Contributions a la topologie des ensembles denombrables}%%
         {\jRN{FM}}{1}{1920}{17--27}{#1}}%%
   \ITEE{#3}{MMbekhta1992}{%%
      \BIB{#2}{M. Mbekhta}%%
         {Sur la structure des composantes connexes semi-Fredholm de $B(H)$}%%
         {\jRN{PAMS}}{116}{1992}{521--524}{#1}}%%
   \ITEE{#3}{JEMcCarthy1996}{%%
      \BIB{#2}{J.E. McCarthy}%%
         {Boundary values and Cowen-Douglas curvature}%%
         {\jRN{JFA}}{137}{1996}{1--18}{#1}}%%
   \ITEE{#3}{JMelleray2007}{%%
      \BIB{#2}{J. Melleray}%%
         {Computing the complexity of the relation of isometry between separable Banach spaces}%%
         {\jRN{MLQ}}{53}{2007}{128--131}{#1}}%%
   \ITEE{#3}{JMelleray2007a}{%%
      \BIB{#2}{J. Melleray}%%
         {On the geometry of Urysohn's universal metric space}%%
         {\jRN{TopA}}{154}{2007}{384--403}{#1}}%%
   \ITEE{#3}{JMelleray2008}{%%
      \BIB{#2}{J. Melleray}%%
         {Some geometric and dynamical properties of the Urysohn space}%%
         {\jRN{TopA}}{155}{2008}{1531--1560}{#1}}%%
   \ITEE{#3}{JMelleray,FVPetrov,AMVershik2008}{%%
      \BIB{#2}{J. Melleray, F.V. Petrov, A.M. Vershik}%%
         {Linearly rigid metric spaces and the embedding problem}%%
         {\jRN{FM}}{199}{2008}{177--194}{#1}}%%
   \ITEE{#3}{EMichael1953}{%%
      \BIB{#2}{E. Michael}%%
         {Some extension theorems for continuous functions}%%
         {\jRN{PacJM}}{3}{1953}{789--806}{#1}}%%
   \ITEE{#3}{EMichael1954}{%%
      \BIB{#2}{E. Michael}%%
         {Local properties of topological spaces}%%
         {\jRN{DukeMJ}}{21}{1954}{163--171}{#1}}%%
   \ITEE{#3}{EMichael1956}{%%
      \BIB{#2}{E. Michael}%%
         {Selected selection theorems}%%
         {\jRN{AmMMon}}{58}{1956}{233--238}{#1}}%%
   \ITEE{#3}{EMichael1956a}{%%
      \BIB{#2}{E. Michael}%%
         {Continuous selections. I}%%
         {\jRN{AnnM}}{63}{1956}{361--382}{#1}}%%
   \ITEE{#3}{EMichael1956b}{%%
      \BIB{#2}{E. Michael}%%
         {Continuous selections. II}%%
         {\jRN{AnnM}}{64}{1956}{562--580}{#1}}%%
   \ITEE{#3}{EMichael1959}{%%
      \BIB{#2}{E. Michael}%%
         {A theorem on semi-continuous set-valued functions}%%
         {\jRN{DukeMJ}}{26}{1959}{647--652}{#1}}%%
   \ITEE{#3}{JVanMill1986}{%%
      \BIB{#2}{J. van Mill}%%
         {Another counterexample in ANR theory}%%
         {\jRN{PAMS}}{97}{1986}{136--138}{#1}}%%
   \ITEE{#3}{JVanMill2001}{%%
      \BIb{#2}{J. van Mill}%%
         {The Infinite-Dimensional Topology of Function Spaces %%
         \textup{(North-Holland Mathematical Library, vol. 64)}}%%
         {Elsevier, Amsterdam}{2001}{#1}}%%
   \ITEE{#3}{WMlak1991}{%%
      \BIb{#2}{W. Mlak}%%
         {Hilbert Spaces and Operator Theory}%%
         {PWN --- Polish Scientific Publishers and Kluwer Academic Publishers, Warszawa-Dordrecht}{1991}{#1}}%%
   \ITEE{#3}{JMogilski1979}{%%
      \BIB{#2}{J. Mogilski}%%
         {$CE$-decomposition of $l_2$-manifolds}%%
         {\jRN{BAPolSSSM}}{27}{1979}{309--314}{#1}}%%
   \ITEE{#3}{RLMoore1916}{%%
      \BIB{#2}{R.L. Moore}%%
         {On the foundations of plane analysis situs}%%
         {\jRN{TAMS}}{17}{1916}{131--164}{#1}}%%
   \ITEE{#3}{KMorita1955}{%%
      \BIB{#2}{K. Morita}%%
         {A condition for the metrizability of topological spaces and for $n$-dimensionality}%%
         {\jRN{SciRepTokyoA}}{5}{1955}{33--36}{#1}}%%
   \ITEE{#3}{AMukherjea,NATserpes1976}{%%
      \BIb{#2}{A. Mukherjea and N.A. Tserpes}%%
         {Measures on topological semigroups}%%
         {Springer Lecture Notes in Math. Vol. 547, Berlin}{1976}{#1}}%%
   \ITEE{#3}{JMycielski1974}{%%
      \BIB{#2}{J. Mycielski}%%
         {Remarks on invariant measures in metric spaces}%%
         {\jRN{CollM}}{32}{1974}{105--112}{#1}}%%
   \ITEE{#3}{SNNaboko1984}{%%
      \BIB{#2}{S.N. Naboko}%%
         {Conditions for similarity to unitary and selfadjoint operators}%%
         {\jRN{FunkAnalPril}}{18}{1984}{16--27}{#1}}%%
   \ITEE{#3}{LNachbin1965}{%%
      \BIb{#2}{L. Nachbin}%%
         {The Haar Integral}%%
         {D. Van Nostrand Company, Inc., Princeton-New Jersey-Toronto-New York-London}{1965}{#1}}%%
   \ITEE{#3}{BSz-Nagy1947}{%%
      \BIB{#2}{B. Sz.-Nagy}%%
         {On uniformly bounded linear transformations in Hilbert space}%%
         {\jRN{ActaSM}}{11}{1947}{152--157}{#1}}%%
   \ITEE{#3}{TDNarang,SKGarg1991}{%%
      \BIB{#2}{T.D. Narang and S.K. Garg}%%
         {On the uniqueness of best approximation in non-archimedian spaces}%%
         {\jRN{PeriodMHung}}{22}{1991}{121--124}{#1}}%%
   \ITEE{#3}{JVonNeumann1934}{%%
      \BIB{#2}{J. von Neumann}%%
         {Zum Haarschen Mass in topologischen Gruppen}%%
         {\jRN{ComposM}}{1}{1934}{106--114}{#1}}%%
   \ITEE{#3}{JVonNeumann1937}{%%
      \BiB{#2}{J. von Neumann}%%
         {Some matrix-inequalities and metrization of matrix-space}{\jRN{TomskUnivRev}{} \textbf{1} (1937), 286--300; %%
         in }{Collected Works}{Pergamon, New York}{1962}{Vol. 4, 205--219}{#1}}%%
   \ITEE{#3}{pn2002}{\bibITEM{#2}{#1} \mypaplist{pn1}}%%
   \ITEE{#3}{pn2006a}{\bibITEM{#2}{#1} \mypaplist{pn2}}%%
   \ITEE{#3}{pn2006b}{\bibITEM{#2}{#1} \mypaplist{pn3}}%%
   \ITEE{#3}{pn2007}{\bibITEM{#2}{#1} \mypaplist{pn4}}%%
   \ITEE{#3}{pn2008a}{\bibITEM{#2}{#1} \mypaplist{pn5}}%%
   \ITEE{#3}{pn2008b}{\bibITEM{#2}{#1} \mypaplist{pn6}}%%
   \ITEE{#3}{pn2009a}{\bibITEM{#2}{#1} \mypaplist{pn7}}%%
   \ITEE{#3}{pn2009b}{\bibITEM{#2}{#1} \mypaplist{pn8}}%%
   \ITEE{#3}{pn2009c}{\bibITEM{#2}{#1} \mypaplist{pn9}}%%
   \ITEE{#3}{pn2010a}{\bibITEM{#2}{#1} \mypaplist{pn12}}%%
   \ITEE{#3}{pn2010b}{\bibITEM{#2}{#1} \mypaplist{pn13}}%%
   \ITEE{#3}{pn2011a}{\bibITEM{#2}{#1} \mypaplist{pn10}}%% JCA
   \ITEE{#3}{pn2011b}{\bibITEM{#2}{#1} \mypaplist{pn15}}%% U-topol
   \ITEE{#3}{pn2009x}{%% U-F2
      \bibITEM{#2}{#1} \mypaplist{pn11}}%% U-f
   \ITEE{#3}{pn2010x}{%%
      \bibITEM{#2}{#1} \mypaplist{pn14}}%% NTPVG
   \ITEE{#3}{pn2010y}{%%
      \bibITEM{#2}{#1} \mypaplist{pn15}}%% U-topolo
   \ITEE{#3}{pnXXXXa}{%% U-MEAS-LIP
      \bibITEM{#2}{#1} \mypaplist{pnX1}}%% S-W
   \ITEE{#3}{pnXXXXb}{%% U-F2, U-F, SMF
      \bibITEM{#2}{#1} \mypaplist{pnX2}}%% functor
   \ITEE{#3}{pnXXXXc}{%% SPECTRUM
      \bibITEM{#2}{#1} \mypaplist{pnX3}}%% orbits
   \ITEE{#3}{pnXXXXd}{%% uvAg
      \bibITEM{#2}{#1} \mypaplist{pnX13}}%% note_ANR
   \ITEE{#3}{MNiezgoda1998}{%%
      \BIB{#2}{M. Niezgoda}%%
         {Group majorization and Schur type inequalities}%%
         {\jRN{LAA}}{268}{1998}{9--30}{#1}}%%
   \ITEE{#3}{MNiezgoda1998a}{%%
      \BIB{#2}{M. Niezgoda}%%
         {An analytical characterization of effective and of irreducible groups inducing cone orderings}%%
         {\jRN{LAA}}{269}{1998}{105--114}{#1}}%%
   \ITEE{#3}{MNiezgoda,TYTam2001}{%%
      \BIB{#2}{M. Niezgoda and T.Y. Tam}%%
         {On norm property of $G(c)$-radii and Eaton triples}%%
         {\jRN{LAA}}{336}{2001}{119--130}{#1}}%%
   \ITEE{#3}{APazy1983}{%%
      \BIb{#2}{A. Pazy}{Semigroups of Linear Operators %%
         and Applications to Partial Differential Equations \textup{(Applied Mathematical Sciences, vol. 44)}}%%
         {Springer-Verlag, New York}{1983}{#1}}%%
   \ITEE{#3}{APelc1982}{%%
      \BIB{#2}{A. Pelc}%%
         {Semiregular invariant measures on abelian groups}%%
         {\jRN{PAMS}}{86}{1982}{423--426}{#1}}%%
   \ITEE{#3}{RPenrose1955}{%%
      \BIB{#2}{R. Penrose}%%
         {A generalized inverse for matrices}%%
         {\jRN{ProcCambPhS}}{51}{1955}{406--413}{#1}}%%
   \ITEE{#3}{VPestov2006}{%%
      \BIb{#2}{V. Pestov}%%
         {Dynamics of infinite-dimensional groups. The Ramsey-Dvoretzky-Milman phenomenon}%%
         {University Lecture Series \textbf{40}, AMS, Providence, RI}{2006}{#1}}%%
   \ITEE{#3}{VPestov2007}{%%
      \BiB{#2}{V. Pestov}%%
         {Forty-plus annotated questions about large topological groups}%%
         {in:}{Open Problems in Topology II}{Elliot Pearl (editor), Elsevier B.V., Amsterdam}{2007}{439--450}{#1}}%%
   \ITEE{#3}{PVPetersen1993}{%%
      \BiB{#2}{P.V. Petersen}%%
         {Gromov-Hausdorff convergence of metric spaces}{in book:}{Differential Geometry: Riemannian Geometry %%
         (Los Angeles, CA, 1990)}{Amer. Math. Soc., Providence, RI}{1993}{489--504}{#1}}%%
   \ITEE{#3}{DRamachandran,MMisiurewicz1982}{%%
      \BIB{#2}{D. Ramachandran and M. Misiurewicz}%%
         {Hopf's theorem on invariant measures for a group of transformations}%%
         {\jRN{SM}}{74}{1982}{183--189}{#1}}%%
   \ITEE{#3}{JMRosenblatt1974}{%%
      \BIB{#2}{J.M. Rosenblatt}%%
         {Equivalent invariant measures}%%
         {\jRN{IsraelJM}}{17}{1974}{261--270}{#1}}%%
   \ITEE{#3}{WRudin1962}{%%
      \BIb{#2}{W. Rudin}%%
         {Fourier Analysis on Groups \textup{(Interscience Tracts in Pure and Applied Mathematics, Number 12)}}%%
         {Interscience Publishers, New York}{1962}{#1}}%%
   \ITEE{#3}{WRudin1991}{%%
      \BIb{#2}{W. Rudin}%%
         {Functional Analysis}%%
         {McGraw-Hill Science}{1991}{#1}}%%
   \ITEE{#3}{KSakai,MYaguchi2003}{%%
      \BIB{#2}{K. Sakai and M. Yaguchi}%%
         {Characterizing manifolds modeled on certain dense subspaces of non-separable Hilbert spaces}%%
         {\jRN{TsukubaJM}}{27}{2003}{143--159}{#1}}%%
   \ITEE{#3}{RSchori1971}{%%
      \BIB{#2}{R. Schori}%%
         {Topological stability for infinite-dimensional manifolds}%%
         {\jRN{ComposM}}{23}{1971}{87--100}{#1}}%%
   \ITEE{#3}{ZSemadeni1971}{%%
      \BIb{#2}{Z. Semadeni}%%
         {Banach Spaces of Continuous Functions (Vol. I)}%%
         {\jRN{PWN}}{1971}{#1}}%%
   \ITEE{#3}{JPSerre1951}{%%
      \BIB{#2}{J.-P. Serre}%%
         {Homologie singuli\`{e}re des espaces fibr\'{e}s}%%
         {\jRN{AnnM}}{54}{1951}{425--505}{#1}}%%
   \ITEE{#3}{WSierpinski1928}{%%
      \BIB{#2}{W. Sierpi\'nski}%%
         {Sur les projections des ensembles compl\'{e}mentaires aux ensembles \textup{(A)}}%%
         {\jRN{FM}}{11}{1928}{117--122}{#1}}%%
   \ITEE{#3}{RCSteinlage1975}{%%
      \BIB{#2}{R.C. Steinlage}%%
         {On Haar measure in locally compact $T_2$ spaces}%%
         {\jRN{AmJM}}{97}{1975}{291--307}{#1}}%%
   \ITEE{#3}{JStochel,FHSzafraniec1989}{%%
      \BIB{#2}{J. Stochel and F.H. Szafraniec}%%
         {On normal extensions of unbounded operators. III. Spectral properties}%%
         {\jRN{PublRIMSKyoto}}{25}{1989}{105--139}{#1}}%%
   \ITEE{#3}{AHStone1962}{%%
      \BIB{#2}{A.H. Stone}%%
         {Absolute $\FFf_{\sigma}$-spaces}%%
         {\jRN{PAMS}}{13}{1962}{495--499}{#1}}%%
   \ITEE{#3}{AHStone1962a}{%%
      \BIB{#2}{A.H. Stone}%%
         {Non-separable Borel sets}%%
         {\jRN{DissM}}{28}{1962}{41 pages}{#1}}%%
   \ITEE{#3}{AHStone1972}{%%
      \BIB{#2}{A.H. Stone}%%
         {Non-separable Borel sets II}%%
         {\jRN{GTopA}}{2}{1972}{249--270}{#1}}%%
   \ITEE{#3}{MHStone1937}{%%
      \BIB{#2}{M.H. Stone}%%
         {Application of the theory of Boolean rings to general topology}%%
         {\jRN{TAMS}}{41}{1937}{375--481}{#1}}%%
   \ITEE{#3}{MHStone1948}{%%
      \BIB{#2}{M.H. Stone}%%
         {The generalized Weierstrass approximation theorem}%%
         {\jRN{MMag}}{21}{1948}{167--184}{#1}}%%
   \ITEE{#3}{WTakahashi1970}{%%
      \BIB{#2}{W. Takahashi}%%
         {A convexity in metric space and nonexpansive mappings, I}%%
         {\jRN{KodaiMSemRep}}{22}{1970}{142--149}{#1}}%%
   \ITEE{#3}{MTakesaki2002}{%%
      \BIb{#2}{M. Takesaki}%%
         {Theory of Operator Algebras I \textup{(Encyclopaedia of Mathematical Sciences, Volume 124)}}%%
         {Springer-Verlag, Berlin-Heidelberg-New York}{2002}{#1}}%%
   \ITEE{#3}{MTakesaki2003}{%%
      \BIb{#2}{M. Takesaki}%%
         {Theory of Operator Algebras II \textup{(Encyclopaedia of Mathematical Sciences, Volume 125)}}%%
         {Springer-Verlag, Berlin-Heidelberg-New York}{2003}{#1}}%%
   \ITEE{#3}{MTakesaki2003a}{%%
      \BIb{#2}{M. Takesaki}%%
         {Theory of Operator Algebras III \textup{(Encyclopaedia of Mathematical Sciences, Volume 127)}}%%
         {Springer-Verlag, Berlin-Heidelberg-New York}{2003}{#1}}%%
   \ITEE{#3}{TYTam1999}{%%
      \BIB{#2}{T.Y. Tam}%%
         {An extension of a result of Lewis}%%
         {\jRN{ELA}}{5}{1999}{1--10}{#1}}%%
   \ITEE{#3}{TYTam2000}{%%
      \BIB{#2}{T.Y. Tam}%%
         {Group majorization, Eaton triples and numerical range}%%
         {\jRN{LMLA}}{47}{2000}{11--28}{#1}}%%
   \ITEE{#3}{TYTam2002}{%%
      \BIB{#2}{T.Y. Tam}%%
         {Generalized Schur-concave functions and Eaton triples}%%
         {\jRN{LMLA}}{50}{2002}{113--120}{#1}}%%
   \ITEE{#3}{TYTam,WCHill2001}{%%
      \BIB{#2}{T.Y. Tam and W.C. Hill}%%
         {On $G$-invariant norms}%%
         {\jRN{LAA}}{331}{2001}{101--112}{#1}}%%
   \ITEE{#3}{AFTiman,IAVestfrid1983}{%%
      \BIB{#2}{A.F. Timan and I.A. Vestfrid}%%
         {Any separable ultrametric space can be isometrically imbedded in $l_2$}%%
         {\jRN{FAA}}{17}{1983}{70--71}{#1}}%%
   \ITEE{#3}{HTorunczyk1970}{%%
      \BIB{#2}{H. Toru\'{n}czyk}%%
         {Remarks on Anderson's paper ``On topological infinite deficiency''}%%
         {\jRN{FM}}{66}{1970}{393--401}{#1}}%%
   \ITEE{#3}{HTorunczyk1970a}{%%
      \BIb{#2}{H. Toru\'{n}czyk}%%
         {$G$-$K$-absorbing and skeletonized sets in metric spaces}%%
         {Ph.D. thesis, Inst. Math. Polish Acad. Sci., Warszawa}{1970}{#1}}%%
   \ITEE{#3}{HTorunczyk1972}{%%
      \BIB{#2}{H. Toru\'{n}czyk}%%
         {A short proof of Hausdorff's theorem on extending metrics}%%
         {\jRN{FM}}{77}{1972}{191--193}{#1}}%%
   \ITEE{#3}{HTorunczyk1974}{%%
      \BIB{#2}{H. Toru\'{n}czyk}%%
         {Absolute retracts as factors of normed linear spaces}%%
         {\jRN{FM}}{86}{1974}{53--67}{#1}}%%
   \ITEE{#3}{HTorunczyk1975}{%%
      \BIB{#2}{H. Toru\'{n}czyk}%%
         {On Cartesian factors and the topological classification of linear metric spaces}%%
         {\jRN{FM}}{88}{1975}{71--86}{#1}}%%
   \ITEE{#3}{HTorunczyk1978}{%%
      \BIB{#2}{H. Toru\'{n}czyk}%%
         {Concerning locally homotopy negligible sets and characterization of $l_2$-manifolds}%%
         {\jRN{FM}}{101}{1978}{93--110}{#1}}%%
   \ITEE{#3}{HTorunczyk1980}{%%
      \BiB{#2}{H. Toru\'{n}czyk}{Characterization of infinite-dimensional manifolds}{in:}%%
         {Proceedings of the International Conference on Geometric Topology (Warsaw, 1978)}%%
         {\jRN{PWN}}{1980}{431--437}{#1}}%%
   \ITEE{#3}{HTorunczyk1981}{%%
      \BIB{#2}{H. Toru\'{n}czyk}%%
         {Characterizing Hilbert space topology}%%
         {\jRN{FM}}{111}{1981}{247--262}{#1}}%%
   \ITEE{#3}{HTorunczyk1985}{%%
      \BIB{#2}{H. Toru\'{n}czyk}%%
         {A correction of two papers concerning Hilbert manifolds}%%
         {\jRN{FM}}{125}{1985}{89--93}{#1}}%%
   \ITEE{#3}{KTsuda1985}{%%
      \BIB{#2}{K. Tsuda}%%
         {A note on closed embeddings of finite dimensional metric spaces}%%
         {\jRN{BLondMS}}{17}{1985}{273--278}{#1}}%%
   \ITEE{#3}{PSUrysohn1925}{%%
      \BIB{#2}{P.S. Urysohn}%%
         {Sur un espace m\'{e}trique universel}%%
         {\jRN{CRASParis}}{180}{1925}{803--806}{#1}}%%
   \ITEE{#3}{PSUrysohn1927}{%%
      \BIB{#2}{P.S. Urysohn}%%
         {Sur un espace m\'{e}trique universel}%%
         {\jRN{BullSM}}{51}{1927}{43--64, 74--96}{#1}}%%
   \ITEE{#3}{VVUspenskij1986}{%%
      \BIB{#2}{V.V. Uspenskij}%%
         {A universal topological group with a countable basis}%%
         {\jRN{FAA}}{20}{1986}{86--87}{#1}}%%
   \ITEE{#3}{VVUspenskij1990}{%%
      \BIB{#2}{V.V. Uspenskij}%%
         {On the group of isometries of the Urysohn universal metric space}%%
         {\jRN{CMUC}}{31}{1990}{181--182}{#1}}%%
   \ITEE{#3}{VVUspenskij2004}{%%
      \BIB{#2}{V.V. Uspenskij}%%
         {The Urysohn universal metric space is homeomorphic to a Hilbert space}%%
         {\jRN{TopA}}{139}{2004}{145--149}{#1}}%%
   \ITEE{#3}{VVUspenskij2008}{%%
      \BIB{#2}{V.V. Uspenskij}%%
         {On subgroups of minimal topological groups}%%
         {\jRN{TopA}}{155}{2008}{1580--1606}{#1}}%%
   \ITEE{#3}{VSVaradarajan1963}{%%
      \BIB{#2}{V.S. Varadarajan}%%
         {Groups of automorphisms of Borel spaces}%%
         {\jRN{TAMS}}{109}{1963}{191--220}{#1}}%%
   \ITEE{#3}{AMVershik1998}{%%
      \BIB{#2}{A.M. Vershik}%%
         {The universal Urysohn space, Gromov's metric triples, and random metrics on the series of natural numbers}%%
         {\jRN{UspekhiMN}}{53}{1998}{57--64}{#1} English translation: \jRN{RussMS}{} \textbf{53} (1998), 921--928. %%
         Correction: \jRN{UspekhiMN}{} \textbf{56} (2001), p. 207. English translation: \jRN{RussMS}{} \textbf{56} %%
         (2001), p. 1015.}%%
   \ITEE{#3}{AMVershik2002}{%%
      \BIb{#2}{A.M. Vershik}%%
         {Random metric spaces and the universal Urysohn space}%%
         {Fundamental Mathematics Today. 10th anniversary of the Independent Moscow University. MCCME Publ.}{2002}{#1}}%%
   \ITEE{#3}{NWeaver1999}{%%
      \BIb{#2}{N. Weaver}%%
         {Lipschitz Algebras}%%
         {World Scientific}{1999}{#1}}%%
   \ITEE{#3}{JWeidmann1980}{%%
      \BIb{#2}{J. Weidmann}%%
         {Linear Operators in Hilbert Spaces}%%
         {(Graduate Texts in Mathematics, vol. 68) Springer-Verlag New York Inc.}{1980}{#1}}%%
   \ITEE{#3}{JEWest1969}{%%
      \BIB{#2}{J.E. West}%%
         {Approximating homotopies by isotopies in Fr\'{e}chet manifolds}%%
         {\jRN{BAMS}}{75}{1969}{1254--1257}{#1}}%%
   \ITEE{#3}{JEWest1969a}{%%
      \BIB{#2}{J.E. West}%%
         {Fixed-point sets of transformation groups on infinite-product spaces}%%
         {\jRN{PAMS}}{21}{1969}{575--582}{#1}}%%
   \ITEE{#3}{JEWest1970}{%%
      \BIB{#2}{J.E. West}%%
         {The ambient homeomorphy of infinite-dimensional Hilbert spaces}%%
         {\jRN{PacJM}}{34}{1970}{257--267}{#1}}%%
   \ITEE{#3}{JHCWhitehead1949}{%%
      \BIB{#2}{J.H.C. Whitehead}%%
         {Combinatorial homotopy I}%%
         {\jRN{BAMS}}{55}{1949}{213--245}{#1}}%%
   \ITEE{#3}{GTWhyburn1942}{%%
      \BIb{#2}{G. T. Whyburn}%%
         {Analytic Topology}%%
         {Amer. Math. Soc. Colloquium Publications (vol. XXVIII), New York}{1942}{#1}}%%
   \ITEE{#3}{RYTWong1967}{%%
      \BIB{#2}{R.Y.T. Wong}%%
         {On homeomorphisms of certain infinite dimensional spaces}%%
         {\jRN{TAMS}}{128}{1967}{148--154}{#1}}%%
   \ITEE{#3}{LYang,JZhang1987}{%%
      \BIB{#2}{L. Yang and J. Zhang}%%
         {Average distance constants of some compact convex space}%%
         {\jRN{JChinUST}}{17}{1987}{17--23}{#1}}%%
   \ITEE{#3}{PZakrzewski1993}{%%
      \BIB{#2}{P. Zakrzewski}%%
         {The existence of invariant $\sigma$-finite measures for a group of transformations}%%
         {\jRN{IsraelJM}}{83}{1993}{275--287}{#1}}%%
   \ITEE{#3}{PZakrzewski2002}{%%
      \BIb{#2}{P. Zakrzewski}%%
         {Measures on Algebraic-Topological Structures, Handbook of Measure Thoery}%%
         {E. Pap, ed., Elsevier, Amsterdam}{2002, 1091--1130}{#1}}%%
   \ITEE{#3}{KZhu2000}{%%
      \BIB{#2}{K. Zhu}%%
         {Operators in Cowen-Douglas classes}%%
         {\jRN{IllinoisJM}}{44}{2000}{767--783}{#1}}%%
   }

\newcommand{\mypaplist}[2][]{%%
   \ITEE{#2}{pn1}{%%
      \myBIB{Separate and joint similarity to families of normal operators}%%
         {\jRN[#1]{SM}}{149}{2002}{39--62}}%%
   \ITEE{#2}{pn2}{%%
      \myBIB{Locally arcwise connected metrizable spaces with the fixed point property are complete-metrizable}%%
         {\jRN[#1]{TopA}}{153}{2006}{1639--1642}}%%
   \ITEE{#2}{pn3}{%%
      \myBIB{Invariant measures for equicontinuous semigroups of continuous transformations %%
         of a compact Hausdorff space}{\jRN[#1]{TopA}}{153}{2006}{3373--3382}}%%
   \ITEE{#2}{pn4}{%%
      \myBIB{Approximation of the Hausdorff distance by the distance of continuous surjections}%%
         {\jRN[#1]{TopA}}{154}{2007}{655--664}}%%
   \ITEE{#2}{pn5}{%%
      \myBIB{Generalized Haar integral}%%
         {\jRN[#1]{TopA}}{155}{2008}{1323--1328}}%%
   \ITEE{#2}{pn6}{%%
      \myBIB{Integration and Lipschitz functions}%%
         {\jRN[#1]{RCMP}}{57}{2008}{391--399}}%%
   \ITEE{#2}{pn7}{%%
      \myBIB{Canonical Banach function spaces generated by Urysohn universal spaces. Measures as Lipschitz maps}%%
         {\jRN[#1]{SM}}{192}{2009}{97--110}}%%
   \ITEE{#2}{pn8}{%%
      \myBIB{Urysohn universal spaces as metric groups of exponent $2$}%%
         {\jRN[#1]{FM}}{204}{2009}{1--6}}%%
   \ITEE{#2}{pn9}{%%
      \myBIB{Central subsets of Urysohn universal spaces}%%
         {\jRN[#1]{CMUC}}{50}{2009}{445--461}}%%
   \ITEE{#2}{pn10}{%%
      \myBIB[P. Niemiec and T.Y. Tam]{A representation of $G$-in\-variant norms for Eaton triple}%%
         {\jRN[#1]{JCA}}{18}{2011}{59--65}}%%
   \ITEE{#2}{pn11}{%% U-F2
      \myBIB{Functor of extension of contractions on Urysohn universal spaces}%%
         {\jRN[#1]{ACS}}{}{2009}{\texttt{DOI: 10.1007/s10485-009-9218-z}}}%%
   \ITEE{#2}{pn12}{%%
      \myBIB{Ultra-$\mM$-separability}%%
         {\jRN[#1]{TopA}}{157}{2010}{669--673}}%%
   \ITEE{#2}{pn13}{%%
      \myBIB{Functor of extension of $\Lambda$-isometric maps between central subsets %%
         of the unbounded Urysohn universal space}{\jRN[#1]{CMUC}}{51}{2010}{541--549}}%%
   \ITEE{#2}{pn14}{%%
      \myBIB{Normed topological pseudovector groups}{\jRN[#1]{ACS}}{}{2010}%%
         {\ITE{\equal{#1}{}}{\texttt{DOI: 10.1007/s10485\-010-9239-7}}{\texttt{DOI: 10.1007/s10485-010-9239-7}}}}%%
   \ITEE{#2}{pn15}{%%
      \myBIB{Topological structure of Urysohn universal spaces}%%
         {\jRN[#1]{TopA}}{158}{2011}{352--359}}%%
   \ITEE{#2}{pnX1}{%% U-MEAS-LIP
      \myBAPP{Strengthened Stone-Weierstrass type theorem}%%
         {accepted for publication in \jRN[#1]{OpusM}}}%%
   \ITEE{#2}{pnX2}{%% U-F2, U-F, SMF
      \myBAPP{Functor of continuation in Hilbert cube and Hilbert space}%%
         {to appear in \jRN[#1]{FM}}}%%
   \ITEE{#2}{pnX3}{%% SPECTRUM
      \myBAPP{Norm closures of orbits of bounded operators}%%
         {to appear.}}%%
   \ITEE{#2}{pnX4}{%%
      \myBAPP{A note on invariant measures}%%
         {accepted for publication in \jRN[#1]{OpusM}}}%%
   \ITEE{#2}{pnX6}{%%
      \myBAPP{Extending maps by injective $\sigma$-$Z$-maps in Hilbert manifolds}%%
         {to appear in \jRN[#1]{BullPol}}}%%
   \ITEE{#2}{pnX7}{%%
      \myBAPP{Spaces of measurable functions}%%
         {submitted to \jRN[#1]{CollectM}}}%%
   \ITEE{#2}{pnX8}{%%
      \myBAPP{Normal systems over ANR's, rigid embeddings and nonseparable absorbing sets}%%
         {submitted to \jRN[#1]{MichMJ}}}%%
   \ITEE{#2}{pnX9}{%%
      \myBAPP{Borel structure of the spectrum of a closed operator}%%
         {submitted to \jRN[#1]{JFA}}}%%
   \ITEE{#2}{pnX10}{%%
      \myBAPP{Central points and measures and dense subsets of compact metric spaces}%%
         {submitted to \jRN[#1]{NonlinA}}}%%
   \ITEE{#2}{pnX11}{%%
      \myBAPP{Generalized absolute values and polar decompositions of a bounded operator}%%
         {submitted to \jRN[#1]{IEOT}.}}%%
   \ITEE{#2}{pnX12}{%%
      \myBAPP{Ultrametrics, extending of Lipschitz maps and nonexpansive selections}%%
         {submitted to \jRN[#1]{MProcCambPhS}}}%%
   \ITEE{#2}{pnX13}{%%
      \myBAPP{A note on ANR's}%%
         {submitted to \jRN[#1]{TopA}}}%%
   }

\newcommand{\rank}{\operatorname{rank}}\newcommand{\grp}[1]{\langle#1\rangle}

\begin{document}

\title{Universal valued Abelian groups}
\myData
\begin{abstract}
The counterparts of the Urysohn universal space in category of metric spaces and the Gurari\v{\i} space in category
of Banach spaces are constructed for separable valued Abelian groups of fixed (finite) exponents (and for valued groups
of similar type) and their uniqueness is established. Geometry of these groups, denoted by $\GGG_r(N)$, is investigated
and it is shown that each of $\GGG_r(N)$'s is homeomorphic to the Hilbert space $l^2$. Those of $\GGG_r(N)$'s which are
Urysohn as metric spaces are recognized. `Linear-like' structures on $\GGG_r(N)$ are studied and it is proved that every
separable metrizable topological vector space may be enlarged to $\GGG_r(0)$ with a `linear-like' structure which extends
the linear structure of the given space.
\end{abstract}
\maketitle

%%\thisPaper{pn??}
%% UZUPELNIC REFERENCJE DO SOLECKIEGO

\SECT{Introduction}

In several branches of mathematics, such as metric space theory or functional analysis, there are known examples
of the so-called \textit{universal} spaces. Probably the most famous example in this subject is the Banach space
$\CCc([0,1])$ of all real-valued continuous functions on $[0,1]$ which is universal for separable normed vector spaces
(over $\RRR$) as well as for separable metric spaces. Here universality means that every separable normed vector space
(respectively separable metric space) is isometrically-linearly isomorphic (repsectively isometric) to a linear subspace
(to a subset) of $\CCc([0,1])$. This is known as the Banach-Mazur theorem. Undoubtedly, the `concreteness' of the universal
space in this theorem merits its fame. However, it is no so easy to characterize $\CCc([0,1])$ among Banach (or metric)
spaces. There is also a general idea, in the spirit of Fra\"{\i}ss\'{e} limits, of constructing universal spaces which
are sumiltaneously uniquely determined by certain conditions related to the so-called \textit{universal disposition
property} (see below; with this terminology we follow e.g. Wies\l{}aw Kubi\'{s}). Surely, the first example of a universal
space of this kind was given by Urysohn \cite{ur1,ur2}. The Urysohn universal metric space, usually denoted by $\UUU$, is
a unique (up to isometry) separable complete metric space which has universal disposition property for finite metric
spaces. That is:
\begin{quote}
\textit{Every isometric map of a subset of a finite metric space $X$ into $\UUU$ is extendable to an isometric map
of the whole space $X$ into $\UUU$.}
\end{quote}
There is also a bounded version of $\UUU$. It may be defined in the following way. A metric space $X$ is said to be
\textit{Urysohn} iff
\begin{enumerate}[(U1)]\addtocounter{enumi}{-1}
\item $X$ is separable and complete,
\item (universality) every separable metric space of diameter no greater than the diameter of $X$ is isometrically
   embeddable into $X$,
\item ($\omega$-homogeneity) every isometry between two arbitrary finite subsets of $X$ is extendable to an isometry
   of $X$ onto $X$.
\end{enumerate}
A fundamental result on Urysohn spaces states that for every $r \in [0,\infty]$ there is a unique (up to isometry) Urysohn
metric space (denoted by $\UUU_r$) of diameter $r$. It is also easily shown that $\UUU_r$ is uniquely determined (among
separable complete metric spaces of diameter no greater than $r$) by universal disposition property for finite metric
spaces of diameter no greater than $r$.\par
If we pass from the category of metric spaces to Banach spaces, universal disposition property may be defined as follows
(cf. \cite{gur}). A Banach space $E$ is said to have \textit{universal disposition property (for finite dimensional Banach
spaces)} iff every isometric linear map of a linear subspace of a finite dimensional Banach space $F$ into $E$ is
extendable to an isometric linear map of $F$ into $E$. Gurari\v{\i} \cite{gur} has proved that there is no separable Banach
space with universal disposition property. In the same paper he has constructed a separable Banach space $\GGG$, which is
nowadays called the \textit{Gurari\v{\i}} space, with \textit{almost universal disposition property} defined
in the following way (the same space was also built by Lazar and Lindenstrauss \cite{l-l} in a different context):
\begin{quote}
\textit{Every isometric linear map $\psi\dd E \to \GGG$ of a linear subspace $E$ of a finite dimensional Banach space $F$
admits a linear extension $\widehat{\psi}\dd F \to \GGG$ such that $(1-\epsi) \|x\| \leqsl \|\widehat{\psi}(x)\| \leqsl
(1+\epsi) \|x\|$ for all $x \in F$ where $\epsi$ is an arbitrarily given real number from $(0,1)$.}
\end{quote}
Gurari\v{\i} has also shown that $\GGG$ is universal for separable Banach spaces (i.e. that every separable Banach space
admits an isometric linear embedding into $\GGG$). It was Lusky \cite{lus} who first proved the uniqueness of $\GGG$
up to isometric linear isomorphism (other proof of the uniqueness is contained in \cite{lu2}; the proof which involves
the back-and-forth technique was given by Solecki \cite{sol}). Thus, the Gurari\v{\i} space is a unique separable Banach
space with almost universal disposition property.\par
There is a striking resemblance between (almost) universal disposition properties of $\UUU$ and $\GGG$. In a sense, both
these spaces correspond to each other in categories of metric spaces and Banach spaces. The aim of this paper is to prove
the existence (together with uniqueness) of a universal (in the sense of embedding by isometric group homomorphisms) group
with universal disposition property (for finite groups) in the class (and some of their subclasses) of all separable valued
Abelian groups of class $\OOo_0$, defined by the following condition:
\begin{quote}
A valued Abelian group $(G,+,p)$ is said to be \textit{of class $\OOo_0$} iff $\lim_{n\to\infty} \frac{p(na)}{n} = 0$
for every $a \in G$.
\end{quote}
(In particular, if $p$ is a value on an Abelian group $(G,+)$, then each of the valued groups $(G,+,\min(p,1))$,
$(G,+,\frac{p}{p+1})$ and $(G,+,p^{\alpha})$ for $0 < \alpha < 1$ is of class $\OOo_0$.) What we exactly mean is explained
below.\par
Denote by $\Gg$ the class of all separable valued Abelian groups. Let $\Gg_{\infty}(0)$ and $\Gg_1(0)$ stand
for the classes of all groups $(G,+,p) \in \Gg$ of class $\OOo_0$ and, respectively, for which $p \leqsl 1$. Note that
$\Gg_1(0) \subset \Gg_{\infty}(0)$. Additionally, for natural $N > 1$ let $\Gg_{\infty}(N)\ (\subset \Gg_{\infty}(0))$
consist of all groups $(G,+,p) \in \Gg$ of exponent $N$. Finally, put $\Gg_1(N) = \Gg_1(0) \cap \Gg_{\infty}(N)$.\par
We say a function $\omega\dd [0,\infty) \to [0,\infty)$ is a \textit{modulus of continuity} iff
\begin{enumerate}[\upshape($\omega$1)]
\item $\omega$ is monotone increasing, that is, $\omega(x) \leqsl \omega(y)$ provided $0 \leqsl x \leqsl y$,
\item $\omega(x + y) \leqsl \omega(x) + \omega(y)$ for any $x, y \geqsl 0$,
\item $\lim_{t\to0^+} \omega(t) = \omega(0) = 0$.
\end{enumerate}
(Observe that we allow the zero function to be a modulus of continuity.)\par
The main two results of the paper are:

\begin{thm}{main1}
Let $r \in \{1,\infty\}$ and $N \in \{0,2,3,4,\ldots\}$. There is a unique (up to isometric group isomorphism) valued
Abelian group, denoted by $\GGG_r(N)$, with the following three properties:
\begin{enumerate}[\upshape(G1)]
\item $\GGG_r(N)$ is complete and $\GGG_r(N) \in \Gg_r(N)$,
\item whenever $(H,+,q)$ is a finite Abelian group (of exponent $N$, provided $N \neq 0$) with $q \leqsl r$, $K$ is
   a subgroup of $H$ and $\varphi\dd K \to \GGG_r(N)$ is an isometric group homomorphism, then for every $\epsi \in (0,1)$
   there is a group homomorphism $\varphi_{\epsi}\dd H \to \GGG_r(N)$ such that
   $$
   \max_{x \in K} p(\varphi(x) - \varphi_{\epsi}(x)) \leqsl \epsi
   $$
   and
   \begin{equation}\label{eqn:alm-iso}
   (1-\epsi) q(y) \leqsl p(\varphi_{\epsi}(y)) \leqsl (1+\epsi) q(y)
   \end{equation}
   for $y \in H$, where $p$ is the value of $\GGG_r(N)$,
\item if $N = 0$, finite rank elements of $\GGG_r(N)$ form a dense subset of $\GGG_r(N)$.
\end{enumerate}
\end{thm}

\begin{thm}{main2}
Let $r \in \{1,\infty\}$ and $N \in \{0,2,3,4,\ldots\}$ and let $p$ stand for the value of $\GGG_r(N)$.
\begin{enumerate}[\upshape(A)]
\item Let $(H,+,q) \in \Gg_r(N)$, $K$ be a closed subgroup of $H$ and $\varphi\dd K \to \GGG_r(N)$ be a continuous group
   homomorphism whose range has compact closure in $\GGG_r(N)$.
   \begin{enumerate}[\upshape(i)]
   \item There are moduli of continuity $\omega \not\equiv 0$, $\varrho \not\equiv 0$ and $\tau$ such that for each
      $x \in K$,
      \begin{equation}\label{eqn:om}
      p(\varphi(x)) \leqsl (\omega \circ q)(x),
      \end{equation}
      \begin{equation}\label{eqn:tau-rho}
      \tau(\dist_q(x,\ker \varphi)) \leqsl (\varrho \circ p)(\varphi(x))
      \end{equation}
      (where $\dist_q(x,\ker \varphi) = \inf\{q(x-y)\dd\ y \in \ker \varphi\}$) and
      \begin{equation*}\tag{$\omega$-$\varrho$-$\tau$}\begin{cases}
      \tau(t) + \varrho(s) \leqsl \varrho(\omega(t) + s) & \textup{for any } t, s \geqsl 0,\\
      \tau(1) \leqsl \varrho(1) & \textup{provided } r = 1.
      \end{cases}\end{equation*}
      If $\varphi$ is open as a map of $K$ onto $\varphi(K)$, the above $\tau$ may be chosen nonzero.
   \item For every modulus of continuity $\omega \not\equiv 0$ such that \eqref{eqn:om} is fulfilled for all $x \in K$
      there is a continuous group homomorphism $\varphi_{\omega}\dd H \to \GGG_r(N)$ extending $\varphi$ and satisfying
      \eqref{eqn:om} for each $x \in H$ with $\varphi$ replaced by $\varphi_{\omega}$, and such that $\ker \varphi_{\omega}
      = \ker \varphi$ and whenever \eqref{eqn:tau-rho} and \textup{($\omega$-$\varrho$-$\tau$)} are fulfilled for $\tau$
      and $\varrho \not\equiv 0$ (and $x \in K$), then \eqref{eqn:tau-rho} is satisfied for any $x \in H$ with $\varphi$
      replaced by $\varphi_{\omega}$.
   \end{enumerate}
   In particular, if $\varphi$ is isometric, it admits an extension being an isometric group homomorphism.
\item Every (topological) isomorphism between two compact subgroups of $\GGG_r(N)$ is extendable to an automorphism
   of the topological group $\GGG_r(N)$.
   What is more, if $K$ is a compact subgroup of $\GGG_r(N)$, $\varphi\dd K \to \GGG_r(N)$ is a group homomorphism,
   $\omega \not\equiv 0$ and $\tau \not\equiv 0$ are moduli of continuity such that \textup{($\omega$-$\tau$-$\id$)} and
   \textup{($\tau$-$\omega$-$\id$)} are fulfilled (where $\id$ denotes the identity map on $[0,\infty)$) and for each
   $x \in K$,
   \begin{equation}\label{eqn:om-tau}
   p(\varphi(x)) \leqsl (\omega \circ p)(x) \qquad \textup{and} \qquad p(x) \leqsl (\tau \circ p)(\varphi(x)),
   \end{equation}
   then there is a group automorphism $\psi\dd \GGG_r(N) \to \GGG_r(N)$ extending $\varphi$ such that \eqref{eqn:om-tau}
   is satisfied for each $x \in \GGG_r(N)$ with $\varphi$ replaced by $\psi$.\par
   In particular, if $\varphi$ is isometric, it admits an extension being an isometric automorphism.
\end{enumerate}
\end{thm}

Note that the point (G2) (in the statement of \THM{main1}) is a counterpart of the condition, proposed by Solecki
\cite{sol}, which characterizes the above mentioned Gurari\v{\i} space up to isometric linear isomorphism. Observe also
that (G3) is quite natural, since (G2) refers only to finite rank elements of the group.\par
\THM{main2} implies that every member of $\Gg_r(N)$ admits an embedding to $\GGG_r(N)$ by an isometric group homomorphism
and that in condition (G2) one may substitute $\epsi = 0$, that is, $\GGG_r(N)$ has universal disposition property
for finite groups of class $\Gg_r(N)$.\par
We shall also show that the group $\GGG_r(N)$ as a metric space is universal for separable metric spaces of diameter
no greater than $r$, that is, that every such space is isometrically embeddable into the metric space $\GGG_r(N)$. It turns
out that $\GGG_r(N)$ (as a metric space) is Urysohn iff $N \in \{0,2\}$. In particular, the groups $\GGG_r(2)$ are Boolean
Urysohn groups introduced by us in \cite{pn1}. What is more, for different pairs $(N,r)$ and $(M,s)$ the metric spaces
$\GGG_r(N)$ and $\GGG_s(M)$ are isometric iff $r=s$ and $\{N,M\} = \{0,2\}$. Also all the groups $\GGG_r(N)$ are pairwise
nonisomorphic as topological groups.\par
In Section 8 we prove that each of the topological spaces $\GGG_r(N)$ is homeomorphic to the Hilbert space $l^2$.
To establish this result we develop our earlier study of the so-called \textit{topological pseudovector groups},
introduced in \cite{pn2}. Namely, a \textit{pseudovector Abelian group} is a triple $(G,+,*)$ such that $(G,+)$ is
an Abelian group, `$*$' is an action of $[0,\infty)$ on $G$ satisfying the following axioms: $0 * x = 0_G$, $1 * x = x$,
$(st) * x = s * (t * x)$ for all $x \in G$ and $s, t \geqsl 0$, and for every $t \geqsl 0$ the function $G \ni x \mapsto
t * x \in G$ is a group homomorphism. If, in addition, $G$ is a topological group, it is said to be a \textit{topological
pseudovector group} provided the action `$*$', as a function of $[0,\infty) \times G$ into $G$, is continuous. A value
$p$ on the pseudovector Abelian group $G$ is called a \textit{norm} if $p(t * x) = t p(x)$ for any $t \geqsl 0$ and $x \in
G$, and it is called a \textit{subnorm} iff $p(x) \leqsl p(s * x) \leqsl s p(x)$ for each $s \geqsl 1$ and $x \in G$.
If a (sub)norm $p$ induces the topology on $G$ with respect to which the action `$*$' is continuous, then the quadruple
$(G,+,*,p)$ is called a \textit{(sub)normed} topological pseudovector group. In Section~7 we show that every metrizable
topological pseudovector group admits a subnorm inducing its topology. The main result on pseudovector structures
of $\GGG_r(N)$ is the following

\begin{thm}{main3}
Let $r \in \{1,\infty\}$ and $N \in \{0,2,3,4,\ldots\}$. Every nontrivial (sub)normed topological pseudovector group
$(G,+,*,\|\cdot\|_G)$ such that $(G,+,\|\cdot\|_G) \in \Gg_r(N)$ may be enlarged to a (sub)normed pseudovector topological
group $(\tilde{G},+,*,\|\cdot\|)$ such that the valued groups $\tilde{G}$ and $\GGG_r(N)$ are isometrically group
isomorphic.
\end{thm}

In the above result one may erase the word `nontrivial' provided that $r = \infty$ or the final value $\|\cdot\|$ has to be
a subnorm rather than a norm. As a main application of the above result we obtain the above mentioned theorem which
says that $\GGG_r(N)$ is homeomorphic to $l^2$. (This result may be immediately deduced for $N = 0,2$ from the fact that
the metric spaces $\GGG_r(N)$ with $N=0,2$ are Urysohn and from Uspenskij's theorem \cite{usp} on the topology
of the Urysohn space.)\par
The proof of the existence of the groups $\GGG_r(N)$ is based on the general technique of Fra\"{\i}ss\'{e} limits. First
we construct a countable valued Abelian group $\QQQ\GGG_r(N)$, a counterpart of the so-called rational Urysohn metric
space, and then we prove that its completion is the group we search for. Although this approach is just an adaptation
of the original construction of the Urysohn space $\UUU$ described in \mbox{\cite{ur1,ur2}}, the details are more
complicated, mainly because finite groups admit no one-point extensions, in the opposite to finite metric spaces.
The unusual property of compact subgroups of $\GGG_r(N)$, formulated in \THM{main2}, corresponds to Huhunai\v{s}vili's
theorem \cite{huh} for compact subsets of $\UUU$ (which says that every isometry between two compact subsets of $\UUU$ is
extendable to an isometry of $\UUU$ onto itself). It is not so strong for $N \neq 0$, since every compact metric Abelian
group of finite exponent is totally disconnected. In contrast, $G_r(0)$ contains a copy of every metrizable compact Abelian
group, among which one may find groups which are universal (in the sense of topological embedding), as topological spaces,
for separable metrizable topological spaces, such as the countable infinite Cartesian power of $\RRR / \ZZZ$.\par
There are known examples of the so-called universal Polish groups, that is, of completely metrizable separable
(non-Abelian) topological groups $G$ such that every Polish group is isomorphic (as a topological group) to a closed
subgroup of $G$. For example, Uspenskij have shown that the homeomorphism group of the Hilbert cube \cite{us2} as well
as the isometry group of the Urysohn universal metric space \cite{us3} are (nonisomorphic) universal Polish groups
(cf. Remark just after Theorem~5.2 of \cite{me2}). In the opposite to this, the author knows no example of a universal
Polish Abelian group (i.e. a Polish Abelian group which is universal for Polish Abelian groups). In this paper we give two
such examples: $\GGG_1(0)$ and $\GGG_{\infty}(0)$. (Both these valued groups are of course `metrically' universal
for separable valued Abelian groups with values bounded by $1$.)\par
The part on pseudovector structures extends our earlier (introductory) work \cite{pn2} in this subject. The proofs
presented here are quite new and much more general. \THM{main3} generalizes and strengthens Theorem~4.3 of \cite{pn2}.
Since every norm on a nontrivial pseudovector group is unbounded, to equip the groups $\GGG_1(N)$ with `normed-like'
pseudovector structures we have to extend the notion of a norm to a subnorm, which is done in the recent paper.
It seems to be interesting whether one may distinguish a special subnormed topological pseudovector structure
on $\GGG_r(N)$ which will make this group universal for subnormed topological pseudovector groups of `$\Gg_r(N)$-like'
class (this cannot be infered from \THM{main3}). The one idea is to define a counterpart of the Gurari\v{\i} space for
pseudovector groups of suitable class (this is discussed in details in Section~9). We do not know yet whether such
a pseudovector group, denoted by $\PPP\VVV\GGG_r(N)$, exists and we leave this as an open problem. However, we prove that
if it only exists, it has to be isometrically group isomorphic to $\GGG_r(N)$ and that $\PPP\VVV\GGG_r(N)$ as a PV group
is unique up to isometric linear isomorphism. What is more, every subnormed topological PV group of suitable class
(to which $\PPP\VVV\GGG_r(N)$ belongs) is embeddable into $\PPP\VVV\GGG_r(N)$ by means of a isometric linear
homomorphism.\par
For well understanding of this exposition it is enough to know basic facts on metrizable and valued Abelian groups,
e.g. the material of Chapter~1 of \cite{ber}. The reader interested in Urysohn universal
space is referred to a survey article on the subject \cite{me2} or to \cite{gk2}, \cite{kat}, \cite{ur1,ur2}.
\vspace{0.3cm}

\textbf{Notation and terminology.} In this paper $\RRR$, $\QQQ$ and $\ZZZ$ stand for the sets of all real, all rational
and, respectively, all integer numbers. The symbol `$\id$' is reserved to denote the identity map. All groups are Abelian
and we use the additive notation. For simplicity, the action of any group is always denoted by (the same sign) `$+$' and
its neutral element is denoted by $0$. If $G$ is a group, $a \in G$ and $k \in \ZZZ$, the value at $k$ of a unique group
homomorphism of $\ZZZ$ into $G$ which sends $1$ to $a$ is denoted by $ka$ or $k \cdot a$. The group $G$ is said to be
\textit{of exponent $N$} (where $N \in \ZZZ$, $N \geqsl 2$) iff $N \cdot x = 0$ for any $x \in G$. The subgroup of $G$
consisting of all finite rank elements of $G$ is denoted by $G_{fin}$.\par
For $s,t \in [0,\infty]$ let $s \wedge t$ and $s \vee t$ stand for the minimum and the maximum of $s$ and $t$
(respectively). Similarly, if $f$ is a real-valued function and $t \in [0,\infty]$, $f \wedge t$ is a function with
the same domain as $f$ and $(f \wedge t)(x) = f(x) \wedge t$. In a similar manner we define $f \vee t$, and $f \wedge g$
and $f \vee g$ for two real-valued functions $f$ and $g$ with common domain.\par
A \textit{value} on a group $G$ is a function $p\dd G \to [0,\infty)$ such that for any $x, y \in G$,
\begin{enumerate}[(V1)]
\item $p(x) = 0 \iff x = 0$,
\item $p(-x) = p(x)$,
\item $p(x + y) \leqsl p(x) + p(y)$.
\end{enumerate}
If in the above the condition (V1) is replaced by `$p(0) = 0$', $p$ is called a \textit{semivalue}. Every value $p$ on $G$
induces an invariant metric $G \times G \ni (x,y) \mapsto p(x - y) \in [0,\infty)$. The topology induced by the value $p$
is the topology induced by the latter metric. So, we may speak of a separable, complete, compact (etc.) valued group.
Two values on a group are \textit{equivalent} iff they induce the same topology. A value on a topological group is said
to be \textit{compatible} if it induces the given topology of the group. By $\delta_G$ we denote the \textit{discrete}
value on $G$ defined by $\delta_G(g) = 1$ for $g \in G \setminus \{0\}$. For two valued groups $(G,+,p)$ and $(G',+,p')$
we shall write $(G',+,p') \supset (G,+,p)$ iff $G \subset G'$, $p'$ extends $p$ and the addition of $G'$ extends
the addition of $G$. If this happens, we say that $(G,+,p)$ is \textit{enlarged} to $(G',+,p')$.\par
Subgroups need not be closed. A subgroup generated by a subset $A$ of a group $G$ is denoted by $\grp{A}$. We write
$\grp{a_1,\ldots,a_n}$ instead of $\grp{\{a_1,\ldots,a_n\}}$. If $\psi\dd G \to H$ is a homomorphism between valued groups
$(G,+,p)$ and $(H,+,q)$, we use the term a \textit{group homomorphism} to underline that we make no additional assumptions
on the topological behavior of $\psi$. The group homomorphism is \textit{isometric} if $q(\psi(x)) = p(x)$ for each
$x \in G$. Adapting terminology of functional analysis, we say $\psi$ is \textit{$\epsi$-almost isometric} with $\epsi \in
(0,1)$ provided $(1-\epsi) p(x) \leqsl q(\psi(x)) \leqsl (1+\epsi) p(x)$ for every $x \in G$ (compare with
\eqref{eqn:alm-iso}).\par
The classes $\Gg$ and $\Gg_r(N)$ (with $r \in \{1,\infty\}$ and $N \in \ZZZ_+ \setminus \{1\}$) have the same meaning
as in Introduction. A metric space $(X,d)$ is said to be \textit{topologically} (respectively \textit{metrically})
\textit{universal} for a class of metric spaces provided every member of the class is homeomorphic (respectively
isometric) to a subset of $X$.\par
Whenever $f$ and $g$ are two real-valued functions defined on a common nonempty domain, the supremum distance
($\in [0,\infty]$) of $f$ and $g$ is denoted by $\|f - g\|_{\infty}$. Similarly, if $f$ and $g$ take values in a valued
group $(G,+,p)$ and there is no danger of confusion, we shall also write $\|f - g\|_{\infty}$ for the supremum distance
of $f$ and $g$ induced by $p$. The diameter of a metric space $(X,d)$ is denoted by $\diam X$ or $\diam (X,d)$. When
$(G,+,p)$ is a valued group, we shall also write $\diam (G,p)$.\par
For any subset $A$ of $\RRR$ let $A_+$ stand for the set $A \cap [0,\infty)$. In particular, $\RRR_+ = [0,\infty)$ and
$\ZZZ_+ = \{0,1,2,3,\ldots\}$. For two integers $k$ and $l$ we write $k | l$ if $k$ divides $l$ (i.e. if $l = mk$ for some
$m \in \ZZZ$). Notice that $k | 0$ for each $k \in \ZZZ$ and $0 | l$ iff $l = 0$.

\SECT{Preliminaries}

The results of this section will be used later. To make the paper better organized, the section is divided into parts.
In some of results we will deal with values on groups which takes values in a subset $Q$ of $\RRR$ such that:
\begin{equation}\label{eqn:Q}
Q_+ \textup{ is dense in } \RRR_+, \qquad 0 \in Q  \qquad \textup{and} \qquad Q_+ + Q_+ \subset Q.
\end{equation}
Such considerations will find applications in next sections.
\vspace{0.3cm}

\noindent\textbf{\thesection A. Boundedness.}
Denote by $\Omega$ the set of all moduli of continuity, that is, all functions (including the zero one)
satisfying conditions ($\omega$1)--($\omega$3). Additionally, let $\Omega^* = \Omega \setminus \{0\}$. Note that whenever
$\omega \in \Omega^*$ and $p$ is a (semi-)value on a group $G$, so is $\omega \circ p$ and if $q$ is a semivalue on $G$
such that
\begin{equation}\label{eqn:bdval}
q \leqsl \omega \circ p,
\end{equation}
then $q$ is continuous with respect to $p$. However, it may turn out that for some continuous (with respect to $p$)
semivalue $q$ there is no $\omega \in \Omega$ witnessing \eqref{eqn:bdval} (for example, this happens if $p$ is bounded
and $q$ is not, e.g. $p = q \wedge 1$). So, the condition \eqref{eqn:bdval} is stronger than continuity. Whenever it is
fulfilled for some $\omega \in \Omega^*$, we say that $q$ is \textit{$p$-bounded}. With every group homomorphism $\psi\dd
(G,+,p) \to (G',+,p')$ we may associate a semivalue $p_{\psi}$ on $G$: $p_{\psi} = p' \circ \psi$. It is easily seen that
$\psi$ is continuous iff $p_{\psi}$ is continuous (with respect to $p$). We say that $\psi$ is \textit{bounded} iff
$p_{\psi}$ is $p$-bounded. That is:
$$
\psi \textup{ is bounded} \iff \exists\,\omega \in \Omega\dd\ p' \circ \psi \leqsl \omega \circ p.
$$
The following are left as exercises. Some of them are quite easy, other are well-known results of real analysis:
\begin{enumerate}[\upshape(MC1)]
\item For every $\omega \in \Omega$ there is a finite limit $\lim_{x\to\infty} \frac{\omega(x)}{x}$ and it is equal to
   $\inf_{x>0} \frac{\omega(x)}{x}$.
\item If $\omega, \tau \in \Omega$, then $\omega \circ \tau \in \Omega$.
\item Every $\omega \in \Omega$ is uniformly continuous, and $\omega^{-1}(\{0\}) = \{0\}$ provided $\omega \not\equiv 0$.
\item If $\omega \in \Omega^*$ and $(x_n)_{n=1}^{\infty} \subset \RRR_+$, then $\lim_{n\to\infty} \omega(x_n) = 0 \iff
   \lim_{n\to\infty} x_n = 0$.
\end{enumerate}
For each $r \in [0,\infty)$ denote by $\Omega_r$ the set of all $\omega \in \Omega$ such that $\lim_{x\to\infty}
\frac{\omega(x)}{x} \leqsl r$ (cf. (MC1)). For us, the set $\Omega_0$ is of great importance.

\begin{exm}{om0}
Every bounded modulus of continuity belongs to $\Omega_0$. If $\omega \in \Omega$ and $\tau \in \Omega_0$, then
$\omega \circ \tau, \tau \circ \omega \in \Omega_0$. The following functions are members of $\Omega_0$: $x \mapsto
x \wedge 1$, $x \mapsto \frac{x}{x+1}$ and $x \mapsto x^{\alpha}$ ($0 < \alpha < 1$).
\end{exm}

The next result is well-known. It is a variation of \cite[\S1, Theorem~1]{a-p}. We shall use it
to characterize bounded group homomorphisms.

\begin{lem}{bdom}
For a function $f\dd \RRR_+ \to \RRR_+$ \tfcae
\begin{enumerate}[\upshape(i)]
\item there is $\omega \in \Omega$ such that $f \leqsl \omega$,
\item $\limsup_{x\to0^+} f(x) = f(0) = 0$ and $$r := \limsup_{n\to\infty} \frac{\sup f([0,2^n])}{2^n} < \infty.$$
\end{enumerate}
What is more, if \textup{(ii)} is fulfilled [and $f$ is bounded by $M$], there is $\omega \in \Omega_{2r}$ [bounded by $M$]
such that $f \leqsl \omega$.
\end{lem}

As an immediate consequence of \LEM{bdom} we obtain

\begin{pro}{bdhom}
Let $\psi\dd (G,+,p) \to (G',+,p')$ be a group homomorphism and let
$$
f\dd \RRR_+ \ni \xi \mapsto \sup \{p'(\psi(x))\dd\ x \in G,\ p(x) \leqsl \xi\} \in [0,\infty].
$$
Then:
\begin{enumerate}[\upshape(a)]
\item $\psi$ is continuous iff $\lim_{t\to0^+} f(t) = 0$,
\item $\psi$ is bounded iff $\lim_{t\to0^+} f(t) = 0$ and $\limsup_{t\to\infty} \frac{f(t)}{t} < \infty$.
\end{enumerate}
\end{pro}

\begin{cor}{bdhom}
If $\psi\dd (G,+,p) \to (G',+,p')$ is a continuous group homomorphism such that the set $\psi(G)$ is bounded in the metric
space $(G',p')$, then $\psi$ is bounded.
\end{cor}

\begin{cor}{conthom}
A group homomorphism $\psi\dd (G,+,p) \to (G',+,p')$ is continuous iff there are $\omega, \tau \in \Omega^*$ such that
$\tau \circ p' \circ \psi \leqsl \omega \circ p$.
\end{cor}
\begin{proof}
Sufficiency is clear (thanks to (MC4)). To prove the necessity, put $\tau(t) = t \wedge 1$ and apply \LEM{bdom}
for suitable $f\dd \RRR_+ \to [0,1]$.
\end{proof}
\vspace{0.3cm}

\noindent\textbf{\thesection B. Extending a value.}

\begin{lem}{valQ}
Let $Q$ be as in \eqref{eqn:Q}. Let $(D,+,\lambda)$ be a finite valued group and $D_0$ its subgroup such that
$\lambda(D_0) \subset Q$. Then for every $\epsi > 0$ there is a value $\bar{\lambda}$ on $D$ extending
$\lambda\bigr|_{D_0}$ such that $\bar{\lambda}(D) \subset Q$ and $\|\lambda - \bar{\lambda}\|_{\infty} \leqsl \epsi$.
What is more, if $\lambda \leqsl 1$, then $\bar{\lambda} \leqsl 1$.
\end{lem}
\begin{proof}
We may assume that $\epsi \in (0,1)$. For $h \in D_0$ let $\lambda'(h) = 0$ and for $h \in D \setminus D_0$ let
$\lambda'(h) \in [0,\epsi]$ be such that $\lambda'(-h) = \lambda'(h)$ and $\lambda(h) + \lambda'(h) \in Q$. For $x \in D$
put
$$
\bar{\lambda}(x) = \inf\{\sum_{j=1}^n (\lambda(h_j) + \lambda'(h_j))\dd\ n \geqsl 1,\ h_1,\ldots,h_n \in D,\
x = \sum_{j=1}^n h_j\}.
$$
It is easily seen that $\bar{\lambda}$ is a value on $D$ such that $\lambda \leqsl \bar{\lambda} \leqsl \lambda +
\lambda'$. This yields that $\|\lambda - \bar{\lambda}\|_{\infty} \leqsl \epsi$ and $\bar{\lambda} = \lambda$ on $D_0$.
Further, since $D$ is finite, the infimum in the formula for $\bar{\lambda}(x)$ is reached and therefore $\bar{\lambda}(D)
\subset Q$. Finally, if $\lambda \leqsl 1$, take $M \in Q \cap [1-\epsi,1]$ such that $\lambda(D_0) \subset [0,M]$ and
replace $\bar{\lambda}$ by $\bar{\lambda} \wedge M$.
\end{proof}

\begin{lem}{valom}
Let $r \in \{1,\infty\}$. Let $(D,+,\lambda)$ be a valued group with $\lambda \leqsl r$, $D_0$ its closed subgroup,
$\lambda_0$ a semivalue on $D_0$ and let $\omega \in \Omega^*$ be such that $\lambda_0 \leqsl
(\omega \circ \lambda)\bigr|_{D_0}$. Then there is a semivalue $\bar{\lambda}$ on $D$ which extends $\lambda_0$ and
satisfies the following conditions:
\begin{enumerate}[\upshape(a)]
\item $\bar{\lambda} \leqsl \omega \circ \lambda$,
\item $\bar{\lambda}^{-1}(\{0\}) = \lambda_0^{-1}(\{0\})$; in particular: $\bar{\lambda}$ is a value provided so is
   $\lambda_0$,
\item if $\lambda_0 \leqsl r$, then $\bar{\lambda} \leqsl r$,
\item if $\lambda_0$ is a value equivalent to $\lambda\bigr|_{D_0}$, then $\bar{\lambda}$ is equivalent to $\lambda$,
\item if for some $\tau, \varrho \in \Omega$ with $\varrho \not\equiv 0$ and \textup{($\omega$-$\varrho$-$\tau$)} one has
   $$\tau(\dist_{\lambda}(h,\lambda_0^{-1}(\{0\}))) \leqsl (\varrho \circ \lambda_0)(h)$$ for every $h \in D_0$, then
   $\tau(\dist_{\lambda}(x,\lambda_0^{-1}(\{0\}))) \leqsl (\varrho \circ \bar{\lambda})(x)$ for $x \in D$.
\end{enumerate}
\end{lem}
\begin{proof}
For $x \in D$ put
$$
\bar{\lambda}(x) = \inf\{(\omega \circ \lambda)(x-h) + \lambda_0(h)\dd\ h \in D_0\}.
$$
Further, if $\lambda_0 \leqsl r$, replace $\bar{\lambda}$ by $\bar{\lambda} \wedge r$. One easily verifies that
$\bar{\lambda}$ is a semivalue extending $\lambda_0$ which satisfies (a) and (d). The point (b) follows from the closedness
of $D_0$. We shall only show (e). For simplicity, put $q(x) = \dist_{\lambda}(x,\lambda_0^{-1}(\{0\}))$ ($x \in D$). Thanks
to the continuity of $\varrho$, it suffices to check that $(\tau \circ q)(x) \leqsl \varrho((\omega \circ \lambda)(x-h) +
\lambda_0(h))$ for $x \in D$ and $h \in D_0$ (the second condition in ($\omega$-$\varrho$-$\tau$) allows us to replace
$\bar{\lambda}$ by $\bar{\lambda} \wedge r$). But $(\tau \circ q)(x) \leqsl (\tau \circ \lambda)(x-h) + (\varrho \circ
\lambda_0)(h)$ and hence it is enough to show that
$$
(\tau \circ \lambda)(x-h) + (\varrho \circ \lambda_0)(h) \leqsl \varrho((\omega \circ \lambda)(x-h) + \lambda_0(h)),
$$
which is the first condition in ($\omega$-$\varrho$-$\tau$) with $t = \lambda(x-h)$ and $s = \lambda_0(h)$.
\end{proof}

\begin{exm}{o-r-t}
Let $\omega_0, \varrho_0 \in \Omega^*$, $\tau_0 \in \Omega$ and $r \in \{1,\infty\}$. Put $\omega = \omega_0 \vee \tau_0$,
$\varrho = \varrho_0 + \id$ and $\tau = \tau_0 \wedge \varrho(r)$ ($\varrho(\infty) := \lim_{t\to\infty} \varrho(t)$). Then
$\omega, \varrho$ and $\tau$ are moduli of continuity such that $\tau \leqsl \tau_0$, $\varrho_0 \leqsl \varrho$, $\omega_0
\leqsl \omega$ and ($\omega$-$\varrho$-$\tau$) is fulfilled. The example shows that we may always replace moduli
of continuity appearing in the inequalities in points (a) and (e) of \LEM{valom} by other moduli in such a way that these
inequalities are still satisfied and the strange condition ($\omega$-$\varrho$-$\tau$) is fulfilled.
\end{exm}
\vspace{0.3cm}

\noindent\textbf{\thesection C. The class $\OOo_0$.}
We call a group \textit{of class $\OOo_{fin}$} iff every its element has finite rank. That is, $G$ is of class $\OOo_{fin}$
if $G = G_{fin}$.\par
Let $(G,+,p)$ be an arbitrary valued group and $x$ its element. Since the sequence $a_n := p(nx)\ (n \geqsl 1)$ satisfies
the condition
$$
a_{n+m} \leqsl a_n + a_m \quad (n,m \geqsl 1),
$$
the sequence $(\frac{a_n}{n})_{n=1}^{\infty}$ has finite limit and $\lim_{n\to\infty} \frac{a_n}{n} = \inf_{n\geqsl1}
\frac{a_n}{n}$. Put $p_{0*}\dd G \to \RRR_+$,
$$
p_{0*}(x) = \lim_{n\to\infty} \frac{p(nx)}{n}.
$$
Observe that $p_{0*}$ is a semivalue on $G$ such that $p_{0*} \leqsl p$. Thus $p_{0*}$ is $p$-bounded and hence the set
$G_{0*} = p_{0*}^{-1}(\{0\})$ is a closed subgroup of $(G,+,p)$. Let $G'$ be the quotient group $G / G_{0*}$ equipped with
the value $p'$ (naturally) induced by $p_{0*}$. Observe that
\begin{equation}\label{eqn:00}
p'(kx) = |k| p'(x)
\end{equation}
for every $x \in G'$ and $k \in \ZZZ$ and hence $p'_{0*} = p'$.\par
We say the valued group $(G,+,p)$ is \textit{of class $\OOo_0$} iff $G_{0*} = G$ or, equivalently, if $p_{0*} \equiv 0$.
It is \textit{of class $\OOo_{\infty}$} iff $p_{0*} = p$ (that is, if \eqref{eqn:00} is fulfilled with $p'$ replaced
by $p$). Finally, $(G,+,p)$ is \textit{of class $\OOo_1$} iff $p_{0*}$ is a value on $G$ or, equivalently, if $G_{0*}$ is
trivial. Note that $\OOo_\infty \subset \OOo_1$.\par
We begin with an interesting characterization of members of the above mentioned classes.\par
Making use of \eqref{eqn:00} and repeating the proof of the Hahn-Banach theorem one easily gets

\begin{lem}{H-B}
Let $(G,+,p)$ be a valued group.
\begin{enumerate}[\upshape(A)]
\item If $\psi\dd G \to \RRR$ is a group homomorphism such that for each $x \in G$,
   \begin{equation}\label{eqn:nonexp}
   |\psi(x)| \leqsl p(x),
   \end{equation}
   then $|\psi(x)| \leqsl p_{0*}(x)$ for every $x \in G$.
\item For each $a \in G$ there is a group homomorphism $\psi\dd G \to \RRR$ satisfying \eqref{eqn:nonexp} for all $x \in G$
   such that $\psi(a) = p_{0*}(a)$.
\end{enumerate}
\end{lem}

Let us call a group homomorphism $\psi\dd G \to \RRR$ satisfying \eqref{eqn:nonexp} \textit{nonexpansive}.
The above result yields

\begin{thm}{ooo}
Let $(G,+,p)$ be a valued group.
\begin{enumerate}[\upshape(I)]
\item $G$ is of class $\OOo_0$ iff $G$ admits no nonzero bounded group homomorphisms into Banach spaces.
\item $G$ is of class $\OOo_{\infty}$ iff $G$ admits an isometric group homomorphism into a Banach space.
\item $G$ is of class $\OOo_1$ iff $G$ admits a nonexpansive group homomorphism with trivial kernel into a Banach space,
   iff $G$ admits a bounded group homomorphism with trivial kernel into a Banach space.
\end{enumerate}
\end{thm}
\begin{proof}
(I): Suppose $(G,+,p)$ is of class $\OOo_0$, $(E,\|\cdot\|)$ is a Banach space, $\omega \in \Omega$ and $\psi\dd G \to E$
is a group homomorphism such that $\|\psi(x)\| \leqsl (\omega \circ p)(x)$ for $x \in G$. Then we have
\begin{equation}\label{eqn:aux1}
\|\psi(x)\| \leqsl \frac{(\omega \circ p)(nx)}{n} = \frac{\omega(n \cdot \frac{p(nx)}{n})}{n} \leqsl \frac{p(nx)}{n} \to 0.
\end{equation}
The inverse implication follows from the point (B) of \LEM{H-B}.\par
(II): Assume $(G,+,p)$ is of class $\OOo_{\infty}$. Let $Z$ be the set of all nonexpansive group homomorphisms of $G$ into
$\RRR$ and let $E$ be the Banach space of all real-valued bounded functions on $Z$, equipped with the supremum norm. For
$g \in G$ let $e_g \in E$ be given by $e_g(\xi) = \xi(g)$. Then the function $G \ni g \mapsto e_g \in E$ is an isometric
group homomorphism (again by \LEM{H-B}). The inverse implication is immediate.\par
(III): If $(G,+,p)$ is of class $\OOo_1$, then $(G,+,p_{0*})$ is a valued group of class $\OOo_{\infty}$ and thus
the assertion follows from (II). Conversely, if $\psi\dd G \to E$ is a bounded group homomorphism of $G$ into a Banach
space $E$, \eqref{eqn:aux1} shows that $p_{0*}(x) \geqsl \|\psi(x)\|$.
\end{proof}

\begin{rem}{bd-cont}
\THM{ooo} expresses how far is boundedness from continuity: only valued groups of class $\OOo_1$ admit bounded group
homomorphisms with trivial kernels into Banach spaces. In contrast, there are valued groups of class $\OOo_0$ which are
group isomorphic to Banach spaces: if $(E,\|\cdot\|)$ is a Banach space, then $(E,\|\cdot\| \wedge 1)$ is of class
$\OOo_0$. By the way, the latter example shows that membership to any of the classes $\OOo_0$, $\OOo_1$, $\OOo_{\infty}$
is not a topological group invariant.\par
\THM{ooo} also implies that every valued group $G$ has a unique closed subgroup $H$ (namely, $G_{0*}$) such that
$H$ is of class $\OOo_0$ and $G / H$ admits a nonexpansive group homomorphism with trivial kernel into a Banach space.\par
It also follows from \THM{ooo} and the well-known theorem of Banach and Mazur that the Banach space $\CCc([0,1])$ of all
continuous real-valued functions on $[0,1]$, as a valued group, is universal for separable valued Abelian groups of class
$\OOo_{\infty}$.
\end{rem}

From now on, we are only interested in valued groups of class $\OOo_0$. The reader will easily check that every group
with bounded value is of this class and that $\OOo_{fin} \subset \OOo_0$. It turns out that, in a sense, $\OOo_0$ is
the `closure' of $\OOo_{fin}$. Formally this is formulated in the following

\begin{thm}{o0}
A valued group $(G,+,p)$ is of class $\OOo_0$ iff it may be enlarged to a valued group $(\tilde{G},+,\tilde{p})$ such that
$\tilde{G}_{fin}$ is dense in $\tilde{G}$.
\end{thm}
\begin{proof}
The sufficiency is clear, since $\tilde{G}_{0*}$ is closed and contains $\tilde{G}_{fin}$. To prove the necessity,
thanks to transfinite induction, it is enough to show that if $p_{0*}(a) = 0$, then for every $\epsi > 0$ there is
a valued group $(\tilde{G},+,\tilde{p}) \supset (G,+,p)$ and $b \in \tilde{G}_{fin}$ such that $\tilde{p}(a - b) \leqsl
\epsi$.\par
Let $N \geqsl 2$ be such that
\begin{equation}\label{eqn:aux2}
p(na) \leqsl \epsi n \quad \textup{for } n \geqsl N.
\end{equation}
Take $M > 0$ such that $M \geqsl p(ja)$ for $j=1,\ldots,N$. Let $H = \grp{b}$ be a cyclic group of rank $N$. We identify
$g \in G$ with $(g,0) \in G \times H =: \tilde{G}$. Define $\tilde{p}\dd \tilde{G} \to \RRR_+$ by
$$
\tilde{p}(g,h) = \inf\{p(g-ka) + |k| \epsi + M \delta_H(kb+h)\dd\ k \in \ZZZ\}.
$$
It is clear that $\tilde{p}$ is a value on $\tilde{G}$ such that $\tilde{p}(g,0) \leqsl p(g)$ for $g \in G$ and
$\tilde{p}((a,0) - (0,b)) \leqsl \epsi$. So, it suffices to show that $p(g) \leqsl \tilde{p}(g,0)$, which is equivalent to
\begin{equation}\label{eqn:aux3}
p(ka) \leqsl |k| \epsi + M \delta_H(kb)
\end{equation}
for $k \in \ZZZ$. We may assume $k \neq 0$. If $|k| \geqsl N$, \eqref{eqn:aux3} is covered by \eqref{eqn:aux2}. Finally,
if $0 < |k| < N$, then $p(ka) \leqsl M = M \cdot \delta_H(kb)$ and we are done.
\end{proof}

Note that if in the above theorem $G$ is separable, $\tilde{G}$ may be constructed to be separable as well and that in that
case the proof is constructive. It is also easily seen that we may force $\tilde{G}$ to have the same diameter as $G$.

\begin{exm}{o0}
Let $(G,+,p)$ be a valued group and $\omega \in \Omega^*$. If $G$ is of class $\OOo_0$ or $\omega \in \Omega_0$, then
$(G,+,\omega \circ p)$ is of class $\OOo_0$. In particular, for every separable valued group $(G,+,p)$, $(G,+,p^{\alpha})
\in \Gg_{\infty}(0)$ for $0 < \alpha < 1$ and $(G,+,p \wedge 1), (G,+,\frac{p}{p+1}) \in \Gg_1(0)$. The value $p$ may be
reconstruct from each of the values $p^{\alpha}$ ($0 < \alpha < 1$) and $\frac{p}{p+1}$, and since $p(x) =
\lim_{\alpha\to1^-} p^{\alpha}(x)$ for every $x \in G$, one may say that each valued group can be `approximated' by valued
groups of class $\OOo_0$. It is also clear that every metrizable topological group admits a compatible value under which
it becomes a valued group of class $\OOo_0$.
\end{exm}

Taking into account \REM{bd-cont}, we see that the image of a valued group of class $\OOo_0$ under a continuous group
homomorphism need not be of class $\OOo_0$. However, the reader will easily check that

\begin{pro}{image}
A subgroup (equipped with the inherited value) and the completion of a valued group of class $\OOo_0$ is of class $\OOo_0$
as well. The image of a valued group of class $\OOo_0$ under a bounded group homomorphism is of class $\OOo_0$.
\end{pro}
\vspace{0.3cm}

\noindent\textbf{\thesection D. Enlarging a finite valued group.}
A very useful method of constructing the Urysohn universal space, introduced by Kat\v{e}tov \cite{kat}, involves
the technique of the so-called \textit{Kat\v{e}tov maps} which correspond to one-point extensions of metric spaces.
In this part we shall describe how to enlarge valued groups by means of these maps. As a corollary of the results presented
below we shall obtain in the sequel a rather surprising result that $\GGG_r(N)$ is Urysohn iff $N \in \{0,2\}$.\par
Recall that a Kat\v{e}tov map on a metric space $(X,d)$ is a function $f\dd X \to \RRR_+$ such that $|f(x) - f(y)| \leqsl
d(x,y) \leqsl f(x) + f(y)$ for any $x, y \in X$. The set of all Kat\v{e}tov maps on $X$ is denoted by $E(X)$. Additionally,
for $r \in [0,\infty]$ let $E_r(X)$ be the set of all $f \in E(X)$ such that $f(x) \leqsl r$ for each $x \in X$, and let
$E^i(X) = E_{\diam X}(X)$. Kat\v{e}tov maps belonging to $E^i(X)$ are called \textit{inner}. If $A \subset B \subset X$,
a Kat\v{e}tov map $f\dd B \to \RRR_+$ is said to be \textit{trivial on $A$ in $X$} iff there is $b \in X$ such that $f(a) =
d(a,b)$ for each $a \in A$. The map $f$ is \textit{trivial in $X$} if $f$ is trivial on its domain in $X$.\par
A fundamental result on the Urysohn space says that a nonempty separable complete metric space $X$ is Urysohn iff every
inner Kat\v{e}tov map is trivial in $X$ on every finite subset of the space (see e.g. \cite{me2}).\par
We begin with

\begin{pro}{kat-exp}
Let $(G,+,p)$ be a valued group of exponent $N \geqsl 3$. If $A \subset G$ and $f \in E(A)$ is trivial in $G$, then for
any $a_1,\ldots,a_N \in A$,
\begin{equation}\label{eqn:trv-N}
|p(\sum_{k=1}^N a_k) - f(a_N)| \leqsl \sum_{k=1}^{N-1} f(a_k).
\end{equation}
\end{pro}
\begin{proof}
Let $x \in G$ be such that $f(a) = p(x-a)$ for $a \in A$. Then, since $N \cdot x = 0$,
\begin{multline*}
|p(\sum_{k=1}^N a_k) - f(a_N)| \leqsl p(\sum_{k=1}^N a_k + (x-a_N)) = p(\sum_{k=1}^{N-1}(a_k - x)) \leqsl\\
\leqsl \sum_{k=1}^{N-1} p(a_k - x) = \sum_{k=1}^{N-1} f(a_k).
\end{multline*}
\end{proof}

\begin{rem}{N}
If $(G,+,p)$ is a valued group of exponent $N \geqsl 3$ and $f$ is a Kat\v{e}tov map whose domain $H$ is a subgroup of $G$,
then \eqref{eqn:trv-N} is equivalent to
\begin{equation}\label{eqn:trvH-N}
f(-\sum_{k=1}^{N-1} a_k) \leqsl \sum_{k=1}^{N-1} f(a_k) \quad \textup{for any } a_1,\ldots,a_{N-1} \in H.
\end{equation}
Indeed, if $a_1,\ldots,a_{N-1} \in H$ are fixed, then, since $-\sum_{k=1}^{N-1} a_k \in H$, $\sup_{a_N \in H}
|p(\sum_{k=1}^N a_k) - f(a_N)| = f(-\sum_{k=1}^{N-1} a_k)$.
\end{rem}

Our aim is to prove, in a sense, the converse of \PRO{kat-exp}. Namely,

\begin{thm}{trv}
Let $r \in \{0,\infty\}$, $N \in \ZZZ_+ \setminus \{1\}$ and let $(G,+,p) \in \Gg_r(N)$ be a bounded valued group. Let $A$
be a nonempty subset of $G$ and $f \in E_r(A)$. If $N > 2$, we assume that $f$ satisfies \eqref{eqn:trv-N} for all
$a_1,\ldots,a_N \in A$. Then there is a bounded valued group $(\tilde{G},+,\tilde{p}) \in \Gg_r(N)$ such that
$(\tilde{G},+,\tilde{p}) \supset (G,+,p)$ and $f$ is trivial in $\tilde{G}$.\par
If, in addition, $Q$ is as in \eqref{eqn:Q}, $G$ is finite and $p(G) \cup f(A) \subset Q$, the group $\tilde{G}$ may be
taken so that it is finite and $\tilde{p}(\tilde{G}) \subset Q$.
\end{thm}
\begin{proof}
First of all observe that \eqref{eqn:trv-N} is fulfilled for any $a_1,\ldots,a_N \in A$ also when $N = 2$, because then
$a_1 + a_2 = a_1 - a_2$. Note that if $\inf f(A) = 0$, then $f$ is trivial in the completion of $G$, so we may and do
assume that $c := \inf f(A) > 0$. Further, let $M = \sup p(G) \vee \sup f(A)$ ($M \in Q$ provided all additional conditions
of the theorem are fulfilled; $f$ is bounded because $G$ is so). If $N \neq 0$, let $m = N$. Otherwise, take $m \geqsl 2$
such that $m - 1 \geqsl \frac{M}{c}$. Let $H = \grp{b}$ be a cyclic group of rank $m$. Put $\tilde{G} = G \times H$.
We identify each $g \in G$ with $(g,0) \in \tilde{G}$. Let us agree that $\sum_{j=a}^b = 0$ provided $a > b$. Define
$\tilde{p}\dd \tilde{G} \to \RRR_+$ by the formula
\begin{multline*}
\tilde{p}(g,h) = \inf\{p(g - \sum_{j=1}^n \epsi_j a_j) + \sum_{j=1}^n f(a_j) + M \delta_H(\sum_{j=1}^n \epsi_j b - h)\dd\\
n \geqsl 0,\ a_1,\ldots,a_n \in A,\ \epsi_1,\ldots,\epsi_n \in \{-1,1\}\}.
\end{multline*}
It is easy to verify that $\tilde{p}$ is a value. What is more, if all additional conditions of the theorem are satisfied,
the infimum in the formula for $\tilde{p}(g,h)$ is reached and thus $\tilde{p}(\tilde{G}) \subset Q$. Observe that
$\tilde{p}(g,0) \leqsl p(g)$ for each $g \in G$ and $\tilde{p}(a,b) \leqsl f(a)$ for $a \in A$. If we show that in both
these inequalities one may put the equality sign, the proof will be completed. (Indeed, to make then the value $\tilde{p}$
bounded by $r$, it suffices to replace $\tilde{p}$ by $\tilde{p} \wedge M$.)\par
First we will check that $\tilde{p}(g,0) = p(g)$ for $g \in G$. Equivalently, we have to show that
\begin{equation}\label{eqn:aux10}
\sum_{j=1}^n f(a_j) + M \delta_H(\sum_{j=1}^n \epsi_j b) \geqsl p(\sum_{j=1}^n \epsi_j a_j)
\end{equation}
for any $n \geqsl 1$ (for $n=0$ it is clear), $a_1,\ldots,a_n \in A$ and $\epsi_1,\ldots,\epsi_n \in \{-1,1\}$.
If $\sum_{j=1}^n \epsi_j b \neq 0$, then $\delta_H(\sum_{j=1}^n \epsi_j b) = 1$ and hence \eqref{eqn:aux10} is fulfilled,
by the definition of $M$. Thus we may assume that $\sum_{j=1}^n \epsi_j b = 0$, that is, $m | \sum_{j=1}^n \epsi_j$.\par
If $\sum_{j=1}^n \epsi_j \neq 0$ and $N = 0$, then $n \geqsl |\sum_{j=1}^n \epsi_j| \geqsl m$, so $\sum_{j=1}^n f(a_j)
\geqsl nc \geqsl mc \geqsl M \geqsl p(\sum_{j=1}^n \epsi_j a_j)$ and consequently \eqref{eqn:aux10} is fulfilled.\par
If $\sum_{j=1}^n \epsi_j \neq 0$ and $N \neq 0$, then $m = N$ and $\sum_{j=1}^n \epsi_j = lN$ for some $l \in \ZZZ
\setminus \{0\}$. This means that, after renumeration of $\epsi_j$'s (and $a_j$'s), we may assume $\epsi_j = l/|l|$
for $j = 1,\ldots,|l|N$ and $\sum_{j=|l|N+1}^n \epsi_j = 0$. Now making use of \eqref{eqn:trv-N} we infer that
$\sum_{k=1}^{|l|} p(\sum_{j=(k-1)N+1}^{kN} a_j) \leqsl \sum_{j=1}^{|l|N} f(a_j)$. So, if only $p(\sum_{j=|l|N+1}^n \epsi_j
a_j) \leqsl \sum_{j=|l|N+1}^n f(a_j)$, \eqref{eqn:aux10} will be satisfied. This reduces the problem to the case when
$\sum_{j=1}^n \epsi_j = 0$ (and $N$ is arbitrary).\par
Finally, if $\sum_{j=1}^n \epsi_j = 0$, then $n$ is even, say $n = 2k$, and---after renumeration---we may assume $\epsi_j =
(-1)^{j-1}$ for $j = 1,\ldots,n$.  But then
\begin{multline*}
p(\sum_{j=1}^n \epsi_j a_j) = p(\sum_{j=1}^k (a_{2j-1} - a_{2j})) \leqsl\\
\leqsl \sum_{j=1}^k (f(a_{2j-1}) + f(a_{2j})) = \sum_{j=1}^n f(a_j)
\end{multline*}
which finishes the proof of \eqref{eqn:aux10}. Now we pass to proving that $\tilde{p}(a,b) = f(a)$ for $a \in A$. This is
equivalent to
\begin{equation}\label{eqn:aux11}
p(a - \sum_{j=1}^n \epsi_j a_j) + \sum_{j=1}^n f(a_j) + M \delta_H(\sum_{j=1}^n \epsi_j b - b) \geqsl f(a)
\end{equation}
for each $n \geqsl 0$ and arbitrary $a_1,\ldots,a_n \in A$ and $\epsi_1,\ldots,\epsi_n \in \{-1,1\}$. As before,
if $\sum_{j=1}^n \epsi_j b \neq b$, then, thanks to the definition of $M$, \eqref{eqn:aux11} is fulfilled. Thus we may
assume $m | \sum_{j=1}^n \epsi_j - 1$ (so, $n > 0$).\par
If $\sum_{j=1}^n \epsi_j \neq 1$ and $N = 0$, $n \geqsl |\sum_{j=1}^n \epsi_j| \geqsl m-1 \geqsl \frac{M}{c}$ and hence
$\sum_{j=1}^n f(a_j) \geqsl nc \geqsl M \geqsl f(a)$.\par
If $\sum_{j=1}^n \epsi_j \neq 1$ and $N \neq 0$, then $\sum_{j=1}^n \epsi_j = lN + 1$ for some $l \in \ZZZ \setminus
\{0\}$. Then either $\epsi_1 = \ldots = \epsi_n = -1$ or at least $|l|N$ of $\epsi_j$'s are equal to $l/|l|$. In the first
case we get $n = |l|N - 1$ and, $|l|$ times making use of \eqref{eqn:trv-N}:
\begin{multline*}
p(a + \sum_{j=1}^n a_j) + \sum_{j=1}^n f(a_j) \geqsl p(a + \sum_{j=1}^{N-1} a_{|l|N-j}) + \sum_{j=1}^{N-1} f(a_{|l|N-j})\\
+ \sum_{k=1}^{|l|-1} [\sum_{j=0}^{N-1} f(a_{kN-j}) - p(\sum_{j=0}^{N-1} a_{kN-j})] \geqsl f(a),
\end{multline*}
which gives \eqref{eqn:aux11}. In the second case we may assume, after renumeration, that $\epsi_j = l/|l|$ for $j =
1,\ldots,|l|N$ and $\sum_{j=|l|N+1}^n \epsi_j = 1$. Then, again by \eqref{eqn:trv-N}, $\sum_{k=1}^{|l|} [\sum_{j=0}^{N-1}
f(a_{kN-j}) - p(\sum_{j=0}^{N-1} a_{kN-j}] \geqsl 0$. So, if only $p(a - \sum_{j=|l|N+1}^n \epsi_j a_j) + \sum_{j=|l|N+1}^n
f(a_j) \geqsl f(a)$, \eqref{eqn:aux11} will be satisfied. This reduces the problem to the case when $\sum_{j=1}^n \epsi_j
= 1$ (and $N$ is arbitrary).\par
Finally, when $\sum_{j=1}^n \epsi_j = 1$, then $n$ is odd, say $n = 2k+1$, and (after renumeration) $\epsi_j = (-1)^{j-1}$.
But then
\begin{multline*}
p(a - \sum_{j=1}^n \epsi_j a_j) + \sum_{j=1}^n f(a_j) \geqsl p(a - a_n) + f(a_n)\\+ \sum_{j=1}^k [f(a_{2j-1}) + f(a_{2j})
- p(a_{2j} - a_{2j-1})] \geqsl f(a).
\end{multline*}
\end{proof}

In \cite{pn1} we have shown that the Urysohn universal metric space $\UUU$ admits a (unique) structure of a metric group
of exponent $2$ (which, in fact, will turn out to be isometrically group isomorphic to $\GGG_{\infty}(2)$). At the end
of \cite{pn1} we asked (following the referee) whether $U$ admits group structures of other exponents. Applying
\REM{N}, below we give a partial (negative) answer to this problem.

\begin{pro}{no3}
There is no nontrivial valued Abelian group of exponent $3$ which is Urysohn as a metric space.
\end{pro}
\begin{proof}
Let $(H,+,q)$ be a nontrivial valued Abelian group. Take $h \in H$ such that $0 < q(h) < \frac12 \sup q(H)$ (if there is
no such $h$, then $H$ is automatically non-Urysohn). Notice that $q(h) = q(-h) = q(h-(-h))$. Let $f\dd \{0,h,-h\} \to
\RRR_+$ be given by $f(h) = f(-h) = \frac12 q(h)$ and $f(0) = \frac32 q(h)$. It is easily seen that $f$ is a Kat\v{e}tov
map bounded by the diameter of $H$. If $f$ was trivial in $H$, we would have (by \eqref{eqn:trvH-N}) $f(0) \leqsl f(h) +
f(-h)$, which fails to be true.
\end{proof}

We do not know whether it is possible to show, using only \eqref{eqn:trvH-N}, the counterpart of \PRO{no3} for exponents
greater than $3$.
\vspace{0.3cm}

\noindent\textbf{\thesection E. Amalgamation lemmas.}
In constructions of Fra\"{\i}ss\'{e} limits, possibility of amalgamating (i.e. `gluing') spaces is the crucial axiom.
However, only the first of the results stated below will be used for this purpose. From now to the end of the section,
$r \in \{1,\infty\}$ and $N \in \ZZZ_+ \setminus \{1\}$ are fixed. To avoid repetitions, let us make the following
preliminary annotation (which will be used in the proofs of the next three lemmas):
\begin{equation}\label{eqn:pre}
\left\{\begin{array}{l}
\tilde{D} = (D_1 \times D_2) / \tilde{D}_0\\
\pi\dd D_1 \times D_2 \to \tilde{D} \textup{ the quotient group homomorphism}\\
\tilde{\psi}_1\dd D_1 \ni x \mapsto \pi(x,0) \in \tilde{D}\\
\tilde{\psi}_2\dd D_2 \ni y \mapsto \pi(0,y) \in \tilde{D}
\end{array}\right.
\end{equation}
Note that below \PRO{image} is applied several times without referring to it.

\begin{lem}{A1}
Let $(D_j,+,\lambda_j) \in \Gg_r(N)$ ($j=0,1,2$) and for $j=1,2$ let $\varphi_j\dd D_0 \to D_j$ be an isometric group
homomorphism. Then there exist a valued group $(D,+,\lambda) \in \Gg_r(N)$ and isometric group homomorphisms $\psi_j\dd
D_j \to D$ ($j=1,2$) such that $\psi_2 \circ \varphi_2 = \psi_1 \circ \varphi_1$. If, in addition, $Q$ is as in
\eqref{eqn:Q}, $D_1$ and $D_2$ are finite and $\lambda_j(D_j) \subset Q$ for $j=1,2$, then $D$ is finite as well
and $\lambda(D) \subset Q$.
\end{lem}
\begin{proof}
Let $\tilde{D_0} = \{(\varphi_1(x),-\varphi_2(x))\dd\ x \in D_0\} \subset D_1 \times D_2$ and $\tilde{D}$, $\pi$,
$\tilde{\psi}_1$ and $\tilde{\psi}_2$ be as in \eqref{eqn:pre}. Then $\tilde{\psi}_2 \circ \varphi_2 = \tilde{\psi}_1 \circ
\varphi_1$. Define $\tilde{\lambda}\dd \tilde{D} \to \RRR_+$ by the formula
$$
\tilde{\lambda}(z) = \inf\{\lambda_1(x_1) + \lambda_2(x_2)\dd\ (x_1,x_2) \in \pi^{-1}(\{z\})\}.
$$
It is clear that $\tilde{\lambda}$ is a well defined semivalue on $\tilde{D}$. Now put $D = \tilde{D} /
\tilde{\lambda}^{-1}(\{0\})$ and define $\lambda$, $\psi_1$ and $\psi_2$ by the rules $\lambda \circ \tilde{\pi}
= \tilde{\lambda}$ and $\psi_j = \tilde{\pi} \circ \tilde{\psi_j}$ ($j=1,2$) where $\tilde{\pi}\dd \tilde{D} \to D$
is the quotient group homomorphism. Finally, replace $\lambda$ by $\lambda \wedge 1$ if $r = 1$. We leave this as a simple
exercise that all assertions are satisfied.
\end{proof}

\begin{lem}{A2}
Let $(D_1,+,\lambda_1), (D_2,+,\lambda_2) \in \Gg_r(N)$ be finite valued groups and $D_0$ be a subgroup of $D_1$. Let
$u\dd D_0 \to D_2$ and $v\dd D_1 \to D_2$ be an isometric and, respectively, an $\epsi$-almost isometric group homomorphism
(where $\epsi \in (0,1)$) such that
\begin{equation}\label{eqn:aux19}
\|u - v\bigr|_{D_0}\|_{\infty} \leqsl \epsi.
\end{equation}
Then there are a finite valued group $(D,+,\lambda) \in \Gg_r(N)$ and isometric group homomorphisms $w_j\dd D_j \to D$
($j=1,2$) such that $w_1\bigr|_{D_0} = w_2 \circ u$ and
\begin{equation}\label{eqn:w1-w2}
\|w_1 - w_2 \circ v\|_{\infty} \leqsl A \epsi
\end{equation}
where $A = 1+\diam(D_1,\lambda_1)$.
\end{lem}
\begin{proof}
Let $\tilde{D}_0 = \{(x,-u(x))\dd\ x \in D_0\} \subset D_1 \times D_2$ and $\tilde{D}$, $\pi$, $\tilde{\psi}_1$ and
$\tilde{\psi}_2$ be as in \eqref{eqn:pre}. Put $D = \tilde{D}$ and $w_j = \tilde{\psi}_j$ ($j=1,2$). It is easily checked
that $w_2 \circ u = w_1\bigr|_{D_0}$ and $D = w_1(D_1) + w_2(D_2)$. Define $\lambda\dd D \to \RRR_+$ by
\begin{multline*}
\lambda(z) = \inf\{\lambda_1(x_1-x_0) + A\epsi \delta_D(w_1(x_0) - (w_2 \circ v)(x_0)) + \lambda_2(x_2 + v(x_0))\dd\\
x_0, x_1 \in D_1,\ x_2 \in D_2,\ z = w_1(x_1) + w_2(x_2)\}.
\end{multline*}
We see that $\lambda$ is a semivalue on $D$ and, since $D_1$ and $D_2$ are finite, the infimum in the formula for
$\lambda(z)$ is reached. Therefore, if $\lambda(z) = 0$, then for some $x_0, x_1 \in D_1$ and $x_2 \in D_2$ one has $z =
w_1(x_1) + w_2(x_2)$ and $x_1 - x_0 = 0$, $w_1(x_0) - (w_2 \circ v)(x_0) = 0$ and $x_2 + v(x_0) = 0$. These yield $x_1 =
x_0$, $x_2 = -v(x_1)$ and $w_1(x_1) = -w_2(x_2)$ and consequently $z = 0$. So, $\lambda$ is a value. Moreover, it follows
from the definition of $\lambda$ that \eqref{eqn:w1-w2} is satisfied and that $\lambda(w_j(x)) \leqsl \lambda_j(x)$ for
$x \in D_j$ ($j=1,2$). We shall now prove that in the latter inequalities one may put the equality sign. Fix $j \in
\{1,2\}$. For $x \in D_j$ we have to show that
\begin{equation}\label{eqn:aux20}
\lambda_1(x_1 - x_0) + A\epsi \delta_D(w_1(x_0) - (w_2 \circ v)(x_0)) + \lambda_2(x_2 + v(x_0)) \geqsl \lambda_j(x)
\end{equation}
provided $x_0, x_1 \in D_1$, $x_2 \in D_2$ and
\begin{equation}\label{eqn:aux20a}
w_1(x_1) + w_2(x_2) = w_j(x).
\end{equation}
First assume $j=1$. In that case \eqref{eqn:aux20a} means that $(x_1 - x,x_2) \in \tilde{D}_0$ and thus $h := x - x_1 \in
D_0$ and $x_2 = u(h)$. So, \eqref{eqn:aux20} has the form $\lambda_1(x_1 - x_0) + A\epsi \delta_D(w_1(x_0)
- (w_2 \circ v)(x_0)) + \lambda_2(u(h) + v(x_0)) \geqsl \lambda_1(x_1 + h)$ which, after substitution $x_1 := x_0$, is
equivalent to
\begin{equation}\label{eqn:aux21}
A\epsi \delta_D(w_1(x_0) - (w_2 \circ v)(x_0)) + \lambda_2(u(h) + v(x_0)) \geqsl \lambda_1(x_0 + h).
\end{equation}
If $w_1(x_0) = w_2(v(x_0))$, then $x_0 \in D_0$ and $v(x_0) = u(x_0)$ and consequently \eqref{eqn:aux21} changes into
$\lambda_2(u(h) + u(x_0)) \geqsl \lambda_1(x_0 + h)$ which is fulfilled because $u$ is isometric. So, we may assume that
$w_1(x_0) \neq w_2(v(x_0))$ and then we have to show that $A\epsi + \lambda_2(u(h) + v(x_0)) \geqsl \lambda_1(x_0 + h)$.
But $v$ is $\epsi$-almost isometric, hence $\lambda_2(v(x_0 + h)) \geqsl (1 - \epsi) \lambda_1(x_0 + h)$. So, thanks
to \eqref{eqn:aux19}, we obtain
\begin{multline*}
A\epsi + \lambda_2(u(h) + v(x_0)) \geqsl \epsi \lambda_1(x_0 + h) + \epsi + \lambda_2(v(x_0) + v(h))\\
- \lambda_2(u(h) - v(h)) \geqsl \lambda_1(x_0 + h).
\end{multline*}
When $j=2$, we argue in a similar way. In that case \eqref{eqn:aux20a} gives $h := x_1 \in D_0$ and $x - x_2 = u(h)$.
After substitution $x_2 := -v(x_0)$ the inequality \eqref{eqn:aux20} changes into $\lambda_1(h - x_0) + A\epsi
\delta_D(w_1(x_0) - (w_2 \circ v)(x_0)) \geqsl \lambda_2(u(h) - v(x_0))$. As before, the situation when $w_1(x_0) =
w_2(v(x_0))$ is simple. Otherwise we have to prove that $A\epsi + \lambda_1(h - x_0) \geqsl \lambda_2(u(h) - v(x_0))$.
We have
\begin{multline*}
\lambda_2(u(h) - v(x_0)) \leqsl \lambda_2(u(h) - v(h)) + \lambda_2(v(h) - v(x_0))\\ \leqsl \epsi + (1+\epsi)
\lambda_1(h - x_0) \leqsl A\epsi + \lambda_1(h - x_0).
\end{multline*}
To end the proof, replace $\lambda$ by $\lambda \wedge r$ to guarantee that $(D,+,\lambda) \in \Gg_r(N)$.
\end{proof}

\begin{lem}{A3}
Let $(D_j,+,\lambda_j) \in \Gg_r(N)$ ($j=1,2$), $E_1$ and $E_2$ be subgroups of $D_1$ and $\varphi_j\dd E_j \to D_2$
($j=1,2$) be isometric group homomorphisms such that for all $(x_1,x_2) \in E_1 \times E_2$,
\begin{equation}\label{eqn:aux22}
|\lambda_2(\varphi_1(x_1) - \varphi_2(x_2)) - \lambda_1(x_1 - x_2)| \leqsl \epsi
\end{equation}
(where $\epsi > 0$). Then there exist a valued group $(D,+,\lambda) \in \Gg_r(N)$ and isometric group homomorphisms
$\psi_j\dd D_j \to D$ ($j=1,2$) such that $\psi_2 \circ \varphi_1 = \psi_1\bigr|_{E_1}$ and
\begin{equation}\label{eqn:aux23}
\|\psi_1\bigr|_{E_2} - \psi_2 \circ \varphi_2\|_{\infty} \leqsl \epsi.
\end{equation}
\end{lem}
\begin{proof}
Let $\tilde{D}_0 = \{(x,-\varphi_1(x))\dd\ x \in E_1\} \subset D_1 \times D_2$ and $\tilde{D}$, $\pi$, $\tilde{\psi}_1$ and
$\tilde{\psi}_2$ be as in \eqref{eqn:pre}. Define $\tilde{\lambda}\dd \tilde{D} \to \RRR_+$ by
\begin{multline*}
\tilde{\lambda}(z) = \inf\{\lambda_1(x_1 - x_2) + \epsi \delta_{\tilde{D}}(\tilde{\psi}_1(x_2) - (\tilde{\psi}_2 \circ
\varphi_2)(x_2)) + \lambda_2(y + \varphi_2(x_2))\dd\\ x_1 \in D_1,\ x_2 \in E_2,\ y \in D_2,\ z = \tilde{\psi}_1(x_1)
+ \tilde{\psi}_2(y)\}.
\end{multline*}
As usual, $\tilde{\lambda}$ is a semivalue on $\tilde{D}$. We see that $\tilde{\lambda}(\tilde{\psi}_1(x_2)
- (\tilde{\psi}_2 \circ \varphi_2)(x_2)) \leqsl \epsi$ for $x_2 \in E_2$ (this corresponds to \eqref{eqn:aux23}) and
$\tilde{\lambda}(\tilde{\psi}_j(x)) \leqsl \lambda_j(x)$ for $x \in D_j$ ($j=1,2$). We want to show that in fact
$\tilde{\lambda}(\tilde{\psi}_j(z)) = \lambda_j(z)$ for $z \in D_j$, which is equivalent to
\begin{equation}\label{eqn:aux24}
\lambda_1(x - h) + \epsi \delta_{\tilde{D}}(\tilde{\psi}_1(h) - (\tilde{\psi}_2 \circ \varphi_2)(h))
+ \lambda_2(y + \varphi_2(h)) \geqsl \lambda_j(z)
\end{equation}
provided $x \in D_1$, $h \in E_2$, $y \in D_2$ and
\begin{equation}\label{eqn:aux25}
\tilde{\psi}_j(z) = \tilde{\psi}_1(x) + \tilde{\psi}_2(y).
\end{equation}
First assume $j=1$. We infer from \eqref{eqn:aux25} that $k := x - z \in E_1$ and $y = -\varphi_1(k)$. Then
\eqref{eqn:aux24} is equivalent to
\begin{equation}\label{eqn:aux26}
\epsi \delta_{\tilde{D}}(\tilde{\psi}_1(h) - (\tilde{\psi}_2 \circ \varphi_2)(h)) + \lambda_2(\varphi_2(h) - \varphi_1(k))
\geqsl \lambda_1(h - k).
\end{equation}
When $\tilde{\psi}_1(h) = (\tilde{\psi}_2 \circ \varphi_2)(h)$, $\varphi_2(h) = \varphi_1(h)$ and \eqref{eqn:aux26} is
satisfied, thanks to the isometricity of $\varphi_1$. Otherwise, \eqref{eqn:aux26} follows from \eqref{eqn:aux22}.\par
The case of $j=2$ is similar and is left for the reader.\par
To finish the proof, define $D$, $\lambda$, $\psi_1$ and $\psi_2$ in exactly the same way as at the end of the proof
of \LEM{A1}.
\end{proof}

\SECT{Proof of \THM{main1}}

For the duration of this section $r \in \{1,\infty\}$, $N \in \ZZZ_+ \setminus \{1\}$ and a set $Q \subset \RRR$ such as
in \eqref{eqn:Q} are fixed. Note that it is not assumed that $Q$ is countable. However, in some results its countability is
necessary and then we add this assumption to their statements. For simplicity, we put the following

\begin{dfn}{Qg}
Suppose $Q$ is countable. A valued group $(G,+,p)$ is said to be a \textit{$Q$-group} iff $G$ is finite or a countable
group of class $\OOo_{fin}$ and $p(G) \subset Q$.
\end{dfn}

As a special case of the general technique of Fra\"{\i}ss\'{e} limits, we get

\begin{thm}{QG}
Suppose $Q$ is countable. There is a unique (up to isometric group isomorphism) $Q$-group $Q\GGG_r(N) \in \Gg_r(N)$ with
the following property. Whenever $(H,+,q) \in \Gg_r(N)$ is a finite $Q$-group and $K$ is a subgroup of $H$, every isometric
group homomorphism of $K$ into $Q\GGG_r(N)$ is extendable to an isometric group homomorphism of $H$ into $Q\GGG_r(N)$. 
\end{thm}
\begin{proof}
The uniqueness follows from the back-and-forth method. The existence may be provide in a standard way. There are only
countably many (up to isometric group isomorphism) finite $Q$-groups belonging to $\Gg_r(N)$ and thus there is a sequence
$(H_n,+,q_n)_{n=1}^{\infty} \subset \Gg_r(N)$ of finite $Q$-groups in which every such group appears infinitely many times.
Using repeatedly \LEM{A1}, inductively define a sequence of finite $Q$-groups $(G_n,+,p_n) \in \Gg_r(N)$, starting with
$G_0 = \{0\}$, such that $(G_n,+,p_n) \subset (G_{n+1},+,p_{n+1})$ and every isometric group homomorphism of a subgroup
of $H_n$ is extendable to an isometric group homomorphism of $H_n$ into $G_{n+1}$. Finally, put $(Q\GGG_r(N),+,p) =
\bigcup_{n=1}^{\infty} (G_n,+,p_n)$. The details are left for the reader.
\end{proof}

It is clear that the completion of a member of $\Gg_r(N)$ is of the same class. In what follows, we fix the following
situation. We assume that $(G,+,p) \in \Gg_r(N)$ is complete and for some \textbf{dense} subgroup $G_0$ of $G$
the following conditions are fulfilled:
\begin{enumerate}[\upshape(QG1)]
\item whenever $(H,+,q) \in \Gg_r(N)$ is a finite group with $Q$-valued value, $K$ is its subgroup and $\varphi\dd K \to
   G_0$ is an isometric group homomorphism, then for every $\epsi \in (0,1)$ there is an $\epsi$-almost isometric group
   homomorphism $\psi\dd H \to G_0$ such that $\|\psi\bigr|_K - \varphi\|_{\infty} \leqsl \epsi$,
\item $p(G_0) \subset Q$ and the set of finite rank elements of $G_0$ is dense in $G_0$.
\end{enumerate}
(We underline that it is not assumed here that $Q$ is countable.) Our aim is to show that
\begin{itemize}
\item[(UEP)] whenever $(H,+,q) \in \Gg_r(N)$ is a finite valued group, $K$ is its subgroup and $\varphi\dd K \to G$ is
   an isometric group homomorphism, there is an isometric group homomorphism $\psi\dd H \to G$ which extends $\varphi$.
\end{itemize}
The proof of (UEP) is preceded by a few lemmas.

\begin{lem}{QG1}
Let $a \in G$ be of finite rank $k \geqsl 2$. For every $\epsi > 0$ there is $b \in G_0$ such that $\rank(b) = k$ and
$p(ja - jb) \leqsl \epsi$ for each $j \in \ZZZ$.
\end{lem}
\begin{proof}
Let $\delta \in (0,1)$ be such that $\delta \leqsl \frac{\epsi}{6k}$. By (QG2), there is a finite rank element $c \in G_0$
for which $p(a - c) \leqsl \delta$. Let $H = \grp{a,c}$ and $K = \grp{c} \subset H$. Notice that $p(K) \subset Q$.
We conclude from \LEM{valQ} that there is a $Q$-valued value $q$ on $H$ which extends $p\bigr|_K$, is bounded by $r$ and
satisfies $\|p\bigr|_H - q\|_{\infty} \leqsl \delta$. We see that $(H,+,q) \in \Gg_r(N)$. So, by (QG1) applied to $\id\dd
K \to G_0$, there is a $\delta$-almost isometric (with respect to $q$) group homomorphism $\varphi\dd H \to G_0$ such that
$p(x - \varphi(x)) \leqsl \delta$ for $x \in K$. Put $b = \varphi(a)$. Then $\rank(b) = \rank(a)$ and
\begin{multline*}
p(a - b) \leqsl p(a - c) + p(c - \varphi(c)) + p(\varphi(c) - \varphi(a))\\ \leqsl 2\delta + (1 + \delta) q(c - a)
\leqsl 2\delta + 2(p(c - a) + \delta) \leqsl 6\delta \leqsl \frac{\epsi}{k},
\end{multline*}
so $p(ja - jb) \leqsl \epsi$ for $j = 0,1,\ldots,k-1$ and we are done.
\end{proof}

\begin{lem}{QG2}
Let $H$ be a finite subgroup of $G$. For each $\epsi > 0$ there is a group homomorphism $\varphi\dd H \to G_0$ with trivial
kernel such that $p(\varphi(x) - x) \leqsl \epsi$ for each $x \in H$.
\end{lem}
\begin{proof}
We assume $H$ is nontrivial. Since every finite Abelian group is group isomorphic to a direct product of cyclic groups,
there is $s \geqsl 1$ and $a_1,\ldots,a_s$ such that the function $\Phi\dd \grp{a_1} \times \ldots \times \grp{a_s} \ni
(x_1,\ldots,x_s) \mapsto \sum_{j=1}^s x_j \in H$ is a group isomorphism. Let $\mu = \min\{p(x)\dd\ x \in H \setminus
\{0\}\} > 0$. We assume $\epsi < \mu$. By \LEM{QG1}, there are $b_1,\ldots,b_s \in G_0$ with $\rank(b_j) = \rank(a_j)$
and $p(lb_j - la_j) \leqsl \epsi / s$ for $j = 1,\ldots,s$ and $l \in \ZZZ$. We infer from this that there are group
isomorphisms $\Lambda_j\dd \grp{a_j} \to \grp{b_j}$ such that $\Lambda_j(a_j) = b_j$. Let $$\Lambda = \Lambda_1 \times
\ldots \times \Lambda_s\dd \grp{a_1} \times \ldots \times \grp{a_s} \to \grp{b_1} \times \ldots \times \grp{b_s}$$
($\Lambda(x_1,\ldots,x_s) = (\Lambda_1(x_1),\ldots,\Lambda_s(x_s))$). Further, if $l_1,\ldots,l_s \in \ZZZ$, then
$p(\sum_{j=1}^s l_j a_j - \sum_{j=1}^s l_j b_j) \leqsl \sum_{j=1}^s p(l_j a_j - l_j b_j) \leqsl \epsi < \mu$ and thus
$\sum_{j=1}^s l_j a_j = 0$ iff $\sum_{j=1}^s l_j b_j = 0$. This yields that the group homomorphism $$\Psi\dd \grp{b_1}
\times \ldots \times \grp{b_s} \ni (y_1,\ldots,y_s) \mapsto \sum_{j=1}^s y_j \in G_0$$ has trivial kernel. Now it suffices
to put $\varphi = \Psi \circ \Lambda \circ \Phi^{-1}$.
\end{proof}

\begin{lem}{QG3}
Let $(H,+,q) \in \Gg_r(N)$ be a finite group, $K$ its subgroup and let $\varphi\dd K \to G$ be an isometric group
homomorphism. Then for every $\epsi \in (0,1)$ there is an $\epsi$-almost isometric group homomorphism $\psi\dd H \to G$
such that $\|\psi\bigr|_K - \varphi\|_{\infty} \leqsl \epsi$.
\end{lem}
\begin{proof}
Again, we assume $H$ is nontrivial. Let $\mu = \min\{q(h)\dd\ h \in H \setminus \{0\}\}$. Take $\delta$ such that
\begin{equation}\label{eqn:auxd}
\delta \in (0,\frac12), \quad \delta < c, \quad (1 + 2\delta)^2 \leqsl 1 + \epsi, \quad (1 - 2\delta)^2 \geqsl 1 - \epsi.
\end{equation}
By \LEM{QG2}, there is a group homomorphism $\kappa\dd \varphi(K) \to G_0$ with trivial kernel such that $p(\kappa(x) - x)
\leqsl \delta^2$. Put $\psi_0 = \kappa \circ \varphi\dd K \to G_0$. Then $\psi_0$ has trivial kernel and
\begin{equation}\label{eqn:aux27}
\|\psi_0 - \varphi\|_{\infty} \leqsl \delta^2.
\end{equation}
Let $\lambda_0\dd K \ni x \mapsto p(\psi_0(x)) \in \RRR_+$. Observe that $\lambda_0$ is a value bounded by $r$ and
$\lambda_0(K) \subset Q$ (see (QG2)). Moreover, for $x \in K \setminus \{0\}$ we have (thanks to \eqref{eqn:auxd} and
\eqref{eqn:aux27}):
$$
q(x) = p(\varphi(x)) \leqsl p(\psi_0(x)) + \delta^2 \leqsl \lambda_0(x) + \delta q(x)
$$
and
$$
\lambda_0(x) \leqsl p(\varphi(x)) + \delta^2 \leqsl q(x) + \delta q(x) = (1 + \delta) q(x).
$$
The above estimations gives $q\bigr|_K \leqsl \frac{1}{1 - \delta} \lambda_0$ and $\lambda_0 \leqsl (1 + \delta)
q\bigr|_K$. Now by \LEM{valom} (with $\omega(t) = (1 + \delta) t$, $\varrho(t) = \frac{t}{1 - \delta}$ and $\tau(t) = t$),
there is a value $\lambda_1$ on $H$ which extends $\lambda_0$, is bounded by $r$ and satisfies
\begin{equation}\label{eqn:aux28}
(1 - \delta) q \leqsl \lambda_1 \leqsl (1 + \delta) q.
\end{equation}
Further, since $\lambda_1(K) = \lambda_0(K) \subset Q$, we infer from \LEM{valQ} that there is a value $\lambda$ on $H$
bounded by $r$, extending $\lambda_1\bigr|_K$ and satisfying $\lambda(H) \subset Q$ and
\begin{equation}\label{eqn:aux29}
\|\lambda_1 - \lambda\|_{\infty} \leqsl \delta^2.
\end{equation}
Note that $\psi_0\dd K \to G_0$ is isometric with respect to $\lambda$ and therefore, by (QG1), there is a $\delta$-almost
isometric group homomorphism $\psi\dd (H,+,\lambda) \to (G,+,p)$ such that $\|\psi\bigr|_K - \psi_0\|_{\infty} \leqsl
\delta$. The latter inequality combined with \eqref{eqn:auxd} and \eqref{eqn:aux27} gives $\|\psi\bigr|_K
- \varphi\|_{\infty} \leqsl \epsi$. So, we only need to check that $\psi$, as a group homomorphism of $(H,+,q)$ into
$(G,+,p)$, is $\epsi$-almost isometric. For $h \in H \setminus \{0\}$ we have, thanks to \eqref{eqn:auxd},
\eqref{eqn:aux28} and \eqref{eqn:aux29}:
\begin{multline*}
p(\psi(h)) \leqsl (1 + \delta) \lambda(h) \leqsl (1 + \delta)(\lambda_1(h) + \delta^2)\\ \leqsl (1 + \delta)[(1 + \delta)
q(h) + \delta q(h)] \leqsl (1 + 2\delta)^2 q(h) \leqsl (1 + \epsi) q(h)
\end{multline*}
and
\begin{multline*}
p(\psi(h)) \geqsl (1 - \delta) \lambda(h) \geqsl (1 - \delta)(\lambda_1(h) - \delta^2)\\ \geqsl (1 - \delta)[(1 - \delta)
q(h) - \delta q(h)) \geqsl (1 - 2\delta)^2 q(h) \geqsl (1 - \epsi) q(h)
\end{multline*}
which finishes the proof.
\end{proof}

\begin{rem}{interesting}
From the beginning of the section to this moment we have never used the assumption that $G$ is complete. So, the assertion
of \LEM{QG3} is fulfilled for every group $G$ which is `between' $G_0$ and its completion. In particular, \LEM{QG3} holds
true for $G = G_0$.
\end{rem}

\begin{lem}{QG4}
Let $(H,+,q) \in \Gg_r(N)$ be a finite group and $K$ its subgroup. Further, let $\varphi\dd K \to G$ and $\psi\dd H \to G$
be, respectively, an isometric and an $\epsi$-almost isometric group homomorphism (where $\epsi \in (0,1)$) such that
$\|\psi\bigr|_K - \varphi\|_{\infty} \leqsl \epsi$. For every $\delta \in (0,\epsi)$ there is a $\delta$-almost isometric
group homomorphism $\psi'\dd H \to G$ such that $\|\psi'\bigr|_K - \varphi\|_{\infty} \leqsl \delta$ and
$\|\psi - \psi'\|_{\infty} \leqsl C \epsi$ where $C = 3 + 2 \diam (H,q)$.
\end{lem}
\begin{proof}
Let $L = \varphi(K) + \psi(H)$ and $p' = p\bigr|_L$. Then $(L,+,p') \in \Gg_r(N)$ and $L$ is finite. By \LEM{A2}, there are
a finite valued group $(D,+,\lambda) \in \Gg_r(N)$ and isometric group homomorphisms $w_H\dd H \to D$ and $w_L\dd L \to D$
such that
\begin{equation}\label{eqn:aux30}
w_L \circ \varphi = w_H\bigr|_K \qquad \textup{and} \qquad \|w_H - w_L \circ \psi\|_{\infty} \leqsl A\epsi
\end{equation}
where $A = 1 + \diam (H,q)$. Put $D_0 = w_L(L)$. Observe that $w_L^{-1}\dd D_0 \to L \subset G$ is isometric. Hence,
by \LEM{QG3}, there is a $\delta$-almost isometric group homomorphism $v\dd D \to G$ such that
\begin{equation}\label{eqn:aux31}
\|v\bigr|_{D_0} - w_L^{-1}\|_{\infty} \leqsl \delta.
\end{equation}
Then $\psi' = v \circ w_H\dd H \to G$ is $\delta$-almost isometric. Now if $x \in K$, then by \eqref{eqn:aux30},
$w_H(x) = (w_L \circ \varphi)(x)$ and thus
$$
p(\varphi(x) - \psi'(x)) = p(w_L^{-1}((w_L \circ \varphi)(x)) - v((w_L \circ \varphi)(x))) \leqsl \delta,
$$
thanks to \eqref{eqn:aux31}. Finally, if $h \in H$, then $\psi(h) \in L$ and (by \eqref{eqn:aux30} and \eqref{eqn:aux31}):
\begin{multline*}
p(\psi(h) - \psi'(h)) \leqsl p(w_L^{-1}((w_L \circ \psi)(h)) - v((w_L \circ \psi)(h)))\\
+ p(v((w_L \circ \psi)(h)) - v(w_H(h))) \leqsl \|w_L^{-1} - v\bigr|_{D_0}\|_{\infty}\\
+ (1 + \delta) p((w_L \circ \psi)(h) - w_H(h)) \leqsl \delta + 2 \|w_L \circ \psi - w_H\|_{\infty}\\\leqsl \epsi + 2A \epsi
= C \epsi.
\end{multline*}
\end{proof}

Now we are ready to prove

\begin{thm}{UEP}
The group $G$ satisfies \textup{(UEP)}.
\end{thm}
\begin{proof}
Let $(H,+,q)$, $K$ and $\varphi$ be as in (UEP). Put $C = 3 + 2 \diam (H,q)$ and $\epsi_n = 2^{-n}$ ($n \geqsl 1$).
By \LEM{QG3}, there is an $\epsi_1$-almost isometric group homomorphism $\psi_1\dd H \to G$ such that $\|\psi_1\bigr|_K
- \varphi\|_{\infty} \leqsl \epsi_1$. Suppose we have $\psi_{n-1}$ for some $n \geqsl 2$. Applying \LEM{QG4}, we obtain
$\psi_n\dd H \to G$ such that
\begin{enumerate}[(1$_n$)]
\item $\psi_n$ is an $\epsi_n$-almost isometric group homomorphism,
\item $\|\psi_n\bigr|_K - \varphi\|_{\infty} \leqsl \epsi_n$,
\item $\|\psi_n - \psi_{n-1}\|_{\infty} \leqsl C \epsi_{n-1}$.
\end{enumerate}
The condition (3$_n$) ensures that the sequence $(\psi_n(h))_{n=1}^{\infty}$ converges for every $h \in H$ and thus we may
define a group homomorphism $\psi\dd H \to G$ by $\psi(h) = \lim_{n\to\infty} \psi_n(h)$. Then (1$_n$) yields that $\psi$
is isometric and (2$_n$) gives $\psi\bigr|_K = \varphi$.
\end{proof}

\textit{Proof of \THM{main1}}. Let $G$ be the completion of $\QQQ\GGG_r(N)$ (cf. \THM{QG}).
By \THM{UEP}, $G$ satisfies (UEP). This gives the existence of $\GGG_r(N)$. The uniqueness follows again from \THM{UEP},
applied to the situation when $Q = \RRR$, (G3) and the back-and-forth method.\qed
\vspace{0.3cm}

Note that we have proved that $\GGG_r(N)$ fulfills (UEP). This fact will be used in the next section. The results of this
section gives more than just the assertion of \THM{main1}. Namely,

\begin{pro}{completion}
In each of the following cases the completion of a valued group $(G,+,p)$ is isometrically group isomorphic to $\GGG_r(N)$.
\begin{enumerate}[\upshape(A)]
\item $G = Q\GGG_r(N)$ where $Q$ is a countable set as in \eqref{eqn:Q} (cf. \THM{QG}).
\item $(G,+,p) \in \Gg_r(N)$ and $G$ satisfies conditions \textup{(G2)} and \textup{(G3)} of \THM{main1} (with $\GGG_r(N)$
   replaced by $G$).
\item $(G,+,p)$ satisfies conditions \textup{(QG1)} and \textup{(QG2)} (with $G_0$ replaced by $G$) for some set $Q$
   as in \eqref{eqn:Q}.
\end{enumerate}
\end{pro}

\SECT{Proof of \THM{main2}}

As in the previous section, we fix $r \in \{1,\infty\}$ and $N \in \ZZZ_+ \setminus \{1\}$. Our first aim of this part
is to show that
\begin{itemize}
\item[(CEP)] Whenever $(L,+,q) \in \Gg_r(N)$, $K$ and $H$ are, respectively, a compact and a finite subgroup of $L$, and
$\varphi\dd K \to \GGG_r(N)$ is an isometric group homomorphism, then there is an isometric group homomorphism $\psi\dd
K + H \to \GGG_r(N)$ which extends $\varphi$.
\end{itemize}
Similarly as in the previous section, the proof of (CEP) is preceded by a few auxiliary lemmas. In some of them we use
the Hausdorff distance, which is denoted by us by $\dist_q(A,B)$ if only $A$ and $B$ are two compact nonempty subsets
of a valued group with value $q$.

\begin{lem}{K1}
Let $a \in \GGG_r(0)$ be such that the closure $K$ of $\grp{a}$ is compact. Then for every $\epsi > 0$ there is a finite
rank element $b$ of $\GGG_r(0)$ such that $\dist_p(\grp{b},K) \leqsl \epsi$ where $p$ is the value of $\GGG_r(0)$.
\end{lem}
\begin{proof}
We asssume $\rank(a) = \infty$. Since the elements of the sequence $(na)_{n=1}^{\infty}$ form a dense subset of $K$, there
is $m \geqsl 2$ such that the set $\{0,a,\ldots,(m-1)a\}$ is an ($\epsi / 4$)-net for $K$ and
\begin{equation}\label{eqn:aux32}
p(ma) \leqsl \frac{\epsi}{8}.
\end{equation}
By (G3), there is a finite rank element $c \in \GGG_r(0)$ for which
\begin{equation}\label{eqn:aux33}
p(a - c) \leqsl \frac{\epsi}{4m}.
\end{equation}
Put $H = \grp{c}$ and let $H' = \grp{c'}$ be a cyclic group of rank $m$. For each $k \in \ZZZ$ let $s(k) \in
\{0,1,\ldots,m-1\}$ be such that $m | k - s(k)$. Define $q_0\dd H \times H' \to \RRR_+$ by $q_0(x,lc') = p(x + s(l)a)
+ \delta_{H'}(lc')\epsi / 8$ ($l \in \ZZZ$). Observe that $q_0$ is well defined; $q_0(x,y) = 0$ iff $x = 0$ and $y = 0$;
and $q_0(x,0) = p(x)$. We claim that $q_0$ satisfies the triangle inequality, that is, $q_0(x + x',(l + l')c') \leqsl
q_0(x,lc') + q_0(x',l'c')$. Indeed, if $s(l) + s(l') < m$, then $s(l + l') = s(l) + s(l')$ and then the latter inequality
is immediate. And when $s(l) + s(l') \geqsl m$, we have $s(l) >0$, $s(l') > 0$ and $s(l + l') = s(l) + s(l') - m$ and
hence, by \eqref{eqn:aux32},
\begin{multline*}
q_0(x,lc') + q_0(x',l'c') \geqsl p(x + x' + (s(l) + s(l'))a) + \frac{\epsi}{8} + \frac{\epsi}{8} \geqsl\\
p(x + x' + s(l+l')a - ma) + p(ma) + \frac{\epsi}{8}\delta_{H'}((l+l')c') \geqsl q_0(x + x',(l + l')c').
\end{multline*}
Now let $q\dd H \times H' \to \RRR_+$ be given by $$q(x,y) = [\frac12(q_0(x,y) + q_0(-x,-y))] \wedge r.$$ It follows from
the properties of $q_0$ that $q$ is a value and $q(x,0) = p(x)$ for $x \in H$. Moreover, for $j \in \{0,1,\ldots,m-1\}$
one has
\begin{equation}\label{eqn:aux34}
q(jc,-jc') \leqsl \frac{\epsi}{2}.
\end{equation}
Indeed, we may assume that $j \neq 0$ and then $s(j) = j$ and $s(-j) = m - j$. So, by \eqref{eqn:aux32} and
\eqref{eqn:aux33}, $q_0(jc,-jc') \leqsl p(j(c - a)) + p(ma) + \frac{\epsi}{8} \leqsl \frac{\epsi}{2}$ and $q_0(-jc,jc') =
p(j(c - a)) + \frac{\epsi}{8} \leqsl \frac{\epsi}{2}$. Further, since the function $H \times \{0\} \ni (h,0) \mapsto h
\in \GGG_r(0)$, we conclude from (UEP) that there exists an isometric group homomorphism $\psi\dd H \times H' \to
\GGG_r(0)$ with $\psi(h,0) = h$ for any $h \in H$. Put $b = \psi(0,c')$. Then $\grp{b} = \{0,b,\ldots,(m-1)b\}$. Now if
$j \in \{0,1,\ldots,m-1\}$, we get (see \eqref{eqn:aux33} and \eqref{eqn:aux34}):
\begin{multline*}
p(jb - ja) \leqsl p(jc - jb) + p(jc - ja) \leqsl p(\psi(jc,0) - j\psi(0,c')) + j p(c - a)\\ \leqsl p(\psi(jc,-jc'))
+ \frac{\epsi}{4} = q(jc,-jc') + \frac{\epsi}{4} \leqsl \frac34\epsi.
\end{multline*}
Conversely, if $x \in K$, there is $j \in \{0,1,\ldots,m-1\}$ such that $p(x - ja) \leqsl \epsi / 4$ and then
$p(x - jb) \leqsl p(x - ja) + p(ja - jb) \leqsl \epsi$ which finally gives $\dist_p(\grp{b},K) \leqsl \epsi$.
\end{proof}

From now to the end of the \textbf{next} section $p$ denotes the value of $\GGG_r(N)$.

\begin{lem}{K2}
Let $K$ be a compact subgroup of $\GGG_r(N)$. For each $\epsi > 0$ there is a finite subgroup $H$ of $\GGG_r(N)$ such that
$\dist_p(H,K) \leqsl \epsi$.
\end{lem}
\begin{proof}
There is $s \geqsl 2$ and $a_1,\ldots,a_s \in K$ such that $\{a_1,\ldots,a_s\}$ is an $(\epsi / 2)$-net for $K$.
If $N \neq 0$, put $H = \grp{a_1,\ldots,a_s}$ and notice that $\dist_p(H,K) \leqsl \epsi / 2$ since $\{a_1,\ldots,a_s\}
\subset H \subset K$. We now assume $N = 0$. By \LEM{K1}, there are finite rank elements $b_1,\ldots,b_s \in \GGG_r(0)$
with $\dist_p(\grp{b_j},K_j) \leqsl \epsi / s$ where $K_j$ is the closure of $\grp{a_j}$ ($j=1,\ldots,s$). Let $H =
\grp{b_1,\ldots,b_s}$. For every $x \in H$ there are $x_j \in \grp{b_j}$ ($j=1,\ldots,s$) for which $x = \sum_{j=1}^s x_j$.
Then for every $j$ one can find $y_j \in K_j$ such that $p(x_j - y_j) \leqsl \epsi / s$ and hence $p(x - y) \leqsl \epsi$
for $y = \sum_{j=1}^s y_j \in K$. Conversely, if $y \in K$, there is $j$ such that $p(x - a_j) \leqsl \epsi / 2$, and there
is $w \in \grp{b_j} \subset H$ with $p(a_j - w) \leqsl \epsi / s \leqsl \epsi / 2$ which yields $p(x - w) \leqsl \epsi$
and we are done.
\end{proof}

For need of the nearest three results let $(L,+,q)$, $K$, $H$ and $\varphi$ be as in (CEP).

\begin{lem}{K3}
For every $\epsi > 0$ there is an isometric group homomorphism $\psi\dd H \to \GGG_r(N)$ such that for any $h \in H$
and $k \in K$,
\begin{equation}\label{eqn:h-k}
|p(\psi(h) - \varphi(k)) - q(h - k)| \leqsl \epsi.
\end{equation}
\end{lem}
\begin{proof}
Let $\tilde{K} = \varphi(K)$. By \LEM{K2}, there is a finite subgroup $F$ of $\GGG_r(N)$ such that
\begin{equation}\label{eqn:auxF}
\dist_p(F,\tilde{K}) \leqsl \frac{\epsi}{2}.
\end{equation}
It follows from \LEM{A1}, applied to $\varphi_1 = \varphi\dd K \to \tilde{K} + F$ and $\varphi_2 =
\id\dd K \to L$, that there exist a valued group $(D,+,\lambda) \in \Gg_r(N)$ and isometric group homomorphisms $\Phi\dd L
\to D$ and $\Psi\dd \tilde{K} + F \to D$ for which $\Psi \circ \varphi = \Phi\bigr|_K$. Further, the group $\Phi(H) +
\Psi(F)$ is a finite subgroup of $D$ and the group homomorphism $\Psi^{-1}\bigr|_{\Psi(F)}\dd \Psi(F) \to \GGG_r(N)$ is
isometric. So, thanks to (UEP), there is an isometric group homomorphism $\tau\dd \Phi(H) + \Psi(F) \to \GGG_r(N)$ which
extends $\Psi^{-1}\bigr|_{\Psi(F)}$. Put $\psi = \tau \circ \Phi\bigr|_H$. Fix $h \in H$ and $k \in K$. Take $f \in F$
such that $p(f - \varphi(k)) \leqsl \epsi / 2$ (see \eqref{eqn:auxF}). Then $p(\psi(h) - \varphi(k)) \leqsl
p(\tau(\Phi(h)) - f) + \epsi / 2$ and
\begin{multline*}
p(\tau(\Phi(h)) - f) = p(\tau(\Phi(h)) - \tau(\Psi(f))) = \lambda(\Phi(h) - \Psi(f))\\ \leqsl \lambda(\Phi(h) - \Phi(k))
+ \lambda(\Phi(k) - \Psi(f)) = q(h - k)\\ + \lambda(\Psi(\varphi(k)) - \Psi(f)) = q(h - k) + p(\varphi(k) - f) \leqsl
q(h - k) + \frac{\epsi}{2}
\end{multline*}
which shows that $p(\psi(h) - \varphi(k)) \leqsl q(h - k) + \epsi$. Conversely,
\begin{multline*}
q(h - k) = \lambda(\Phi(h) - \Phi(k)) \leqsl \lambda(\Phi(h) - \Psi(f)) + \lambda(\Psi(f) - \Phi(k))\\ = p(\tau(\Phi(h))
- \tau(\Psi(f))) + \lambda(\Psi(f) - \Psi(\varphi(k)))\\= p(\psi(h) - f) + p(f - \varphi(k)) \leqsl p(\psi(h) - \varphi(k))
+ 2 p(f - \varphi(k))
\end{multline*}
which finishes the proof of \eqref{eqn:h-k}, because $2 p(f - \varphi(k)) \leqsl \epsi$.
\end{proof}

\begin{lem}{K4}
Let $\psi\dd H \to \GGG_r(N)$ be an isometric group homomorphism such that \eqref{eqn:h-k} is fulfilled for any $h \in H$
and $k \in K$. For each $\delta > 0$ there exists an isometric group homomorphism $\psi'\dd H \to \GGG_r(N)$ for which
$\|\psi - \psi'\|_{\infty} \leqsl \epsi + \delta$ and for all $h \in H$ and $k \in K$,
\begin{equation}\label{eqn:h-k'}
|p(\psi'(h) - \varphi(k)) - q(h - k)| \leqsl \delta.
\end{equation}
\end{lem}
\begin{proof}
By \LEM{A3}, there are a valued group $(D,+,\lambda) \in \Gg_r(N)$ and isometric group homomorphisms $w_L\dd L \to D$
and $w_G\dd \varphi(K) + \psi(H) \to D$ such that $w_L\bigr|_K = w_G \circ \varphi$ and
\begin{equation}\label{eqn:aux39}
\|w_L\bigr|_H - w_G \circ \psi\|_{\infty} \leqsl \epsi.
\end{equation}
Then $Z := w_G(\varphi(K) + \psi(H))$ and $F := w_L(H) + w_G(\psi(H))$ are, respectively, a compact and a finite subgroup
of $D$, and $w_G^{-1}\dd Z \to \GGG_r(N)$ is isometric. So, thanks to \LEM{K3}, there is an isometric group homomorphism
$\xi\dd F \to \GGG_r(N)$ such that for any $z \in Z$ and $f \in F$,
\begin{equation}\label{eqn:aux40}
|p(w_G^{-1}(z) - \xi(f)) - \lambda(z - f)| \leqsl \delta.
\end{equation}
Put $\psi' = \xi \circ w_L\bigr|_H\dd H \to \GGG_r(N)$. Now fix $h \in H$ and $k \in K$ and put $z = w_G(\varphi(k)) \in Z$
and $f = w_L(h) \in F$. Then we have $p(\psi'(h) - \varphi(k)) = p(\xi(f) - w_G^{-1}(z))$ and (since $w_G \circ \varphi
= w_L\bigr|_K$)
$$
q(h - k) = \lambda(w_L(h) - w_L(k)) = \lambda(w_L(h) - w_G(\varphi(k))) = \lambda(f - z).
$$
Hence \eqref{eqn:h-k'} follows from \eqref{eqn:aux40}. It suffices to show that $\|\psi - \psi'\|_{\infty} \leqsl \epsi +
\delta$. For $h \in H$ we get
$$
p(\psi(h) - \psi'(h)) \leqsl p(\psi(h) - \xi((w_G \circ \psi)(h))) + p(\xi((w_G \circ \psi)(h)) - \xi(w_L(h)))
$$
and, by \eqref{eqn:aux39}, $p(\xi((w_G \circ \psi)(h)) - \xi(w_L(h))) = \lambda((w_G \circ \psi)(h) - w_L(h)) \leqsl
\epsi$. So, it remains to check that $p(\psi(h) - \xi((w_G \circ \psi)(h))) \leqsl \delta$. But this follows from
\eqref{eqn:aux40} for $z = f = (w_G \circ \psi)(h) \in Z \cap F$.
\end{proof}

Finally, we have

\begin{thm}{CEP}
The assertion of \textup{(CEP)} is satisfied.
\end{thm}
\begin{proof}
Put $\epsi_n = \frac{1}{2^n}$. By \LEM{K3}, there is an isometric group homomorphism $\psi_1\dd H \to \GGG_r(N)$ such that
\eqref{eqn:h-k} (with $\psi$ replaced by $\psi_1$) is fulfilled for any $h \in H$ and $k \in K$. Now suppose that
$\psi_{n-1}\dd H \to \GGG_r(N)$ is constructed. We infer from \LEM{K4} that there is a group homomorphism $\psi_n\dd H \to
\GGG_r(N)$ such that
\begin{enumerate}[(1$_n$)]
\item $\psi_n$ is isometric,
\item $\|\psi_{n-1} - \psi_n\|_{\infty} \leqsl \epsi_{n-1} + \epsi_n \leqsl 2\epsi_{n-1}$,
\item $|p(\psi_n(h) - \varphi(k)) - q(h - k)| \leqsl \epsi_n$ for all $h \in H$ and $k \in K$.
\end{enumerate}
We conclude from (1$_n$) and (2$_n$) that a group homomorphism $\varphi_0\dd H \to \GGG_r(N)$ given by $\varphi_0(h) =
\lim_{n\to\infty} \psi_n(h)$ is well defined and isometric. What is more, (3$_n$) yields
\begin{equation}\label{eqn:aux41}
p(\varphi_0(h) - \varphi(k)) = q(h - k)
\end{equation}
for any $h \in H$ and $k \in K$. We easily deduce from \eqref{eqn:aux41} that the formula $\psi(h + k) = \varphi_0(h) +
\psi(k)$ (where $h \in H$ and $k \in K$) well defines an isometric group homomorphism $\psi\dd K + H \to \GGG_r(N)$ which
extends $\varphi$.
\end{proof}

\textit{Proof of \THM{main2}.} We begin with the note that the assertion of part (B) follows from (A), (G3)
and the back-and-forth method. Thus, it suffices to prove (A). First we shall show the easier point, (i). When $\varphi$
is not open, it suffices to put $\tau \equiv 0$, $\varrho = \id$ and to apply \COR{bdhom} in order to find $\omega \in
\Omega^*$. Similarly, if $\varphi$ is open as a map of $K$ onto $\varphi(K)$, the semivalue $q_0\dd K \ni x \mapsto
\dist_q(x,\ker \varphi) \in \RRR_+$ is continuous with respect to $p \circ \varphi$ and hence for $\tau_0 =
\frac{\id}{\id + 1}$ one may find (thanks to \LEM{bdom}) $\varrho_0 \in \Omega^*$ such that $\tau_0 \circ q_0 \leqsl
\varrho_0 \circ p \circ \varphi$. Further, take suitable $\omega_0 \in \Omega^*$ and apply the idea of \EXM{o-r-t} to
guarantee ($\omega$-$\rho$-$\tau$). We now pass to the main part of the theorem---to point (ii).\par
Put $\lambda_0\dd K \ni x \mapsto p(\varphi(x)) \in \RRR_+$. Then $\lambda_0$ is a semivalue on $K$ such that $\lambda_0
\leqsl (\omega \circ q)\bigr|_K$. \LEM{valom} ensures us the existence of a semivalue $\lambda$ on $H$ bounded by $r$
which extends $\lambda_0$ and satisfies suitable conditions. Put $H' = H / \ker \varphi$. Let $\pi\dd H \to H'$ be
the quotient group homomorphism, $\lambda'$ the value induced by $\lambda$ (recall that $\lambda^{-1}(\{0\}) =
\lambda_0^{-1}(\{0\}) = \ker \varphi$), $K' = \pi(K)$ and let $\varphi'\dd K' \to \GGG_r(N)$ be a group homomorphism such
that $\varphi' \circ \pi\bigr|_K = \varphi$. Notice that $(H',+,\lambda') \in \Gg_r(N)$ (thanks to \PRO{image}) and
\begin{equation}\label{eqn:aux43}
p(\varphi'(x)) = \lambda'(x)
\end{equation}
for every $x \in K'$ (because $\lambda$ extends $\lambda_0$). Now let $(\bar{H},+,\bar{\lambda}) \in \Gg_r(N)$ be
the completion of $(H',+,\lambda')$. The relation \eqref{eqn:aux43} ensures us that $\varphi'$ extends to an isometric
group homomorphism $\bar{\varphi}\dd \bar{K} \to \GGG_r(N)$ defined on a closed subgroup of $\bar{H}$. Since the closure
of $\varphi(K)$ is compact in $\Gg_r(N)$, $\bar{K}$ is compact. Enlarging, if needed, the group $\bar{H}$ (making use
of \THM{o0}) we may assume that $\bar{H}_{fin}$ is dense in $\bar{H}$. Now thanks to (CEP) and the induction argument
we see that there is an isometric group homomorphism $\bar{\psi}\dd \bar{H} \to \GGG_r(N)$ which extends $\bar{\varphi}$.
Define $\varphi_{\omega}\dd H \to \GGG_r(N)$ by $\varphi_{\omega} = \bar{\psi} \circ \pi$ and observe that $\ker \varphi =
\ker \varphi_{\omega}$ and $p(\varphi_{\omega}(h)) = \lambda(h)$ for each $h \in H$. The verification that all other
assertions are fulfilled is left for the reader.\qed
\vspace{0.3cm}

In the next section we shall show that $\GGG_r(2)$ is Urysohn as a metric space. Thus, \THM{main2} extends and strengthens
the results of \cite{pn1}.

\begin{cor}{univ}
\begin{enumerate}[\upshape(A)]
\item The group $\GGG_r(N)$ is universal for the class $\Gg_r(N)$; that is, every member of $\Gg_r(N)$ admits an isometric
   group homomorphism into $\GGG_r(N)$.
\item The groups $\GGG_1(0)$ and $\GGG_{\infty}(0)$ are topologically universal for the class of separable metrizable
   topological Abelian groups. What is more, for every $(G,+,q) \in \Gg$ the valued group $(G,+,q \wedge 1)$ (respectively
   $(G,+,q^{\alpha})$ with $0 < \alpha < 1$) admits an isometric group homomorphism into $\GGG_1(0)$ (respectively into
   $\GGG_{\infty}(0)$).
\end{enumerate}
\end{cor}

For two pairs $(r,N), (s,M) \in \{1,\infty\} \times (\ZZZ_+ \setminus \{1\})$ let us write $(r,N) \preccurlyeq (s,M)$
iff $r \leqsl s$ and $N | M$. The reader will easily check that

\begin{pro}{embed}
Let $(r,N) \in \{1,\infty\} \times (\ZZZ_+ \setminus \{1\})$.
\begin{enumerate}[\upshape(A)]
\item $\GGG_s(M)$ is embeddable in $\GGG_r(N)$ by means of an isometric group homomorphism iff $(s,M) \preccurlyeq (r,N)$.
\item Let $M | N$ (and $M \neq 1$). Then the group $\GGG_r(N,M) := \{x \in \GGG_r(N)\dd M \cdot x = 0\}$ is isometrically
   group isomorphic to $\GGG_r(M)$.
\end{enumerate}
\end{pro}

The above simple result has two interesting consequence, formulated in the next two results.

\begin{pro}{chain}
There is a family $$\{\Phi_{s,M}^{r,N}\dd \GGG_s(M) \to \GGG_r(N)\}_{(s,M) \preccurlyeq (r,N)}$$ of isometric group
homomorphisms such that for each $(r,N)$, $(s,M)$, $(t,L)$ with $(r,N) \preccurlyeq (s,M) \preccurlyeq (t,L)$:
\begin{enumerate}[\upshape(i)]
\item $\Phi_{r,N}^{r,N} = \id_{\GGG_r(N)}$,
\item $\Phi_{s,M}^{t,L} \circ \Phi_{r,N}^{s,M} = \Phi_{r,N}^{t,L}$.
\end{enumerate}
\end{pro}
\begin{proof}
Let $\Psi_{1,0}\dd \GGG_1(0) \to \GGG_{\infty}(0)$ be an isometric group homomorphism. Further, let $\Psi_{\infty,0} =
\id_{\GGG_{\infty}(0)}$ and for $N \geqsl 2$ let $\Psi_{\infty,N}\dd \GGG_{\infty}(N) \to \GGG_{\infty}(0,N)$ and
$\Psi_{1,N}\dd \GGG_{1,N} \to \GGG_{\infty}(0,N) \cap \Psi_{1,0}(\GGG_1(0))$ be isometric group isomorphisms (where
$\GGG_{\infty}(0,N)$ is as in the statement of \PRO{embed}). Now it suffices to put $\Phi_{s,M}^{r,N} = \Psi_{r,N}^{-1}
\circ \Psi_{s,M}$ provided $(s,M) \preccurlyeq (r,N)$.
\end{proof}

\begin{pro}{prod}
Suupose $N \neq 0$. If $N_1,\ldots,N_s \geqsl 2$ are mutually coprime and $N = N_1 \cdot \ldots \cdot N_s$, then
$\GGG_r(N)$ and $\GGG_r(N_1) \times \ldots \times \GGG_r(N_s)$ are isomorphic as topological groups.
\end{pro}
\begin{proof}
The assertion follows from \PRO{embed} and the fact that the function
$$
\GGG_r(N,N_1) \times \ldots \times \GGG_r(N,N_s) \ni (x_1,\ldots,x_s) \mapsto \sum_{j=1}^s x_j \in \GGG_r(N)
$$
is an isomorphism of topological groups.
\end{proof}

In the next section we shall show that $\GGG_r(N)$'s are pairwise nonisomorphic as topological groups.\par
The main assumption of \THM{main2} is that the closure of the image of a group homomorphism is compact. One may ask
whether one may weaken this condition. As the next result shows, nothing else may be done in this direction for group
homomorphims with (metrically) bounded images. (This result has its natural well-known counterpart for the Urysohn
metric space.)

\begin{pro}{almcomp}
Let $K$ be a closed bounded subgroup of $\GGG_r(N)$ such that for every finite group $H$ and every value $q$
on $K \times H$ such that $(K \times H,+,q) \in \Gg_r(N)$ and $q(x,0) = p(x)$ for $x \in K$ there is an isometric group
homomorphism $\psi\dd K \times H \to \GGG_r(N)$ with $\psi(x,0) = x$ for $x \in K$. Then $K$ is compact.
\end{pro}
\begin{proof}
Suppose $K$ is noncompact. We infer from the completeness of $K$ that then there exist $\epsi > 0$ and a sequence
$(x_n)_{n=1}^{\infty} \subset K$ such that $p(x_n - x_m) \geqsl \epsi$ for distinct $n$ and $m$. Put $A = \{x_n\dd\
n \geqsl 1\}$. For every $x \in \GGG_r(N)$ let $e_x\dd A \to \RRR_+$ be given by $e_x(a) = p(x - a)$. It is easily seen
that the set $E = \{e_x\dd\ x \in \GGG_r(N)\} \subset E_r(A)$ is separable (with respect to the supremum metric), since
the map $\GGG_r(N) \ni x \mapsto e_x \in E$ is a nonexpansive surjection. However, we shall show that $E$ contains
an uncountable discrete subset (which will finish the proof).\par
Suppose $f \in E_r(A)$ is such that $f$ satisfies \eqref{eqn:trv-N} for any $a_1,\ldots,a_N \in A$ provided $N > 2$.
Then, by \THM{trv}, $f$ is trivial in some valued group belonging to $\Gg_r(N)$ and of the form $K \times H$ with
$H$ finite. So, thanks to our assumption on $K$, $f$ is trivial in $\GGG_r(N)$ which means that $f \in E$.\par
Let $M = \diam(K) > 0$. Take $\delta \in (0,\epsi)$ such that $\delta < M / 2$. For a subset $J$ of $A$ let $f_J\dd A \to
\RRR_+$ be given by: $f_J(a) = M$ if $a \in J$ and $f_J(a) = M - \delta$ otherwise. A direct calculation shows that
$f_J \in E_r(A)$ and $f_J$ satisfies \eqref{eqn:trv-N} provided $N \neq 0$. So, according to the previous paragraph,
$f_J \in E$. But $\|f_J - f_{J'}\|_{\infty} = \delta$ whenever $J$ and $J'$ are different subsets of $A$ and hence $E$
cannot be separable.
\end{proof}

\begin{rem}{Melleray}
Melleray \cite{me1,me2} has shown that if every isometry between to subsets of the Urysohn metric space $\UUU$ which are
isometric to a given separable complete metric space $X$ is extendable to an isometry of $\UUU$ onto itself, then $X$ is
compact. A counterpart of this result in category of valued groups reads as follows:
\begin{quote}
\textit{If every isometric group isomorphism between two subgroups of $\GGG_r(N)$ which are isometrically group isomorphic
to a complete group $(H,+,q) \in \Gg_r(N)$ is extendable to an isometric group automorphism of $\GGG_r(N)$, then $H$ is
compact.}
\end{quote}
It follows from \PRO{almcomp} that every group $H$ having the above property and bounded value has to be compact. However,
the problem whether the above stated result is true in $\GGG_{\infty}(N)$ we leave open.
\end{rem}

\SECT{Geometry of $\GGG_r(N)$'s}

The part is mainly devoted to investigations of the groups $\GGG_r(N)$'s as metric spaces. We begin with

\begin{thm}{Ur}
The metric spaces $\GGG_r(N)$ with $N \in \{0,2\}$ are Urysohn. In particular, $\GGG_r(2)$ is a Boolean Urysohn metric
group introduced in \cite{pn1}.
\end{thm}
\begin{proof}
This is an almost immediate consequence of (UEP) and \THM{trv}. Let $G_0 = \GGG_r(N)_{fin}$. Let $f \in E^i(G_0)$ and $B$
be a finite nonempty subset of $G_0$. Put $H = \grp{B}$. By \THM{trv}, $H$ may be enlarged to a finite group belonging
to $\Gg_r(N)$ in which $f\bigr|_B$ is trivial. Now we infer from (UEP) that $f\bigr|_B$ is indeed trivial in $G_0$.
So, it follows from the well-known result on the Urysohn space (see e.g. \cite{ur1,ur2}, \cite{kat}, \cite{me2}) that
$\GGG_r(N)$, as the completion of $G_0$, is Urysohn.
\end{proof}

\begin{rem}{QG}
The same argument as in the proof of \THM{Ur} shows that the metric space $\QQQ\GGG_r(N)$ for $N \in \{0,2\}$ is
the so-called rational Urysohn space.
\end{rem}

\THM{Ur} tells us `everything' about the metric spaces $\GGG_r(N)$ with $N \in \{0,2\}$. Therefore \textbf{from now on,
we assume $\pmb{N} \pmb{>} \pmb{2}$.} In \THM{noniso} we shall show that in that case $\GGG_r(N)$ is \textbf{not}
Urysohn.\par
In the same way as in the proof of \THM{Ur} one shows the next result the proof of which is left as an exercise (use
\THM{trv} and \REM{N}; recall that all elements are of finite rank).

\begin{pro}{kat-trv}
\begin{enumerate}[\upshape(A)]
\item Let $B$ be a finite nonempty subset of $\GGG_r(N)$ and $f \in E_r(B)$. Then $f$ is trivial in $\GGG_r(N)$ iff
   $f$ fulfills \eqref{eqn:trv-N} for any $a_1,\ldots,a_N \in B$. In particular, $f$ is trivial in $\GGG_r(N)$ iff
   $f\bigr|_A$ is trivial in $\GGG_r(N)$ for every subset $A$ of $B$ such that $0 < \card(A) \leqsl N$.
\item Let $H$ be a finite subgroup of $\GGG_r(N)$ and $f \in E_r(H)$. The map $f$ is trivial in $\GGG_r(N)$ iff
   $f$ fulfills \eqref{eqn:trvH-N}.
\end{enumerate}
\end{pro}

It turns out that the exponent $N$ is determined by the metric of $\GGG_r(N)$. This is a consequence of the above result
and the following

\begin{exm}{ZN}
Let $Z = \grp{b}$ be a cyclic group of rank $N$ and let $H = Z^N$. For $j=1,\ldots,N$ let $e_j = (e_{j1},\ldots,e_{jN})
\in H$ with $e_{jj} = b$ and $e_{jk} = 0$ for $k \neq j$. Further, let $F = \{\pm e_j\dd\ j=1,\ldots,e_N\} \cup
\{e_j - e_k\dd\ j,k=1,\ldots,N\} \subset H$. of course, $F$ generates $H$. Let $\|\cdot\|_F$ be the `norm' on $H$
generated by $F$ (in the terminology of finitely generated groups, cf. \cite{gro}), i.e. $\|\cdot\|_F$ is a value such
that for nonzero $h \in H$,
$$
\|h\|_F = \min\{n \geqsl 1|\ \exists f_1,\ldots,f_n \in F\dd\ h = \sum_{s=1}^n f_s\}
$$
(note that $F = -F$). We see that
\begin{equation}\label{eqn:aux44}
\|e_j - e_k\|_F = 1 \qquad \textup{for} \qquad j \neq k.
\end{equation}
We claim that if $n_1,\ldots,n_N \in \ZZZ_+$,
\begin{equation}\label{eqn:aux45}
\|\sum_{j=1}^N n_j e_j\|_F \leqsl N - \max(n_1,\ldots,n_N) \quad \textup{provided } \sum_{j=1}^N n_j = N.
\end{equation}
Indeed, suppose that e.g. $\max(n_1,\ldots,n_N) = n_N$ and observe that $\sum_{j=1}^N n_j e_j =
\sum_{j=1}^{N-1} n_j(e_j - e_N)$, thanks to the assumption in \eqref{eqn:aux45}. The latter equality gives
\eqref{eqn:aux45}. Further, we have
\begin{equation}\label{eqn:aux46}
\|\sum_{j=1}^N e_j\|_F = N-1.
\end{equation}
To see this, suppose (for the contrary) that $\|\sum_{j=1}^N e_j\|_F \leqsl N - 2$. This means (since $0 \in F$) that
there are $f_1,\ldots,f_{N-2} \in F$ which sum up to $\sum_{j=1}^N e_j$. Write $f_s = \sum_{j=1}^N \epsi_{js} e_j$
where $\epsi_{js} \in \{0,1,-1\}$. Since the rank of $b$ is $N$, we conclude from this that $\sum_{s=1}^{N-2} \epsi_{js}
= 1$ for $j=1,\ldots,N$ and hence $\sum_{j=1}^N \sum_{s=1}^{N-2} \epsi_{js} = N$. However, $\sum_{j=1}^N \epsi_{js} \in
\{-1,0,1\}$ for $s=1,\ldots,N-2$ (because $f_s \in F$) and therefore $\sum_{s=1}^{N-2} |\sum_{j=1}^N \epsi_{js}| < N-1$
which denies earlier conclusion. So, \eqref{eqn:aux46} is fulfilled (by \eqref{eqn:aux45}).\par
Now put $c = \max(1/2,1 - 2/N)$ and let $f\dd \{e_1,\ldots,e_N\} \to \RRR_+$ be constantly equal to $c$.
By \eqref{eqn:aux44}, $f$ is a Kat\v{e}tov map. Observe that (thanks to \eqref{eqn:aux46})
\begin{equation}\label{eqn:balls}
\|\sum_{j=1}^N e_j\|_F > \sum_{j=1}^N f(e_j)
\end{equation}
(here is the only moment where we need to have $N \geqsl 3$). So, $f$ fails to satisfy \eqref{eqn:trv-N} and thus $H$
cannot be enlarged to a valued group of exponent $N$ in which $f$ is trivial (by \PRO{kat-exp}). However, if $A$ is
a proper nonempty subset of $\{e_1,\ldots,e_N\}$ and $a_1,\ldots,a_N \in A$, then the inequality \eqref{eqn:trv-N} is
fulfilled, by \eqref{eqn:aux45}. Consequently, $H$ may be enlarged to a finite group belonging to $\Gg_r(N)$ in which
$f\bigr|_A$ is trivial.\par
If we now replace $\|\cdot\|_F$ by $\alpha \|\cdot\|_F$ with small enough $\alpha > 0$, we shall obtain an analogous
example in the class $\Gg_1(N)$.
\end{exm}

As a corollary of \THM{Ur}, \PRO{kat-trv} and the above example we obtain

\begin{thm}{noniso}
When $N > 2$, $N$ is the least natural number $k$ with the following property. For every finite nonempty subset $A$
of $\GGG_r(N)$ and each $f \in E_r(A)$, $f$ is trivial in $\GGG_r(N)$ iff $f\bigr|_{B}$ is trivial in $\GGG_r(N)$
for any nonempty $B \subset A$ with $\card(B) \leqsl k$.\par
In particular, for two distinct pairs $(r,N), (s,M) \in \{1,\infty\} \times (\ZZZ_+ \setminus \{1\})$, the metric spaces
$\GGG_r(N)$ and $\GGG_r(M)$ are isometric iff $r = s$ and $\{N,M\} = \{0,2\}$. Hence, $\GGG_r(N)$ is non-Urysohn
for $N>2$.
\end{thm}

Further geometric properties of $\GGG_r(N)$'s are stated below.

\begin{pro}{spheres}
Let $f \in E^i(\GGG_r(N))$.
\begin{enumerate}[\upshape(A)]
\item For every two-point subset $B$ of $\GGG_r(N)$ the map $f\bigr|_B$ is trivial in $\GGG_r(N)$.
\item If $N=4$, $f$ is trivial (in $\GGG_r(N)$) on every three-point subset of the space.
\item If $N\neq4$, there is a three-point set $C \subset \GGG_r(N)$ and a map $g \in E_r(C)$ which is nontrivial
   in $\GGG_r(N)$.
\end{enumerate}
\end{pro}
\begin{proof}
The points (b) and (c) are left as exercises. (To prove (c) for $N > 4$, take $Z$ as in \EXM{ZN}, $H = Z^3$, define
$e_1,e_2,e_3$, $F$ and $\|\cdot\|_F$ in a similar manner and consider $f\dd \{e_1,e_2,e_3\} \to \RRR_+$ constantly
equal to $1/2$. Show that $\|(k + r_1) e_1 + (k + r_2) e_2 + (k + r_3) e_3\|_F \geqsl \frac23 (N-1) > \frac{N}{2}$ for
$r_1,r_2,r_3 \in \{0,1,2\}$ which sum up to $r \in \{0,1,2\}$ such that $N = 3k + r$.)\par
To show (a), it remains to check that \eqref{eqn:trv-N} is fulfilled whenever $a_1,\ldots,a_N \in \{a,b\} \subset
\GGG_r(N)$. In that case \eqref{eqn:trv-N} reduces to $|p(j(a - b)) - f(a)| \leqsl (j-1)f(a) + (N-j)f(b)$ (since
$(N-j)b = -jb$) with $j=2,\ldots,N-1$ (for $j=1$ this inequality follows from the definition of a Kat\v{e}tov map).
Of course $f(a) - p(j(a-b)) \leqsl (j-1)f(a)$, so we only need to check that $p(j(a-b)) \leqsl j f(a) + (N-j)f(b)$.
By the symmetry ($j(a-b) = -(N-j)(a-b)$), we may assume that $j \leqsl N/2$. But then $p(j(a-b)) \leqsl
j(f(a) + f(b)) \leqsl j f(a) + (N-j) f(b)$.
\end{proof}

\begin{cor}{segm}
For any two distinct points $x$ and $y$ of $\GGG_r(N)$ there is an isometric arc of $[0,c]$ to $\GGG_r(N)$ joining
$x$ and $y$, where $c = p(x-y)$.
\end{cor}
\begin{proof}
By \PRO{spheres}, for any $x,y \in \GGG_r(N)$ there is $z \in \GGG_r(N)$ such that $p(x-z) = p (y-z) = p(x-y)/2$.
Since $\GGG_r(N)$ is complete, the assertion follows.
\end{proof}

The above result implies that every metric space which is the image of $\GGG_1(N)$ under a uniformly continuous
function has bounded metric. Since continuous group homomorphisms are uniformly continuous, as a consequence of this
we obtain the result announced in the previous section.

\begin{cor}{nonisog}
The groups $\GGG_r(N)$'s are pairwise nonisomorphic as topological groups.
\end{cor}

The inequality \eqref{eqn:balls} in \EXM{ZN} implies that there are points $x_1,\ldots,x_N$ in $\GGG_r(N)$ and radii
$r_1,\ldots,r_N > 0$ such that $p(x_j - x_k) < r_j + r_k$ for any $j,k$, but the closed balls $\bar{B}(x_j,r_j)$
($j=1,\ldots,N$) have empty intersection. We conclude from this that the metric space $\GGG_r(N)$ fails to have
the property of extending isometric maps between finite subsets of $\GGG_r(N)$ to Lipschitz maps with Lipschitz
constants arbitrarily close to $1$. This is why it is not so easy (as in case of $\GGG_r(0)$ and $\GGG_r(2)$ which
are Urysohn spaces; compare with \cite{usp} or \cite{pn3} where it is shown that the Urysohn space is homeomorphic
to $l^2$) to prove that $\GGG_r(N)$ is an absolute retract (for metric spaces). We shall do this in Section~8.\par
Our last aim of this part is to show that each of the metric spaces $\GGG_r(N)$'s is metrically universal for
separable metric spaces of diameter no greater than $r$.\par
In what follows, $N \geqsl 2$ and $\ZZZ_N = \ZZZ / N\ZZZ$ represents a cyclic group of rank $N$. Moreover, let $\eE$ stand
for the generator of $\ZZZ_N$. Adapting the idea of Lipschitz-free Banach spaces generated by metric spaces (see e.g.
\cite{wea}, \cite{g-k}), we introduce

\begin{dfn}{ZN-free}
Let $X$ be a nonempty set. For $x \in X$ let $H_x = \ZZZ_N$ and let $\ZZZ_N[[X]] = \bigoplus_{x \in X} H_x$ be
the direct product of groups $H_x$'s. That is, $\ZZZ_N[[X]]$ consists of all functions $f\dd X \to \ZZZ_N$ for which
the set $\supp f := \{x \in X\dd\ f(x) \neq 0\}$ is finite and $\ZZZ_N[[X]]$ is equipped with the pointwise addition.
Further, let $\ZZZ_N[X] = \{f \in \ZZZ_N[[X]]\dd\ \sum_{x \in X} f(x) = 0\}$. It is clear that $\ZZZ_N[X]$ is a subgroup
of $\ZZZ_N[[X]]$.\par
For every $x \in X$ let $\widehat{x} \in \ZZZ_N[[X]]$ be such that $\widehat{x}(x) = \eE$ and $\widehat{x}(y) = 0$
for $y \neq x$. The set $\{\widehat{x} - \widehat{y}\dd\ x,y \in X\}$ generates the group $\ZZZ_N[X]$.\par
Whenever $d$ is a metric on $X$, define $p_d\dd \ZZZ_N[X] \to \RRR_+$ by
\begin{multline}\label{eqn:pd}
p_d(f) = \inf\{\sum_{j=1}^n d(x_j,y_j)\dd\\
n \geqsl 1,\ x_1,y_1,\ldots,x_n,y_n \in X,\ f = \sum_{j=1}^n (\widehat{x_j} - \widehat{y_j})\}.
\end{multline}
The triple $(\ZZZ_N[X],+,p_d)$ is called the \textit{valued Abelian group of exponent $N$ generated by the metric
space $(X,d)$}.
\end{dfn}

As we will see in the next result, $p_d$ is indeed a value. For need of this, let us introduce the following
notation. Let $(X,d)$ be a finite metric space. If $\card(X) < 2$, let $\mu(X) := 0$. Otherwise let
$$
\mu(X) := \max\bigr\{ \min\{d(x,y)\dd\ y \in X,\ y \neq x\}\bigr\} > 0.
$$
Then we have

\begin{pro}{pd}
Let $(X,d)$ be a nonempty metric space. For every $f \in \ZZZ_N[X]$:
\begin{enumerate}[\upshape(A)]
\item $p_d(f) \geqsl \mu(\supp f)$; in particular, $p_d$ is a value,
\item if $N = 2$ and $f \neq 0$,
   \begin{multline}\label{eqn:pd2}
   p_d(f) = \min\{\sum_{j=1}^k d(x_j,y_j)\dd\ x_1,y_1,\ldots,x_k,y_k \textup{ are all different}\\
   \textup{and } \{x_1,y_1,\ldots,x_k,y_k\} = \supp(f)\}.
   \end{multline}
\end{enumerate}
\end{pro}
\begin{proof}
First observe that---thanks to the triangle inequality---in the formula \eqref{eqn:pd} we may consider only such
systems $x_1,y_1,\ldots,x_n,y_n \in X$ that
\begin{equation}\label{eqn:aux50}
\{x_1,\ldots,x_n\} \cap \{y_1,\ldots,y_n\} = \varempty
\end{equation}
and---for the same reason---if $N = 2$, we may also restrict to systems $x_1,\ldots,x_n$ and $y_1,\ldots,y_n$ in which
elements are different. This proves (B). Now assume that $N \geqsl 3$, that
\begin{equation}\label{eqn:aux51}
f = \sum_{j=1}^n (\widehat{x_j} - \widehat{y_j})
\end{equation}
and \eqref{eqn:aux50} is fulfilled. Let $F = \{x_1,y_1,\ldots,x_n,y_n\}$. For $a,b \in F$ we write $a \sim b$ provided
there is $j \in \{1,\ldots,n\}$ with $\{a,b\} = \{x_j,y_j\}$; and $a \equiv b$ iff either $a = b$ or there are
$c_0,c_1,\ldots,c_k \in F$ for which $c_0 = a$, $c_k = b$ and $c_j \sim c_{j-1}$ for $j=1,\ldots,k$. It is clear that
`$\equiv$' is an equivalence on $F$ such that
\begin{equation}\label{eqn:aux52}
\{j\dd\ x_j \equiv a\} = \{j\dd\ y_j \equiv a\} \quad \textup{for each } a \in A.
\end{equation}
We infer from \eqref{eqn:aux50} and \eqref{eqn:aux51} that $\supp(f) \subset F$ and for each $a \in A$, $f(a) =
\card(\{j \in \{1,\ldots,n\}\dd\ x_j = a\}) \eE$ or $f(a) = -\card(\{j \in \{1,\ldots,n\}\dd\ y_j = a\}) \eE$, and hence
\begin{equation}\label{eqn:aux53}
\begin{cases}
\card(\{j\dd\ a \in \{x_j,y_j\}\}) \not\equiv 0 \mod N & \textup{if } a \in \supp(f),\\
\card(\{j\dd\ a \in \{x_j,y_j\}\}) \equiv 0 \mod N & \textup{otherwise}.
\end{cases}
\end{equation}
Fix $a \in A$. It follows from \eqref{eqn:aux52} and \eqref{eqn:aux53} that there is $b \in \supp(g) \setminus \{a\}$
such that $b \equiv a$. So, there are $c_0,\ldots,c_k \in A$ for which $c_0 = a$, $c_k = b$ and $c_j \sim c_{j-1}$ for
$j=1,\ldots,k$. Passing into a suitable subset of $\{0,\ldots,k\}$ we may assume that all $c_j$'s are different. This
means that there are distinct indices $\nu_1,\ldots,\nu_k$ such that $\{c_j,c_{j-1}\} = \{x_{\nu_j},y_{\nu_j}\}$. But
then $\sum_{j=1}^n d(x_j,y_j) \geqsl \sum_{s=1}^k d(c_s,c_{s-1}) \geqsl d(a,b)$. So, the assertion follows from
the arbitrarity of $a \in \supp(f)$.
\end{proof}

\begin{cor}{pd-d}
Let $(X,d)$ be a nonempty metric space and $a \in X$. The function
$$
(X,d) \ni x \mapsto \widehat{x} - \widehat{a} \in (\ZZZ_N[X],p_d)
$$
is isometric.
\end{cor}

If we replace in the above result $p_d$ by $p_d \wedge r$, the assertion of \COR{pd-d} well remain true. It is also clear
that $\ZZZ_N[X]$ is separable provided so is $X$. So, an application of \COR{univ} yields

\begin{thm}{univ}
Every separable metric space of diameter no greater than $r$ admits an isometric embedding into the metric space
$\GGG_r(N)$.
\end{thm}

\begin{exm}{odd}
If $X$ is a nonempty set, $H$ is a group of exponent $N$, every function $u\dd X \to H$ induces a unique (well defined)
group homomorphism $\widetilde{u}\dd \ZZZ_N[X] \to H$ such that $\widetilde{u}(\widehat{x} - \widehat{y}) = u(x) - u(y)$
for any $x,y \in X$. What is more, if $d$ is a metric on $X$, $q$ is a value on $H$ and $u\dd (X,d) \to (H,q)$ is
Lipschitz, it easily follows from the formula for $p_d$ that $\widetilde{u}\dd (\ZZZ_N[X],p_d) \to (H,q)$ satisfies
the Lipschitz condition with the same constant as $u$. Similarly, every function $v\dd X \to Y$ between two sets induces
a group homomorphism $\widehat{v}\dd \ZZZ_N[X] \to \ZZZ_N[Y]$ and if $v\dd (X,d) \to (Y,\varrho)$ is Lipschitz, then
$\widehat{v}\dd (\ZZZ_N[X],p_d) \to (\ZZZ_N[Y],p_{\varrho})$ satisfies the Lipschitz condition with the same constant
as $v$.\par
In particular, if $A$ is a subset of $(X,d)$, the inclusion map of $A$ into $X$ induces a nonexpansive group homomorphism
of $\ZZZ_N[A]$ into $\ZZZ_N[X]$. In case of $N = 2$, the latter group homomorphism is isometric, which easily follows from
\PRO{pd} (similar result holds true for the Lipschitz-free Banach spaces generated by metric spaces---see \cite{wea}).
It turns out that for $N > 2$ this homomorphism may not be isometric, which is rather strange. Let us give an example
based on \EXM{ZN}.\par
Let $X = \{0,1,2,\ldots,N\}$ and $A = X \setminus \{0\}$. We equip $X$ with a metric $d$ such that for distinct $j,k \in
A$, $d(j,k) = 1$ and $d(0,j) = \max(1/2,1 - 2/N)$. For clarity, let $\varrho = d_{A \times A}$. Put $f = \sum_{a \in A}
\widehat{a} \in \ZZZ_N[A]$. It follows from the argument in \EXM{ZN} that $p_{\varrho}(f) = N-1$. However, in $\ZZZ_N[X]$,
$f = \sum_{a \in A} (\widehat{a} - \widehat{0})$ and thus $p_d(f) \leqsl N \cdot \max(1/2,1 - 2/N) < N - 1$.\par
Taking into account \THM{trv}, \THM{Ur}, \THM{noniso} and the above example, we see that the case of $N > 2$ is
`singular'.
\end{exm}

\SECT{Pseudovector groups}

Pseudovector groups were introduced in \cite{pn2}. Here we shall generalize the results of \cite{pn2} on pseudovector
structures on the (unbounded) Boolean Urysohn group, mainly in order to prove that all the groups $\GGG_r(N)$'s are
homeomorphic to the Hilbert space. However, some of results of this part may be seen interesting and can play important
role in theory of topological pseudovector groups. All terms, beside the notion of a $\nabla$-norm, are repeated
from the introductory work \cite{pn2}. As before, we restrict our considerations only to Abelian groups.\par
We begin with

\begin{dfn}{pvg}
Let $\FFF$ be a subfield of $\RRR$. A triple $(G,+,*)$ is said to be a \textit{pseudovector (Abelian) group} over $\FFF$
(briefly, an $\FFF$-PV group, or a PV group provided $\FFF = \RRR$) iff $(G,+)$ is an Abelian group and $*\dd \FFF_+
\times G \to G$ is an action such that for any $s,t \in \FFF_+$ and $x \in G$, $0 * x = 0$, $1 * x = x$, $(st) * x =
s * (t * x)$ and the function $G \ni y \mapsto t * y \in G$ is a group homomorphism.\par
A \textit{topological pseudovector group} over $\FFF$ is an $\FFF$-PV group $(G,+,*)$ equipped with a topology $\tau$
such that $(G,+,\tau)$ is a topological group and the action `$*$' is continuous.\par
A \textit{norm} on an $\FFF$-PV group $(G,+,*)$ is a value $\|\cdot\|\dd G \to \RRR_+$ such that $\|t * x\| = t \|x\|$
for each $t \in \FFF_+$ and $x \in G$. The norm $\|\cdot\|$ is \textit{topological} iff the action `$*$' is continuous with
respect to $\|\cdot\|$. In that case we speak of a \textit{normed topological pseudovector group} over $\FFF$ (for short,
a NTPV group over $\FFF$).\par
If $G$ is an $\FFF$-PV group, $\KKK$ is a subfield of $\FFF$ and $A$ any subset of $G$, by $\lin_{\KKK} A$ we denote
the $\KKK$-PV subgroup of $G$ generared by $A$, that is, $\lin_{\KKK} A = \{\sum_{j=1}^n t_j * a_j\dd\ n \geqsl 1,\
t_1,\ldots,t_n \in \KKK_+,\ a_1,\ldots,a_n \in A \cup (-A) \cup \{0\}\}$. If $\KKK = \RRR$, we write $\lin A$ instead
of $\lin_{\RRR} A$.\par
A group homomorphism $u\dd G \to H$ between two $\FFF$-PV groups is \textit{linear} if $u(t * x) = t * u(x)$ for every
$t \in \FFF_+$ and $x \in G$.
\end{dfn}

Our aim is to show that each of the groups $\GGG_r(N)$'s may be endowed with a `normed-like' topological pseudovector
structure. Since every norm on a nontrivial group is unbounded, none of the groups $\GGG_1(N)$ is isometrically group
isomorphic to a NTPV group. Thus, we have to generalize the notion of a norm.

\begin{dfn}{norming}
A function $\kappa\dd \RRR_+ \to \RRR_+$ is said to be \textit{norming} iff $\kappa$ satisfies the following conditions:
\begin{enumerate}[(NF1)]
\item $\kappa(1) = 1$ and $\kappa(x) \geqsl x$ for each $x \geqsl 0$,
\item $\kappa(xy) \leqsl \kappa(x) \kappa(y)$ for any $x,y \in \RRR_+$,
\item $\kappa$ is monotone increasing, i.e. $\kappa(x) \leqsl \kappa(y)$ provided $0 \leqsl x \leqsl y$,
\item $\kappa$ is Lipschitz, that is, there is a constant $L' > 0$ such that $|\kappa(x) - \kappa(y)| \leqsl
   L' |x - y|$ for $x,y \geqsl 0$.
\end{enumerate}
(Notice that it is \textbf{not} assumed that $\kappa(0) = 0$.)\par
A (semi)value $\|\cdot\|$ on an $\FFF$-PV group $(G,+,*)$ is said to be a \textit{$\kappa$-(semi)norm} iff
$\|t * x\| \leqsl \kappa(t) \|x\|$ for each $t \in \FFF_+$ and $x \in G$. When $\|\cdot\|$ is a $\kappa$-norm
on an $\FFF$-PV group $(G,+,*)$, the quadruple $(G,+,*,\|\cdot\|)$ is said to be a \textit{$\kappa$-normed $\FFF$-PV
group}.
\end{dfn}

\begin{rem}{NF3}
In the whole paper, we use (NF3) only two times: in (P6) (which finds no application in this paper) and in the proof
of \LEM{00-PV} below, which may be improved so that (NF3) will not be applied. The reason for adding (NF3) to the axioms
of a norming function is just a matter of our personal `taste'.
\end{rem}

The most important examples on norming functions are $\id\dd \RRR_+ \to \RRR_+$ and $\nabla = \id \vee 1$. We leave these
as simple exercises that $\nabla$ is indeed norming and that a value is an $\id$-norm iff it is a norm. We call
a $\nabla$-norm also a \textit{subnorm}, and \textit{subnormed} means $\nabla$-normed. In fact we need only these two
functions. However, no additional work is needed for arbitrary norming functions and it seems to us instructive to point
out which properties of $\id$ and $\nabla$ are relevant in our considerations.\par
The reader should check with no difficulties the following properties of every norming function $\kappa$. Below we assume
that $L'$ is as in (NF4).
\begin{enumerate}[(P1)]
\item If $\kappa(0) = 0$, then $\id \leqsl \kappa \leqsl L' \id$.
\item If $\kappa(0) \neq 0$, then $\nabla \leqsl \kappa \leqsl (L' + 1) \nabla$.
\item Every norm is a $\kappa$-norm.
\item If $\kappa(0) \neq 0$, every $\nabla$-norm is a $\kappa$-norm.
\item if $\kappa(0) \neq 0$ and $\|\cdot\|$ is a $\kappa$-norm, then $\|\cdot\| \wedge 1$, $\frac{\|\cdot\|}{1+\|\cdot\|}$
   and $\|\cdot\|^{\alpha}$ with $0 < \alpha < 1$ are $\kappa$-norms as well.
\item Composition of two norming functions is a norming function.
\end{enumerate}

As the following result shows, $\nabla$-norms appear much more often than norms.

\begin{pro}{metriz}
Every topological pseudovector group metrizable as a topological space admits a $\nabla$-norm inducing its topology.
\end{pro}
\begin{proof}
Let $(G,+,*)$ be a topological pseudovector group. By a well-known theorem (see e.g. \cite[Theorem~8.8]{ber}), there is
a value $p$ on $G$, bounded by $1$, which induces the topology of $G$. Now define $\|\cdot\|\dd G \to \RRR_+$ by
$$
\|x\| = \sup\{\frac{p(\alpha * x)}{\alpha \vee 1}\dd\ \alpha \geqsl 0\}.
$$
We leave this as an exercise that $\|\cdot\|$ is a subnorm equivalent to $p$ (use the boundedness of $p$
and the continuity of `$*$').
\end{proof}

Note that the above result may be applied to any metrizable topological vector space.\par
From now to the end of the section $\kappa$ is a fixed norming function and $L \geqsl 1$ is such a constant that
for all $x, y \geqsl 0$,
\begin{equation}\label{eqn:L}
\begin{cases}
|\kappa(x) - \kappa(y)| \leqsl L |x - y|, &\\
\kappa(x) \leqsl L \nabla(x)&
\end{cases}
\end{equation}
(cf. (NF4), (P1) and (P2)). Note also that it follows from (P1) that $\kappa(x) \leqsl L x$ for $x \geqsl 0$ provided
$\kappa(0) = 0$.\par
Everywhere below $\FFF$ is a subfield of $\RRR$ and $(G,+,*,\|\cdot\|_G)$ is a $\kappa$-normed pseudovector group. At this
moment, we do not assume that the action `$*$' is continuous.\par
As in \cite{pn2}, we call an element $a \in G$ \textit{continuous} (respectively \textit{Lipschitz}) provided the function
$\FFF_+ \ni t \mapsto t * a \in G$ is continuous (Lipschitz). The set of all continuous (Lipschitz) elements of $G$ is
denoted by $\CcC_{\FFF}(G)$ ($\LlL_{\FFF}(G)$) and in case $\FFF = \RRR$ we write $\CcC(G)$ ($\LlL(G)$). $G$ is said to be
\textit{Lipschitz} if $\LlL_{\FFF}(G) = G$. Similarly, a $\kappa$-seminorm $\|\cdot\|_0$ on $G$ is \textit{Lipschitz} iff
for every $a \in G$ there is $C_a \in \RRR_+$ such that $\|t * a - s * a\|_0 \leqsl C_a |t - s|$ for any $s,t \in
\RRR_+$.\par
The proofs of the next two results are omitted.

\begin{pro}{cont}
\begin{enumerate}[\upshape(A)]
\item For $a \in G$, $a \in \CcC_{\FFF}(G)$ iff $$\lim_{\substack{t\to0\\t\in\FFF_+}} \|t * a\|_G = 0 \qquad \textup{and}
   \qquad \lim_{\substack{t\to1^-\\t\in\FFF^+}} \|a - t * a\|_G = 0.$$
\item For $a \in G$, $a \in \LlL_{\FFF}(G)$ iff the function $\FFF \cap [0,1] \ni t \mapsto t * a \in G$ is Lipschitz.
\item The set $\CcC_{\FFF}(G)$ is a \textbf{closed} pseudovector subgroup (over $\FFF$) of $G$ and $\CcC_{\FFF}(G)$ is
   a topological pseudovector group. The set $\LlL_{\FFF}(G)$ is a pseudovector subgroup (over $\FFF$)
   of $\CcC_{\FFF}(G)$.
\item For every $a \in \CcC_{\FFF}(G)$ the map $\FFF_+ \ni t \mapsto t * a \in G$ is uniformly continuous (as a map
   between metric spaces) on every bounded subset of $\FFF_+$.
\end{enumerate}
\end{pro}

\begin{pro}{compl}
Suppose $\CcC_{\FFF}(G) = G$. Let $(\bar{G},+,\|\cdot\|)$ be the completion of $(G,+,\|\cdot\|_G)$. There is an action
$\bar{*}\dd \RRR_+ \times \bar{G} \to \bar{G}$ such that $(\bar{G},+,\bar{*},\|\cdot\|)$ is a $\kappa$-normed topological
PV group and $t * x = t \bar{*} x$ for $t \in \FFF_+$ and $x \in G$. Moreover, $\LlL_{\FFF}(G) = \LlL(\bar{G}) \cap G$.
\end{pro}

The next result will be used several times by us.

\begin{lem}{bd}
If there is $a \in G \setminus \{0\}$ such that $M := \sup\{\|t * a\|_G\dd\ t \in \FFF_+\} < \infty$, then $\kappa(0)
\neq 0$.
\end{lem}
\begin{proof}
For positive $t, s \in \FFF$ we have $\|t * x\|_G \leqsl \kappa(s) \|\frac{t}{s} * x\|_G \leqsl M \kappa(s)$ and thus
$M \leqsl M \kappa(s)$. So, $\kappa(s) \geqsl 1$ and the assertion follows.
\end{proof}

\begin{lem}{constM}
For every $a \in \CcC_{\FFF}(G)$ and $\delta > 0$ there exists a constant $M = M(a,\delta) \in \RRR_+$ such that for any
$s,t \in \RRR_+$,
\begin{equation}\label{eqn:M}
\|s * a - t * a\|_G \leqsl M |s - t| + \delta \kappa(s \vee t).
\end{equation}
\end{lem}
\begin{proof}
Let $$M = L \sup\bigl\{\frac{\|s * a - t * a\|_G - \delta}{|s - t|}\dd\ s, t \in \FFF \cap [0,1],\ s \neq t\bigr\}.$$
We infer from point (D) of \PRO{cont} that $M < \infty$. We shall show that $M$ satisfies \eqref{eqn:M}. First assume that
$\kappa(0) = 0$. We know that then $\kappa(t) \leqsl Lt$ for every $t \in \RRR_+$. So, for $t > s \geqsl 0$ we have
$\|s * a - t * a\|_G \leqsl \kappa(t) \|(s/t) * a - a\|_G \leqsl \kappa(t) [(M / L) (1 - s/t) + \delta] \leqsl M (t - s)
+ \delta \kappa(t)$.\par
When $\kappa(0) \neq 0$, then $\kappa(t) \geqsl 1$ for every $t \in \RRR_+$ and hence \eqref{eqn:M} is fulfilled when
$t \vee s \leqsl 1$, by the definition of $M$. Finally, for $t > s \geqsl 0$ and $t > 1$, $\kappa(t) \leqsl Lt$ (see
\eqref{eqn:L}) and therefore we may repeat the estimations from the previous paragraph to get the assertion.
\end{proof}

\begin{lem}{APV}
Let $(D_j,+,*,\|\cdot\|_j) \in \Gg_r(N)$ be Lipschitz $\kappa$-normed PV groups over $\FFF$ ($j=0,1,2$) and $\varphi_j\dd
D_0 \to D_j$ ($j=1,2$) be isometric linear homomorphisms. Then there is a Lipschitz PV group $(D,+,*,\|\cdot\|) \in
\Gg_r(N)$ over $\FFF$ and isometric linear homomorphisms $\psi_j\dd D_j \to D$ ($j=1,2$) such that $\psi_2 \circ \varphi_2
= \psi_1 \circ \varphi_1$.
\end{lem}
\begin{proof}
If both $D_1$ and $D_2$ are trivial, it suffices to take $D$ trivial as well. If $D_1$ or $D_2$ is nontrivial, we may
apply \LEM{bd} from which it follows that $\kappa(0) \neq 0$ if $r < \infty$. Thus, in the latter case we may involve (P5).
These comments ensure us that we may repeat the proof of \LEM{A1}. The details are skipped.
\end{proof}

For every valued group $(H,+,q)$ let $L_0[H]$ consist of all functions $u\dd \RRR_+ \to H$ for which there are numbers
$0 = t_0 < t_1 < \ldots < t_n < \infty$ such that $u$ is constant on each of the intervals $[t_{j-1},t_j)$
($j=1,\ldots,n$), say $u([t_{j-1},t_j)) = \{u_j\}$, and $u(t) = 0$ for $t \geqsl t_n$. We shall abbreviate this by writing
$$
u = \sum_{j=1}^n u_j \chi_{[t_{j-1},t_j)}.
$$
$L_0[H]$ is a \textbf{Lipschitz} normed PV group when it is equipped with the pointwise addition, the action given by
$(t * u)(s) = u(s/t)$ for $t > 0$, $s \geqsl 0$ and $u \in L_0[H]$ (of course, $0 * u \equiv 0$) and the norm $\|u\|_q =
\int_0^{\infty} q(u(t)) \dint{t}$ (cf. \cite{pn2}). What is more, for every nonzero $u \in L_0[H]$ there are unique
$n \geqsl 1$, real numbers $0 < t_1 < \ldots < t_n$ and nonzero elements $h_1,\ldots,h_n \in H$ such that $u = \sum_{j=1}^n
t_j * \widehat{h_j}$ where for $h \in H$, $\widehat{h}(s) = h$ for $s \in [0,1)$ and $\widehat{h}(s) = 0$ for $s \geqsl 1$.
Note that the function $H \ni h \mapsto \widehat{h} \in L_0[H]$ is an isometric group homomorphism and every group
homomorphism $\psi\dd H \to K$ of $H$ into a group $K$ induces a unique linear homomorphism $\widehat{\psi}\dd L_0[H] \to
K$ determined by condition $\widehat{\psi}(\widehat{h}) = \psi(h)$ for $h \in H$. Thus, $(L_0[H],+,*)$ may be called
the \textit{free PV group generated by $H$}. A straightforward verification shows that $(L_0[H],+,\|\cdot\|_q)$ is of class
$\OOo_0$ or of exponent $N$ provided so is $(H,+,q)$.\par
From now on, $(G,+,*,\|\cdot\|_G)$ is a $\kappa$-normed PV group. As in the previous sections, we fix $r \in \{1,\infty\}$
and $N \in \ZZZ_+ \setminus \{1\}$. Additionally, we assume the following elementary property:
\begin{equation*}\tag{AX}
G \neq \{0\} \qquad \textup{or} \qquad \kappa(0) \neq 0 \qquad \textup{or} \qquad r = \infty.
\end{equation*}
(This is because of \LEM{bd}. Precisely, we want to enlarge $\kappa$-normed PV groups of class $\Gg_r(N)$. When $G$ is
trivial and $r$ is finite, this is impossible unless $\kappa(0) \neq 0$.)\par
Our aim is to prove counterparts of several results from Section~1 on enlarging groups. As we will see, in case of PV
groups this is much more complicated.

\begin{thm}{g-PVG}
Let $(H,+,q)$ be a finite valued group, $K$ its subgroup and $\|\cdot\|_0$ a Lipschitz $\kappa$-seminorm on $L_0[K]$
such that $\|\widehat{h}\|_0 = q(h)$ for $h \in K$. Then there is a Lipschitz $\kappa$-seminorm on $L_0[H]$ such that
$\|f\| = \|f\|_0$ for $f \in L_0[K]$ and $\|\widehat{h}\| = q(h)$ for $h \in H$.
\end{thm}
\begin{proof}
We assume that $H \neq K$. Since $K$ is finite and $\|\cdot\|_0$ is Lipschitz, there is a constant $\lambda \geqsl 1$ such
that $\|s * \widehat{h} - t * \widehat{h}\|_0 \leqsl \lambda |t - s|$ for every $h \in H$ and $s, t \in \RRR_+$. Let $M =
\max q(H)$, $\mu = \min \{q(h)\dd\ h \in H \setminus \{0\}\}$ and $A = (\lambda + ML) / \mu$. For $f \in L_0[H]$ put
\begin{multline}\label{eqn:aux60}
\|f\| = \inf\{A\|f - g - \sum_{j=1}^n t_j * \widehat{h_j}\|_q + \sum_{j=1}^n \kappa(t_j) q(h_j) + \|g\|_0\dd\\
n \geqsl 1,\ t_1,\ldots,t_n \geqsl 0,\ h_1,\ldots,h_n \in H,\ g \in L_0[K]\}.
\end{multline}
It is clear that $\|\cdot\|$ is a $\kappa$-seminorm on $L_0[H]$ (by (NF2) and (P3)). What is more, $\|\cdot\| \leqsl
A \|\cdot\|_q$ which yields that $\|\cdot\|$ is Lipschitz. Thus, we only need to show that $\|f\| = \|f\|_0$ and
$\|\widehat{h}\| = q(h)$ for $f \in L_0[K]$ and $h \in H$. The inequalities `$\leqsl$' are immediate (since $\kappa(1) =
1$). To prove the inverse inequalities, first of all note that
\begin{multline}\label{eqn:aux61}
\|f\| = \inf\Bigl\{\|\sum_{k=1}^m s_k * \widehat{g_k}\|_0 + A \sum_{j=1}^m (s_j - s_{j-1}) q(\sum_{k=j}^m \epsi_k)\\
+ \sum_{k=1}^m \kappa(s_k) q(f_k - f_{k+1} + g_k + \epsi_k)\dd\
m \geqsl 1,\ 0 = s_0 < \ldots < s_m,\\ \epsi_1,\ldots,\epsi_m,f_1,\ldots,f_m \in H,\ f_{m+1} = 0,\
g_1,\ldots,g_m \in K,\\ f = \sum_{k=1}^m s_k * (\widehat{f_k} - \widehat{f_{k+1}})\Bigr\}.
\end{multline}
Indeed, under the notation of \eqref{eqn:aux61}, it suffices to substitute
$$
t_k = s_k, \qquad g = -\sum_{k=1}^m s_k * \widehat{g_k} \qquad \textup{and} \qquad h_k = f_k - f_{k+1} + g_k + \epsi_k
$$
to obtain the expression as is \eqref{eqn:aux60} (observe that $f = \sum_{k=1}^m s_k * (\widehat{f_k} - \widehat{f_{k+1}})$
iff $f = \sum_{k=1}^m f_k \chi_{[s_{k-1},s_k)}$; and $\|f - g - \sum_{j=1}^m s_j * \widehat{h_j}\|_q = \sum_{j=1}^m
(s_j - s_{j-1}) q(\sum_{k=j}^m \epsi_k)$). Conversely, in \eqref{eqn:aux60} we may assume that $t_1,\ldots,t_n$ are
positive and different and $h_j \neq 0$. Consequently, we may assume that $0 < t_1 < \ldots < t_n < \infty$. Then take
$0 = s_0 < \ldots < s_m < \infty$ such that $f = \sum_{k=1}^m f_k \chi_{[s_{k-1},s_k)}$, $g = \sum_{k=1}^m v_k
\chi_{[s_{k-1},s_k)}$ ($v_k \in K$) and $\sum_{j=1}^n t_j * \widehat{h_j} = \sum_{k=1}^m u_k \chi_{[s_{k-1},s_k)}$. Put
$f_{m+1} = g_{m+1} = u_{m+1} = 0$, $g_k = v_{k+1} - v_k \in K$ and $\epsi_k = u_k - u_{k+1} - (f_k - f_{k+1}) - g_k$ and
check that $\sum_{j=1}^n \kappa(t_j) q(h_j) = \sum_{k=1}^m \kappa(s_k) q(u_k - u_{k+1})$, $\|f - g - \sum_{j=1}^n t_j *
\widehat{h_j}\|_q = \sum_{j=1}^m (s_j - s_{j-1}) q(f_j - v_j - u_j)$ and $f_j - v_j - u_j = -\sum_{k=j}^m \epsi_k$. This
proves the equivalence of \eqref{eqn:aux60} and \eqref{eqn:aux61} (the details are left for the reader).\par
So, if $f \in L_0[K]$, the inequality `$\|f\| \geqsl \|f\|_0$' is equivalent to
\begin{multline*}
\sum_{k=1}^m \kappa(s_k) q(f_k - f_{k+1} + g_k + \epsi_k) + A \sum_{j=1}^m (s_j - s_{j-1}) q(\sum_{k=j}^m \epsi_k) \geqsl\\
\geqsl \|\sum_{k=1}^m s_k * (\widehat{f_k} - \widehat{f_{k+1}})\|_0 - \|\sum_{k=1}^m s_k * \widehat{g_k}\|_0
\end{multline*}
(under the notation of \eqref{eqn:aux61}; in that case
$f_1,\ldots,f_m \in K$). So, substituting $h_k = f_k - f_{k+1} + g_k \in K$ and $\Delta s_j = s_j - s_{j-1}$, the latter
inequality follows from
\begin{equation}\label{eqn:aux62}
A \sum_{j=1}^m (\Delta s_j) q(\sum_{k=j}^m \epsi_k) + \sum_{k=1}^m \kappa(s_k) q(h_k + \epsi_k) \geqsl \|\sum_{k=1}^m s_k *
\widehat{h_k}\|_0.
\end{equation}
We shall prove \eqref{eqn:aux62} by induction on $m$.\par
When $m=1$, we have to show that $A s q(\epsi) + \kappa(s) q(h + \epsi) \geqsl \|s * \widehat{h}\|_0$. If $\epsi = 0$, this
follows from the fact that $\|\cdot\|_0$ is a $\kappa$-seminorm. And if $\epsi \neq 0$, we have
$$
A s q(\epsi) + \kappa(s) q(h + \epsi) \geqsl A\mu s \geqsl \lambda |s - 0| \geqsl \|s * \widehat{h}\|_0.
$$
Now assume that \eqref{eqn:aux62} is satisfied for $m-1$. We distinguish between three cases.\par
Firstly, assume that $\sum_{k=j}^m \epsi_k \neq 0$ for each $j \in \{1,\ldots,m\}$. In that case the left-hand side
expression of \eqref{eqn:aux62} is no less than $A \mu \sum_{j=1}^m \Delta s_j$ and
\begin{multline*}
A \mu \sum_{j=1}^m \Delta s_j \geqsl \lambda \sum_{j=1}^m |s_j - s_{j-1}| \geqsl\\ \geqsl \|\sum_{j=1}^m [s_j *
(\sum_{k=j}^m \widehat{h_k}) - s_{j-1} * (\sum_{k=j}^m \widehat{h_k})]\|_0 = \|\sum_{k=1}^m s_k * \widehat{h_k}\|_0,
\end{multline*}
which gives \eqref{eqn:aux62}. Secondly, if $\epsi_m = 0$, the assertion follows from \eqref{eqn:aux62} for `$m-1$'.\par
Thirdly, assume that $\epsi_m \neq 0$ and there is $z \in \{1,\ldots,m\}$ such that $\sum_{k=z}^m \epsi_k = 0$. We may
assume that $z$ is the largest natural number with this property. This yields that $z < m$ and
\begin{equation}\label{eqn:aux63}
\sum_{k=z}^m \epsi_k = 0 \neq \sum_{k=z+1}^m \epsi_k.
\end{equation}
Let us define
\begin{equation}\label{eqn:s-e-h}
\begin{cases}
s_0',\ldots,s_{m-1}' \equiv s_0,\ldots,s_{z-1},s_{z+1},\ldots,s_m,&\\
\epsi_1',\ldots,\epsi_{m-1}' \equiv \epsi_1,\ldots,\epsi_{z-1},\epsi_z + \epsi_{z+1},\epsi_{z+2},\ldots,\epsi_m,&\\
h_1',\ldots,h_{m-1}' \equiv h_1,\ldots,h_{z-1},h_z + h_{z+1},h_{z+2},\ldots,h_m.
\end{cases}
\end{equation}
We infer from induction hypothesis that
\begin{equation}\label{eqn:aux64}
A \sum_{j=1}^{m-1} (\Delta s_j') q(\sum_{k=j}^{m-1} \epsi_k') + \sum_{k=1}^{m-1} \kappa(s_k') q(h_k' + \epsi_k') \geqsl
\|\sum_{k=1}^{m-1} s_k' * \widehat{h_k'}\|_0.
\end{equation}
It is easily verified, thanks to \eqref{eqn:aux63} and \eqref{eqn:L}, that
\begin{multline}\label{eqn:aux67}
A \sum_{j=1}^m (\Delta s_j) q(\sum_{k=j}^m \epsi_k) =\\ = A \sum_{j=1}^{m-1} (\Delta s_j') q(\sum_{k=j}^{m-1} \epsi_k')
+ A (\Delta s_{z+1}) q(\sum_{k=z+1}^m \epsi_k)
\end{multline}
and
\begin{multline}\label{eqn:aux68}
A (\Delta s_{z+1}) q(\sum_{k=z+1}^m \epsi_k) \geqsl (\lambda + ML) (\Delta s_{z+1}) \geqsl\\
\geqsl \|s_{z+1} * \widehat{h_z} - s_z * \widehat{h_z}\|_0 + q(h_z + \epsi_z) [\kappa(s_{z+1}) - \kappa(s_z)].
\end{multline}
Further,
\begin{multline}\label{eqn:aux69}
\sum_{k=1}^m \kappa(s_k) q(h_k + \epsi_k) = \sum_{k=1}^{m-1} \kappa(s_k') q(h_k' + \epsi_k')\\ + [\kappa(s_z)
q(h_z + \epsi_z) + \kappa(s_{z+1}) q(h_{z+1} + \epsi_{z+1}) - \kappa(s_z') q(h_z' + \epsi_z')]
\end{multline}
and
\begin{equation}\label{eqn:aux70}
[q(h_z + \epsi_z) + q(h_{z+1} + \epsi_{z+1})] \kappa(s_{z+1}) - \kappa(s_z') q(h_z' + \epsi_z') \geqsl 0.
\end{equation}
So, \eqref{eqn:aux67}, \eqref{eqn:aux68}, \eqref{eqn:aux69} and \eqref{eqn:aux70} yield
\begin{multline*}
A \sum_{j=1}^m (\Delta s_j) q(\sum_{k=j}^m \epsi_k) + \sum_{k=1}^m \kappa(s_k) q(h_k + \epsi_k) \geqsl
A \sum_{j=1}^{m-1} (\Delta s_j') q(\sum_{k=j}^{m-1} \epsi_k')\\ + \sum_{k=1}^{m-1} \kappa(s_k') q(h_k' + \epsi_k')
+ \|s_{z+1} * \widehat{h_z} - s_z * \widehat{h_z}\|_0
\end{multline*}
and hence \eqref{eqn:aux62} follows from \eqref{eqn:aux64}.\par
We now pass to the proof that $\|\widehat{x}\| \geqsl q(x)$ for $x \in H \setminus \{0\}$. Under the notation
of \eqref{eqn:aux61}, when $f = \widehat{x}$, there is $l \in \{1,\ldots,m\}$ for which $s_l = 1$ and then $f_j = x$
for $j=1,\ldots,l$ and $f_j = 0$ otherwise. Hence, \eqref{eqn:aux61} has the form (recall that $\kappa(1) = 1$):
\begin{multline*}
A\sum_{j=1}^m (\Delta s_j) q(\sum_{k=j}^m \epsi_k) + \sum_{k \neq l} \kappa(s_k) q(g_k + \epsi_k) + \|\sum_{k=1}^m s_k *
\widehat{g_k}\|_0\\ + q(x + g_l + \epsi_l) \geqsl q(x),
\end{multline*}
which is equivalent, after replacing $g_j$ by $h_j$, to
\begin{multline}\label{eqn:aux65}
A\sum_{j=1}^m (\Delta s_j) q(\sum_{k=j}^m \epsi_k) + \sum_{k \neq l} \kappa(s_k) q(h_k + \epsi_k)\\ + \|\sum_{k=1}^m s_k *
\widehat{h_k}\|_0 \geqsl q(h_l + \epsi_l)
\end{multline}
(where $h_j \in K$). As before, we shall prove \eqref{eqn:aux65} by induction on $m$. When $m=1$, also $l = 1$ and hence
we need to show that $$A q(\epsi) + \|\widehat{h}\|_0 \geqsl q(h + \epsi)$$ which easily follows since $A \geqsl 1$ and
$\|\widehat{h}\|_0 = q(h)$ for $h \in K$.\par
Suppose \eqref{eqn:aux65} is fulfilled for $m-1$. We distinguish between several cases. When $\sum_{k=j}^m \epsi_k \neq 0$
for each $j$, the left-hand side expression of \eqref{eqn:aux65} is no less than $A \mu s_m \geqsl M s_l \geqsl
q(h_l + \epsi_l)$.\par
If $\epsi_m = 0$ and $m \neq l$, \eqref{eqn:aux65} follows from induction hypothesis. Now assume that $\epsi_l = 0$.
Let $s_j'$, $h_j'$, $\epsi_j'$ ($j=1,\ldots,m-1$) be systems obtained from $s_j$, $h_j$, $\epsi_j$ by erasing the $l$-th
members. Making use of \eqref{eqn:aux62}, we see that
\begin{multline*}
A \sum_{j=1}^m (\Delta s_j) q(\sum_{k=j}^m \epsi_k) + \sum_{k \neq l} \kappa(s_k) q(h_k + \epsi_k) + \|\sum_{k=1}^m s_k *
\widehat{h_k}\|_0 =\\ A \sum_{j=1}^{m-1} (\Delta s_j') q(\sum_{k=j}^{m-1} \epsi_k') + \sum_{k=1}^{m-1} \kappa(s_k')
q(h_k' + \epsi_k') + \|\sum_{k=1}^{m-1} s_k' * \widehat{h_k'} + s_l * \widehat{h_l}\|_0\\
\geqsl \|\sum_{k=1}^{m-1} s_k' * \widehat{h_k'}\|_0 + \|\sum_{k=1}^{m-1} s_k' * \widehat{h_k'} + \widehat{h_l}\|_0
\geqsl \|\widehat{h_l}\|_0 = q(h_l + \epsi_l).
\end{multline*}
So, we may now assume that $\epsi_m \neq 0$ and there is $z$ such that $\sum_{k=z}^m \epsi_k = 0$. Again, let $z$ be
the largest natural number with this property. We conclude from this that $z < m$ and \eqref{eqn:aux63}. We shall
separately consider the cases when $z \notin \{l-1,l\}$, $z = l-1$ or $z = l$.\par
When $z \neq l-1,l$, define systems $s_j'$, $\epsi_j'$ and $h_j'$ as in \eqref{eqn:s-e-h}, and let $l' \in
\{1,\ldots,m-1\}$ corresponds to $l$. Note that \eqref{eqn:aux67}, \eqref{eqn:aux68} and \eqref{eqn:aux70} are fulfilled.
Moreover, instead of \eqref{eqn:aux69} we have
\begin{multline*}
\sum_{k \neq l} \kappa(s_k) q(h_k + \epsi_k) = \sum_{k\neq z,z+1,l} \kappa(s_k) q(h_k + \epsi_k)\\
+ [\kappa(s_z) q(h_z + \epsi_z) + \kappa(s_{z+1}) q(h_{z+1} + \epsi_{z+1})] = \sum_{\substack{k<m\\k\neq l'}} \kappa(s_k')
q(h_k' + \epsi_k')\\ + [\kappa(s_z) q(h_z + \epsi_z) + \kappa(s_{z+1}) q(h_{z+1} + \epsi_{z+1}) - \kappa(s_z')
q(h_z' + \epsi_z')].
\end{multline*}
The reader will now easily check that \eqref{eqn:aux65} is satisfied, thanks to the induction hypothesis (since $h_{l'}'
+ \epsi_{l'}' = h_l + \epsi_l$).\par
If $z = l-1$ ($l>1$), proceed in the same way. Here $l' = l-1$ and instead of \eqref{eqn:aux69} one has
\begin{multline*}
\sum_{k \neq l} \kappa(s_k) q(h_k + \epsi_k) = \sum_{k\neq l-1,l} \kappa(s_k) q(h_k + \epsi_k)
+ \kappa(s_{l-1}) q(h_{l-1} + \epsi_{l-1})\\ = \sum_{\substack{k<m\\k\neq l'}} \kappa(s_k')
q(h_k' + \epsi_k') + \kappa(s_{l-1}) q(h_{l-1} + \epsi_{l-1}).
\end{multline*}
Combination of induction hypothesis, \eqref{eqn:aux67}, \eqref{eqn:aux68} and the above inequality reduces
\eqref{eqn:aux65} to
$$
q(h_{l'}' + \epsi_{l'}') + q(h_{l-1} + \epsi_{l-1}) \geqsl q(h_l + \epsi_l)
$$
which is fulfilled because $h_{l'}' = h_{l-1} + h_l$ and $\epsi_{l'}' = \epsi_{l-1} + \epsi_l$.\par
It remains to check what happens if $z = l$. From the maximality of $z$ we infer that $\sum_{k=j} \epsi_k \neq 0$ for
$j > l$ and thus
\begin{multline}\label{eqn:aux71}
A \sum_{j=l+1}^m (\Delta s_j) q(\sum_{k=j}^m \epsi_k) \geqsl A\mu \sum_{j=l+1}^m (s_j - s_{j-1}) \geqsl\\
\geqsl \lambda \sum_{j=l+1}^m |s_j - s_{j-1}| \geqsl \|\sum_{j=l+1}^m [s_j * (\sum_{k=j}^m \widehat{h_k}) - s_{j-1} *
(\sum_{k=j}^m \widehat{h_k})]\|_0 =\\= \|\sum_{k=l+1}^m [s_k * \widehat{h_k} - s_l * \widehat{h_k}]\|_0
= \|\sum_{k=l+1}^m s_k * \widehat{h_k} - \sum_{k=l+1}^m \widehat{h_k}\|_0.
\end{multline}
Further, it follows from \eqref{eqn:aux62} that
$$
A \sum_{j=1}^{l-1} (\Delta s_j) q(\sum_{k=j}^{l-1} \epsi_k) + \sum_{k=1}^{l-1} \kappa(s_k) q(h_k + \epsi_k) \geqsl
\|\sum_{k=1}^{l-1} s_k * \widehat{h_k}\|_0.
$$
This, combined with \eqref{eqn:aux71}, gives
\begin{multline*}
A\sum_{j=1}^m (\Delta s_j) q(\sum_{k=j}^m \epsi_k) + \sum_{k \neq l} \kappa(s_k) q(h_k + \epsi_k) + \|\sum_{k=1}^m s_k *
\widehat{h_k}\|_0 =\\= A \sum_{j=1}^{l-1} (\Delta s_j) q(\sum_{k=j}^{l-1} \epsi_k) + \sum_{k=1}^{l-1} \kappa(s_k) q(h_k +
\epsi_k) + \sum_{k=l+1}^m \kappa(s_k) q(h_k + \epsi_k)\\ + A \sum_{j=l+1}^m (\Delta s_j) q(\sum_{k=j}^m \epsi_k) +
\|\sum_{k=1}^m s_k * \widehat{h_k}\|_0 \geqsl \|\sum_{k=1}^{l-1} s_k * \widehat{h_k}\|_0\\ + \sum_{k=l+1}^m s_k q(h_k +
\epsi_k) + \|\sum_{k=l+1}^m s_k * \widehat{h_k} - \sum_{k=l+1}^m \widehat{h_k}\|_0 + \|\sum_{k=1}^m s_k * \widehat{h_k}\|_0
\geqsl \\ \geqsl \sum_{k=l+1}^m q(h_k + \epsi_k) + \|s_l * \widehat{h_l} + \sum_{k=l+1}^m \widehat{h_k}\|_0 \geqsl \\
\geqsl q(\sum_{k=l+1}^m h_k + \sum_{k=l+1}^m \epsi_k) + q(\sum_{k=l}^m h_k) \geqsl q(h_l - \sum_{k=l+1}^m \epsi_k)
\end{multline*}
but $-\sum_{k=l+1}^m \epsi_k = \epsi_l$ (by \eqref{eqn:aux63}), which finally finishes the proof.
\end{proof}

As a corollary of the above result, we obtain

\begin{thm}{g-PV}
Let $(H,+,q) \in \Gg_r(N)$ be a finite valued group, $K$ be a subgroup of $H$, $(G,+,*,\|\cdot\|_G)$ be a Lipschitz
$\kappa$-normed pseudovector group such that $(G,+,\|\cdot\|_G) \in \Gg_r(N)$ and \textup{(AX)} is fulfilled. And let
$\varphi\dd K \to G$ be an isometric group homomorphism. There is a Lipschitz $\kappa$-normed PV group
$(\bar{G},+,*,\|\cdot\|) \supset (G,+,*,\|\cdot\|_G)$ and an isometric group homomorphism $\psi\dd H \to \bar{G}$ such that
$(\bar{G},+,\|\cdot\|) \in \Gg_r(N)$ and $\psi\bigr|_K = \varphi$.
\end{thm}
\begin{proof}
Let $\widehat{\varphi}\dd L_0[K] \to G$ be the linear homomorphism induced by $\varphi$. Define $\|\cdot\|_0\dd L_0[K] \to
\RRR_+$ by $\|f\|_0 = \|\widehat{\varphi}(f)\|_G$. Then $\|\cdot\|_0$ is a Lipschitz $\kappa$-seminorm on $L_0[K]$ such
that $\|\widehat{h}\|_0 = q(h)$ for $h \in K$ (since $\varphi$ is isometric). So, according to \THM{g-PVG}, there is
a Lipschitz $\kappa$-seminorm $\|\cdot\|_H$ on $L_0[H]$ which extends $\|\cdot\|_0$ and satisfies $\|\widehat{h}\|_H =
q(h)$ for $h \in H$. Now let $\widetilde{H} = L_0[H]/\{f \in L_0[H]\dd\ \|f\|_H = 0\}$ be the PV group equipped with
the Lipschitz $\kappa$-norm induced by $\|\cdot\|_H$, let $\pi\dd L_0[H] \to \widetilde{H}$ be the quotient linear
homomorphism and $\widetilde{K} = \pi(L_0[K])$. Observe that there is an isometric linear homomorphism
$\widetilde{\varphi}\dd \widetilde{K} \to G$ such that $\widetilde{\varphi} \circ \pi\bigr|_{L_0[K]} = \widehat{\varphi}$.
To this end, apply \LEM{APV} to $\widetilde{\varphi}$ and the inclusion map of $\widetilde{K}$ into $\widetilde{H}$
in order to obtain $\bar{G}$. Finally, involve \LEM{bd} and (P5) if needed.
\end{proof}

\begin{lem}{N-PV}
Suppose $a \in \CcC(G) \setminus \{0\}$ is of finite rank. For each $\epsi > 0$ there is a PV group
$(\bar{G},+,*,\|\cdot\|) \supset (G,+,*,\|\cdot\|_G)$ and $b \in \LlL(\bar{G})$ such that $\rank(b) = \rank(a)$ and
$\|a - b\| \leqsl \epsi$.
\end{lem}
\begin{proof}
Let $H = \grp{a}$. The homomorphism $\ZZZ \ni k \mapsto ka \in H \subset G$ induces linear homomorphisms $\psi_G\dd
L_0[\ZZZ] \to G$ and $\psi_H\dd L_0[\ZZZ] \to L_0[H]$ such that $\psi_G(\widehat{1}) = a$ and $\psi_H(\widehat{1}) =
\widehat{a}$. Let $\delta$ be the discrete value on $H$ and $N \geqsl 2$ be the rank of $a$. We equip $L_0[\ZZZ]$ with
the $\kappa$-norm $\|\cdot\|_{\kappa}$ defined as follows. When $f \in L_0[\ZZZ] \setminus \{0\}$, there are unique
$n \geqsl 1$, $0 < t_1 < \ldots < t_n < \infty$ and $k_1,\ldots,k_n \in \ZZZ \setminus \{0\}$ such that $f = \sum_{j=1}^n
t_j * \widehat{k_j}$. We put $\|f\|_{\kappa} = \sum_{j=1}^n \kappa(t_j)$. Further, let $M = M(a,\epsi/N)$ be
as in \LEM{constM}. Put $A = NM$, $\bar{G} = G \times L_0[H]$ and define a $\kappa$-seminorm $\|\cdot\|$ on $\bar{G}$ by
$$
\|(g,f)\| = \inf\{\|g - \psi_G(u)\|_G + \epsi \|u\|_{\kappa} + A \|\psi_H(u) - f\|_{\delta}\dd\ u \in L_0[\ZZZ]\}.
$$
It is easily seen that $\|\cdot\|$ is a $\kappa$-seminorm on $\bar{G}$ such that $\|(g,f)\| > 0$ for $f \neq 0$ and
$\|(g,0)\| \leqsl \|g\|_G$ for every $g \in G$; $\|(a,\widehat{a})\| \leqsl \epsi$ and $\|(0,f)\| \leqsl A \|f\|_{\delta}$
for $f \in L_0[H]$. The latter will imply that $(0,\widehat{a}) \in \LlL(\bar{G})$. So, if we show that $\|(g,0)\| \geqsl
\|g\|$, $\|\cdot\|$ will be a $\kappa$-norm and the proof will be completed. The latter is equivalent to
\begin{equation}\label{eqn:aux75}
\epsi \|u\|_{\kappa} + A \|\psi_H(u)\|_{\delta} \geqsl \|\psi_G(u)\|_G
\end{equation}
where $u \in L_0[\ZZZ]$. We may assume that $u \neq 0$. Writing $u = \sum_{k=1}^m s_k * \widehat{l_k}$ with $0 = s_0 <
\ldots < s_m$ and $l_1,\ldots,l_m \in \ZZZ \setminus \{0\}$, and putting $\nu_j = \sum_{k=j}^m l_k$ we obtain
$\|u\|_{\kappa} = \sum_{k=1}^m \kappa(s_k)$, $u = \sum_{k=1}^m (s_k * \widehat{\nu_k} - s_{k-1} * \widehat{\nu_k})$,
$\psi_H(u) = \sum_{k=1}^m (s_k * \widehat{\nu_k a} - s_{k-1} * \widehat{\nu_k a})$ and $\psi_G(u) = \sum_{k=1}^m [s_k *
(\nu_k a) - s_{k-1} * (\nu_k a)]$. Consequently, $\|\psi_H(u)\|_{\delta} = \sum_{k=1}^m (s_k - s_{k-1})\delta(\nu_k a)$ and
\begin{multline*}
\epsi \|u\|_{\kappa} + A \|\psi_H(u)\|_{\delta} - \|\psi_G(u)\|_G \geqsl\\ \sum_{k=1}^m [\epsi \kappa(s_k)
+ A (s_k - s_{k-1}) \delta(\nu_k a) - \|s_k * (\nu_k a) - s_{k-1} * (\nu_k a)\|_G].
\end{multline*}
So, \eqref{eqn:aux75} will be fulfilled if only
$$
\|s * (ma) - t * (ma)\|_G \leqsl A |s - t| \delta(ma) + \epsi \kappa(s \vee t)
$$
for every $s,t \in \RRR_+$ and $m \in \ZZZ$. Since $\rank(a) = N$, we may assume that $0 < m < N$. But then,
by the definition of $M$, $\|s * (ma) - t * (ma)\|_G \leqsl m \|s * a - t * a\|_G \leqsl N [M |s - t| + (\epsi/N)
\kappa(s \vee t)] = A |s - t| \delta(ma) + \epsi \kappa(s \vee t)$ and we are done.
\end{proof}

\begin{lem}{00-PV}
Suppose $a \in \CcC(G) \setminus \{0\}$ is such that $\lim_{n\to\infty} \|na\|_G/n = 0$. For every $\epsi > 0$ there is
a PV group $(\bar{G},+,*,\|\cdot\|) \supset (G,+,*,\|\cdot\|_G)$ and an element $b \in \LlL(\bar{G})_{fin}$
such that $\|a - b\| \leqsl \epsi$.
\end{lem}
\begin{proof}
Take $N \geqsl 2$ such that for every $n \geqsl N$,
\begin{equation}\label{eqn:aux80}
\frac{\|n a\|_G}{n} \leqsl \frac{\epsi}{2}.
\end{equation}
Further, let $M = M(a,\epsi /(2 N))$ be as in \LEM{constM}. Put $A = NM$. Let $H = \grp{b}$ be a cyclic group of rank $N$
and $\delta$ denote the discrete value on $H$. Now put $\bar{G} = G \times L_0[H]$ and define $\|\cdot\|\dd \bar{G} \to
\RRR_+$ by
\begin{multline*}
\|(g,f)\| = \inf\{\|g - \sum_{j=1}^m s_j * (l_j a)\|_G + \epsi \sum_{j=1}^m \kappa(s_j) + A \|\sum_{j=1}^m s_j *
(l_j \widehat{b}) - f\|_{\delta}\dd\\ m \geqsl 0,\ s_1,\ldots,s_m \geqsl 0,\ l_1,\ldots,l_m \in \{-1,1\}\}.
\end{multline*}
As in the previous proof, observe that $\|\cdot\|$ is a $\kappa$-seminorm such that $\|(g,f)\| > 0$ provided $f \neq 0$,
$\|(g,0)\| \leqsl \|g\|_G$, $\|(0,f)\| \leqsl A \|f\|_{\delta}$ ($g \in G$, $f \in L_0[H]$) and $\|(a,\widehat{b})\|
\leqsl \epsi$. So, if only we show that $\|(g,0)\| \geqsl \|g\|_G$ for $g \in G$, the proof will be completed (because then
$\|\cdot\|$ will be a $\kappa$-norm and $\widehat{b} \in \LlL(\bar{G})$). The latter is equivalent to
\begin{equation}\label{eqn:aux81}
\epsi \sum_{k=1}^m \kappa(s_k) + A \|\sum_{k=1}^m s_k * (l_k \widehat{b})\|_{\delta} \geqsl \|\sum_{k=1}^m s_k *
(l_k a)\|_G
\end{equation}
for every $m \geqsl 0$, $s_1,\ldots,s_m \in \RRR_+$ and $l_1,\ldots,l_m \in \{-1,1\}$. We may assume that $m \geqsl 1$
and then, when $m$ and $l_1,\ldots,l_m$ are fixed, both the left-hand and the right-hand side expressions
of \eqref{eqn:aux71} are continuous with respect to $s_1,\ldots,s_m$. Hence we may assume that $s_1,\ldots,s_m$ are
positive and different. Finally, after renumeration, we may have $0 = s_0 < s_1 < \ldots < s_m$. But then $\|\sum_{k=1}^m
s_k * (l_k \widehat{b})\|_{\delta} = \sum_{j=1}^m (\Delta s_j) \delta(\sum_{k=j}^m l_k b)$ (where, as usually, $\Delta s_j
= s_j - s_{j-1}$). So, \eqref{eqn:aux81} changes into
\begin{multline}\label{eqn:aux82}
\frac{\epsi}{2} \sum_{j=1}^m \kappa(s_j) + \sum_{j=1}^m [A (\Delta s_j) \delta(\sum_{k=j}^m l_k b) + \frac{\epsi}{2}
\kappa(s_j)] \geqsl\\ \geqsl \|\sum_{j=1}^m s_j * (l_j a)\|_G.
\end{multline}
Now we shall construct a system of natural numbers $\{\nu_0,\ldots,\nu_z\}$ (for some $z \geqsl 1$) for which
\begin{enumerate}[(E1)]\addtocounter{enumi}{-1}
\item $0 = \nu_z < \ldots < \nu_0 = m+1$,
\item for every $j,k$ with $1 \leqsl k \leqsl z$ and $\nu_k < j < \nu_{k-1}$,
   $$|\sum_{s=j}^{\nu_{k-1}-1} l_s| < N,$$
\item for every $k$ with $1 \leqsl k < z$, $$|\sum_{s=\nu_k}^{\nu_{k-1}-1} l_s| = N.$$
\end{enumerate}
Put $\nu_0 = m+1$ and suppose $\nu_p$ is defined for some $p \geqsl 0$. If either $\nu_p = 1$ or $|\sum_{s=j}^{\nu_p - 1}
l_s| < N$ for every $j \in \{1,\ldots,\nu_p - 1\}$, put $\nu_{p+1} = 0$ and $z = p+1$ and finish the construction.
Otherwise, take the greatest natural number $\nu_{p+1} \in \{1,\ldots,\nu_p - 1\}$ such that $|\sum_{s=\nu_{p+1}}^{\nu_p-1}
l_s| \geqsl N$. Since $|l_s| = 1$ for every $s$, one has $|\sum_{s=\nu_{p+1}}^{\nu_p-1} l_s| = N$. The verification that
(E0)--(E2) are fulfilled is left as an exercise.\par
For simplicity, let $l_0 = 0$. Now it follows
from (E2) that
\begin{enumerate}[(E1)]\addtocounter{enumi}{2}
\item for every $j,k$ with $0 \leqsl k < z$ and $\nu_{k+1} \leqsl j < \nu_k$,
   $$\sum_{s=j}^m l_s b = \sum_{s=j}^{\nu_k-1} l_s b$$
\end{enumerate}
(recall that $\rank(b) = N$). Thanks to (E3), we have
\begin{multline}\label{eqn:aux83}
\sum_{j=1}^m [A (\Delta s_j) \delta(\sum_{k=j}^m l_k b) + \frac{\epsi}{2} \kappa(s_j)] \geqsl\\
\geqsl \sum_{k=0}^{z-1} \sum_{j=\nu_{k+1}+1}^{\nu_k-1} [A (\Delta s_j) \delta(\sum_{q=j}^{\nu_k-1} l_q b)
+ \frac{\epsi}{2} \kappa(s_j)].
\end{multline}
Further, when $0 \leqsl k < z$,
\begin{multline*}
\sum_{j=\nu_{k+1}}^{\nu_k-1} s_j * (l_j a) = \sum_{j=\nu_{k+1}}^{\nu_k-1} s_j * (\sum_{q=j}^{\nu_k-1} l_q a
- \sum_{q=j+1}^{\nu_k-1} l_q a) =\\= \sum_{j=\nu_{k+1}}^{\nu_k-1} s_j * (\sum_{q=j}^{\nu_k-1} l_q a)
- \sum_{j=\nu_{k+1}+1}^{\nu_k} s_{j-1} * (\sum_{q=j}^{\nu_k-1} l_q a) =\\
= \sum_{j=\nu_{k+1}+1}^{\nu_k-1} (\sum_{q=j}^{\nu_k-1} l_q) (s_j *  a - s_{j-1} * a) + s_{\nu_{k+1}} *
(\sum_{q=\nu_{k+1}}^{\nu_k-1} l_q a)
\end{multline*}
and therefore (by (E2), \eqref{eqn:aux80} and the fact that $s_{\nu_z} = 0$):
\begin{multline*}
\|\sum_{j=\nu_{k+1}}^{\nu_k-1} s_j * (l_j a)\|_G = \|\sum_{k=0}^{z-1} \sum_{j=\nu_{k+1}}^{\nu_k-1} s_j * (l_j a)\|_G
\leqsl \\ \leqsl \sum_{k=0}^{z-1} \sum_{j=\nu_{k+1}+1}^{\nu_k-1} \Bigl|\sum_{q=j}^{\nu_k-1} l_q\Bigr| \cdot \|s_j * a
- s_{j-1} * a\|_G + \sum_{k=1}^{z-1} \kappa(s_{\nu_k}) \|N a\|_G \leqsl\\ \leqsl \sum_{k=0}^{z-1}
\sum_{j=\nu_{k+1}+1}^{\nu_k-1} \Bigl|\sum_{q=j}^{\nu_k-1} l_q\Bigr| \cdot \|s_j * a - s_{j-1} * a\|_G + \frac{\epsi}{2} N
\sum_{k=1}^{z-1} \kappa(s_{\nu_k}).
\end{multline*}
So, taking into account \eqref{eqn:aux83}, we see that \eqref{eqn:aux82} will be satisfied provided
\begin{equation}\label{eqn:aux84}
\Bigl|\sum_{q=j}^{\nu_k-1} l_q\Bigr| \cdot \|s_j * a - s_{j-1} * a\|_G \leqsl A (\Delta s_j) \delta(\sum_{q=j}^{\nu_k-1}
l_q b) + \frac{\epsi}{2} \kappa(s_j)
\end{equation}
whenever $\nu_{k+1} < j < \nu_k$ ($k=0,\ldots,z-1$) and
\begin{equation}\label{eqn:aux85}
N \sum_{k=1}^{z-1} \kappa(s_{\nu_k}) \leqsl \sum_{j=1}^m \kappa(s_j).
\end{equation}
To prove \eqref{eqn:aux84}, put $\lambda = \Bigl|\sum_{q=j}^{\nu_k-1} l_q\Bigr|$, $t = s_j$ and $s = s_{j-1}$. By (E1),
$0 \leqsl \lambda < N$. When $\lambda = 0$, \eqref{eqn:aux84} is clear. And if $\lambda \neq 0$, $\delta(\lambda b) = 1$
and then, by the definition of $M$,
$$
\lambda \|t * a - s * a\|_G \leqsl N (M |t - s| + \frac{\epsi}{2N} \kappa(t \vee s)) = A |t - s| \delta(\lambda b)
+ \frac{\epsi}{2} \kappa(t)
$$
which gives \eqref{eqn:aux84}.\par
Now we pass to \eqref{eqn:aux85}. By (E2), $N \leqsl \nu_{k-1} - \nu_k$ for $k=1,\ldots,z-1$. So, by (NF3) we obtain
$$
N \sum_{k=1}^{z-1} \kappa(s_{\nu_k}) \leqsl \sum_{k=1}^{z-1} \sum_{j=\nu_k}^{\nu_{k-1}-1} \kappa(s_{\nu_k})
\leqsl \sum_{k=1}^{z-1} \sum_{j=\nu_k}^{\nu_{k-1}-1} \kappa(s_j) \leqsl \sum_{j=1}^m \kappa(s_j)
$$
which finishes the proof.
\end{proof}

As an immediate consequence of (P5) and Lemmas \ref{lem:bd}, \ref{lem:N-PV} and \ref{lem:00-PV} we obtain

\begin{thm}{GrN-PV}
If $(G,+,*,\|\cdot\|_G)$ is a $\kappa$-normed topological pseudovector group such that $(G,+,\|\cdot\|_G) \in \Gg_r(N)$,
there is a $\kappa$-normed pseudovector group $(\bar{G},+,*,\|\cdot\|) \supset (G,+,*,\|\cdot\|_G)$ such that
$(\bar{G},+,\|\cdot\|) \in \Gg_r(N)$ and the set $\LlL(\bar{G})_{fin}$ is dense in $\bar{G}$.
\end{thm}

\SECT{Proof of \THM{main3}}

Let $r \in \{1,\infty\}$, $N \in \ZZZ_+ \setminus \{1\}$ and $\kappa$ be a norming function. As it is easily seen,
\THM{main3} immediately follows from

\begin{thm}{3}
Let $(G,+,*,\|\cdot\|_G)$ be a $\kappa$-normed topological pseudovector group such that $(G,+,\|\cdot\|_G) \in \Gg_r(N)$
and \textup{(AX)} is fulfilled. There is a $\kappa$-normed pseudovector group $(\bar{G},+,*,\|\cdot\|) \supset
(G,+,*,\|\cdot\|_G)$ such that the set $\LlL(\bar{G})_{fin}$ is dense in $\bar{G}$ and the valued groups
$\bar{G}$ and $\GGG_r(N)$ are isometrically group isomorphic.
\end{thm}

The above result strengthens and generalizes \cite[Theorem~4.3]{pn2}. The proof of \THM{3} will be preceded by

\begin{lem}{count}
Let $(G',+,*,\|\cdot\|')$ be a Lipschitz $\kappa$-normed pseudovector group such that $(G',+,\|\cdot\|') \in \Gg_r(N)$
and \textup{(AX)} is satisfied. Let $D$ and $\FFF$ be, respectively, a countable subset of $G'$ and a countable
subfield of $\RRR$. For every countable family $\Hh$ of finite valued groups of class $\Gg_r(N)$ there is a Lipschitz
$\kappa$-normed pseudovector group $(G'',+,*,\|\cdot\|'') \supset (G',+,*,\|\cdot\|')$ such that $(G'',+,\|\cdot\|'')
\in \Gg_r(N)$ and with the following property. Whenever $H \in \Hh$, $K$ is a subgroup of $H$ and $\varphi\dd K \to
\lin_{\FFF} D \subset G'$ is an isometric group homomorphism, there is an isometric group homomorphism $\psi\dd H \to
G''$ which extends $\varphi$.
\end{lem}
\begin{proof}
Let $D' = \lin_{\FFF} D$. Since every member of $\Hh$ is a finite group and $D'$ is countable, the family $\{\varphi\dd
K \to D'\}$, where $K$ runs over all subgroups of members of $\Hh$ and $\varphi$ is an isometric group homomorphism, is
countable. So, the assertion follows from induction and \THM{g-PV}.
\end{proof}

\textit{Proof of \THM{3}}.
Thanks to \THM{GrN-PV}, we may and do assume that $A = \LlL(G)_{fin}$ is dense in $G$. Let $A_0$ be a countable
dense subset of $A$. Note that $\lin A_0$ is a dense Lipschitz PV subgroup of $G$ and, as a group, of class $\OOo_{fin}$.
For purpose of this proof, let us call a $\kappa$-normed PV group $(E,+,*,\|\cdot\|_E)$ \textit{of class $\Ll_r(N)_{fin}$}
provided $(E,+,\|\cdot\|_E) \in \Gg_r(N)$ and $E = \LlL(E)_{fin}$.\par
We shall construct, making use of induction, sequences $(\FFF_n)_{n=0}^{\infty}$, $(G_n,+,*,\|\cdot\|_n)$
and $(D_n)_{n=0}^{\infty}$ such that for every $n \geqsl 0$:
\begin{enumerate}[(1$_n$)]
\item $\FFF_n$ is a countable subfield of $\RRR$,
\item $(G_n,+,*,\|\cdot\|_n) \in \Ll_r(N)_{fin}$,
\item $D_n$ is a countable subset of $G_n$ and $G_n = \lin D_n$,
\item $\FFF_0 = \QQQ$ and for $n > 0$,
   $$\FFF_n \supset \FFF_{n-1} \cup \{\|g\|_{n-1}\dd\ g \in \lin_{\FFF_{n-1}} D_{n-1}\},$$
\item $G_0 = \lin A_0$ and $\|\cdot\|_0$ is inherited from $\|\cdot\|_G$; and for $n > 0$,
   $$(G_n,+,*,\|\cdot\|_n) \supset (G_{n-1},+,*,\|\cdot\|_{n-1}) \quad \textup{and} \quad D_n \supset D_{n-1},$$
\item for $n > 0$: whenever $(H,+,q) \in \Gg_r(N)$ is a finite valued group with $q(H) \subset \FFF_{n-1}$ and $\varphi\dd
   K \to \lin_{\FFF_{n-1}} D_{n-1}$ is an isometric group homomorphism of a subgroup $K$ of $H$, there is an isometric
   group homomorphism $\psi\dd H \to \lin_{\FFF_n} D_n$ which extends $\varphi$.
\end{enumerate}
Define $\FFF_0$ and $G_0$ as in (4$_0$) and (5$_0$) and put $D_0 = A_0$. Suppose that for some $n > 0$ we have defined
$\FFF_{n-1}$, $G_{n-1}$ and $D_{n-1}$ with suitable properties. Let $\Hh$ be a countable family of all (up to isometric
group isomorphism) finite $\FFF_{n-1}$-groups of class $\Gg_r(N)$ (cf. \DEF{Qg}). Take a Lipschitz $\kappa$-normed PV
group $(G'',+,*,\|\cdot\|'')$ which witnesses the assertion of \LEM{count} for $\Hh$, $G' = G_{n-1}$, $\FFF = \FFF_{n-1}$
and $D = D_{n-1}$. Further, let $\Ff$ be the family of all pairs $(\varphi,H)$ with $H \in \Hh$ and $\varphi\dd K \to
\lin_{\FFF_{n-1}} D_{n-1}$ an isometric group homomorphisms where $K$ is a subgroup of $H$. The collection $\Ff$ is
countable. For each $(\varphi,H) \in \Ff$ let $\widehat{\varphi}\dd H \to G''$ be an isometric group homomorphism
extending $\varphi$. Now put $D_n = D_{n-1} \cup \bigcup_{(\varphi,H) \in \Ff} \widehat{\varphi}(H)$ and $G_n = \lin D_n
\subset G''$. Let $\|\cdot\|_n$ be the $\kappa$-norm on $G_n$ inherited from $\|\cdot\|''$ and $\FFF_n$ be the subfield
of $\RRR$ generated by $\FFF_{n-1} \cup \{\|g\|_n\dd\ g \in D_n\}$. It is easy to verify that conditions (1$_n$)--(6$_n$)
are fulfilled.\par
Having the sequences $(\FFF_n)_{n=0}^{\infty}$, $(D_n)_{n=0}^{\infty}$ and $(G_n)_{n=0}^{\infty}$, define the PV group
$(\bar{G},+,*,\|\cdot\|)$ as the completion of $\bigcup_{n=0}^{\infty} (G_n,+,*,\|\cdot\|_n)$ (cf. \PRO{compl}), $\FFF =
\bigcup_{n=0}^{\infty} \FFF_n$ and $D = \bigcup_{n=0}^{\infty} D_n$. Note that $$(\bar{G},+,*,\|\cdot\|) \supset
(G,+,*,\|\cdot\|_G)$$ (since $G_0$ is dense in $G$) and $\widehat{D} = \lin_{\FFF} D$ is dense in $\bar{G}$. Further,
$\widehat{D}$ is isometrically group isomorphic to $\FFF\GGG_r(N)$, which follows from the construction. Thus, to this end
it remains to apply \PRO{completion}.\qed
\vspace{0.3cm}

As sonsequences of \THM{3} we obtain (see \PRO{metriz})

\begin{thm}{enl-PV}
Let $r \in \{1,\infty\}$. Every separable metrizable topological pseudovector (Abelian) group is isomorphic
(as a topological PV group) to a pseudovector subgroup of a subnormed topological PV group which is isometrically group
isomorphic to $\GGG_r(0)$.
\end{thm}

\SECT{Topology of $\GGG_r(N)$'s}

This part is devoted to the proof of the following

\begin{thm}{top}
Each of the topological spaces $\GGG_r(N)$'s is homeomorphic to the Hilbert space $l_2$.
\end{thm}

Note that \THM{top} for $N \in \{0,2\}$ follows from \THM{Ur} and the result of Uspenskij \cite{usp}.
However, the proof presented here needs no additional work for such $N$ and thus we give the proof
of \THM{top} in its full generality. We shall do this making use of \THM{main3}.\par
Before we pass to the proof of \THM{top}, we have to recall some notions and results of infinite-dimensional
topology. First of all, recall that the space $l_2$ is the Banach space consisting of all square summable
real sequences equipped with the norm $\|(a_n)_{n=1}^{\infty}\|_2 = (\sum_{n=1}^{\infty} a_n^2)^{1/2}$.\par
A metrizable space $X$ is said to be an \textit{absolute retract} iff it is a retract of any metrizable
space in which $X$ is embedded as a closed set. The space $X$ is \textit{homotopically trivial} provided
every map of the boundary of $[0,1]^n$ (for arbitrary $n \geqsl 1$) into $X$ is extendable to a map
of the whole cube $[0,1]^n$ into $X$. Observe that under such a definition the empty space is homotopically
trivial. A subset $Y$ of the space $X$ is said to be \textit{contractible in $X$} if there is a map $H\dd
Y \times [0,1] \to X$ such that $H(y,1) = y$ for each $y \in Y$ and the map $H(\cdot,0)$ is constant.
An elementary result concerning contractibility and homotopical triviality says that if every compact
nonempty subset of a metrizable space $M$ is contractible in $M$, then $M$ is homotopically trivial.\par
Toru\'{n}czyk in his famous works \cite{to1,to2} has given a characterization of metric spaces which are
homeomorphic to $l_2$. Based on his results, Dobrowolski and Toru\'{n}czyk \cite{d-t} has proved
the following theorem, which will be one of tools of this section.

\begin{thm}{D-T}
A separable completely metrizable topological group is homeomorphic to $l_2$ iff it is a non-locally compact
absolute retract.
\end{thm}

Thus, according to \THM{D-T}, in order to prove \THM{top}, we only need to show that each of the spaces
$\GGG_r(N)$'s is an absolute retract (that $\GGG_r(N)$ is non-locally compact it easily follows from
\THM{univ}). We shall prove this involving a very convenient criterion due to Toru\'{n}czyk \cite{to0}
(cf. \cite[Theorem~?.??]{vMb}) a special case of which is formulated below (compare with the proof
in \cite{usp}).

\begin{thm}{base}
If the intersection of every finite nonempty collection of open balls in a metric space $(X,d)$ with centres
in a given dense subset of $X$ is homotopically trivial, then $X$ is an absolute retract.
\end{thm}

Now we are ready to prove \THM{top}. It turns out that the argument is slightly more complicated in case
of bounded groups than in the unbounded case (the main reason for this is that there are no nontrivial
bounded norms on PV groups). Thus, we divide the proof into these two cases.
\vspace{0.3cm}

\textit{Proof of \THM{top} for $r=\infty$}.
Let $D = \GGG_r(N)_{fin}$ and $p$ be the value of $\GGG_{\infty}(N)$. By (G3), $D$ is dense in $\GGG_{\infty}(N)$. Take
arbitrary points $x_1,\ldots,x_n \in D$ and radii $r_1,\ldots,r_n > 0$. Let $B = \bigcap_{j=1}^n B_p(x_j,r_j)$ where
$B_p(x_j,r_j) = \{z \in \GGG_{\infty}(N)\dd\ p(z - x_j) < r_j\}$. We may assume that $B$ is nonempty. We shall show that
$B$ is contractible (in itself). Since translations $x \mapsto x + a$ ($a \in \GGG_{\infty}(N)$) are (bijective) isometries
on $\GGG_{\infty}(N)$, we may also assume that $0 \in B$. This means that
\begin{equation}\label{eqn:zero}
p(x_j) < r_j \qquad (j=1,\ldots,n).
\end{equation}
Let $H = \grp{x_1,\ldots,x_n}$. Observe that $H$ is a finite group. Denote by $q$ the restriction of $p$
to $H$. Then $(H,+,q) \in \Gg_{\infty}(N)$ and consequently $(L_0[H],+,\|\cdot\|_q) \in \Gg_{\infty}(N)$.
Notice that $\varphi\dd H \ni h \mapsto \widehat{h} \in L_0[H]$ is an isometric group homomorphism.
By \THM{main3}, there is a normed topological PV group $(G,+,*,\|\cdot\|) \supset (L_0[H],+,*,\|\cdot\|_q)$
such that $G$ and $\GGG_{\infty}(N)$ are isometrically group isomorphic. So, thanks to (UEP), there is
an isometric group isomorphism $\psi\dd \GGG_{\infty}(N) \to G$ which extends $\varphi$. Define $H\dd B
\times [0,1] \to \GGG_{\infty}(N)$ by $H(y,t) = \psi^{-1}(t * \psi(y))$. We see that $H$ is continuous and
$H(y,0) = 0$ and $H(y,1) = y$ for $y \in B$. So, it suffices to check that $H(B \times [0,1]) \subset B$.\par
For $y \in B$, $t \in [0,1]$ and $j \in \{1,\ldots,n\}$ we have, by \eqref{eqn:zero},
\begin{multline*}
p(H(y,t) - x_j) = \|t * \psi(y) - \psi(x_j)\| \leqsl\\ \leqsl \|t * \psi(y) - t * \psi(x_j)\|
+ \|t * \psi(x_j) - \psi(x_j)\| =\\
= t \|\psi(y - x_j)\| + \|t * \widehat{x_j} - \widehat{x_j}\|_q = t p(y - x_j) + (1-t) p(x_j) < r_j
\end{multline*}
and we are done.\qed
\vspace{0.3cm}

\textit{Proof of \THM{top} (for arbitrary $r$)}.
We shall improve the previous argument so that it will work also in bounded case. Let $D$; $p$; $x_1,\ldots,x_n \in D$;
$r_1,\ldots,r_n > 0$ and $B$ be as in the previous proof. As there, we may and do assume that \eqref{eqn:zero} is
fulfilled. Let $K$ be a compact subset of $B$. We shall show that $K$ is contractible in $B$.\par
For $j \in \{1,\ldots,n\}$, let $\varrho_j = \max\{p(z - x_j)\dd\ z \in K \cup \{0\}\} < r_j$. Let $\epsi > 0$ be such that
\begin{equation}\label{eqn:aux88}
\varrho_j + 2 \epsi < r_j \qquad (j=1,\ldots,n).
\end{equation}
Further, take $w_1,\ldots,w_l \in D$ for which
\begin{equation}\label{eqn:aux89}
B_p(w_k,\epsi) \cap K \neq \varempty\quad (k=1,\ldots,l)\quad \textup{and}\quad K \subset \bigcup_{k=1}^l B_p(w_k,\epsi).
\end{equation}
The above implies that
\begin{equation}\label{eqn:aux90}
p(w_k - x_j) \leqsl \varrho_j + \epsi \qquad (j \in \{1,\ldots,n\},\ k \in \{1,\ldots,l\}).
\end{equation}
Let $H = \grp{x_1,\ldots,x_n;w_1,\ldots,w_l}$, $q$ denote the restriction of $p$ to $H$ and let $\|\cdot\|_H = \|\cdot\|_q
\wedge r$. Then $(L_0[H],+,*,\|\cdot\|_H)$ is a Lipschitz subnormed PV group such that $(L_0[H],+,\|\cdot\|_H) \in
\Gg_r(N)$ and the function $\varphi\dd H \ni h \mapsto \widehat{h} \in L_0[H]$ is an isometric group homomorphism. Again,
we conclude from \THM{main3} that there is a subnormed topological PV group $(G,+,*,\|\cdot\|) \supset
(L_0[H],+,*,\|\cdot\|_H)$ such that $G$ and $\GGG_r(N)$ are isometrically group isomorphic. Therefore, there is
an isometric group isomorphism $\psi\dd \GGG_r(N) \to G$ which extends $\varphi$. Define $H\dd K \times [0,1] \to
\GGG_r(N)$ as before: $H(y,t) = \psi^{-1}(t * \psi(y))$. We only need to show that $H$ takes values in $B$. To this end,
let $y \in K$, $t \in [0,1]$ and $j \in \{1,\ldots,n\}$. Take $k \in \{1,\ldots,l\}$ such that $p(y - w_k) < \epsi$ and,
using \eqref{eqn:aux88} and \eqref{eqn:aux90}, observe that
\begin{multline*}
p(H(y,t) - x_j) = \|t * \psi(y) - \psi(x_j)\| \leqsl\\
\leqsl \|t * \psi(y - w_k)\| + \|t * \psi(w_k) - \psi(x_j)\|\leqsl\\
\leqsl \nabla(t) \|\psi(y - w_k)\| + \|t * \widehat{w_k} - \widehat{x_j}\| =\\
= p(y - w_k) + t p(w_k - x_j) + (1-t) p(x_j) \leqsl\\
\leqsl \epsi + t (\varrho_j + \epsi) + (1 - t) \varrho_j \leqsl \varrho_j + 2 \epsi < r_j
\end{multline*}
which finishes the proof.\qed

\SECT{Subnormed topological PV groups of class $\OOo_{00}$}

In this section all groups are subnormed topological pseudovector. Observe that if $\|\cdot\|$ is a subnorm on $G$, then
for every $g \in G$ the function $(0,\infty) \ni t \mapsto \|t * g\|/t \in \RRR_+$ is monotone decreasing. Consequently,
there is a finite limit
$$
\|g\|_{*0} = \lim_{t\to\infty} \frac{\|t * g\|}{t}.
$$
The function $\|\cdot\|_{*0}\dd G \to \RRR_+$ is a seminorm on $G$ such that $\|\cdot\|_{*0} \leqsl \|\cdot\|$.
In particular, $G_{*0} = \{g \in G\dd\ \|g\|_{*0} = 0\}$ is a closed PV subgroup of $G$. We call $G$ \textit{of class
$\OOo_{*0}$} iff $\|\cdot\|_{*0} \equiv 0$ or, equivalently, if $G = G_{*0}$.

\begin{dfn}{o00}
A subnormed PV group $G$ is \textit{of class $\OOo_{00}$} if $G$ is of both classes $\OOo_0$ and $\OOo_{*0}$. That is,
$G$ is of class $\OOo_{00}$ if for every $g \in G$,
$$
\lim_{t\to\infty} \frac{\|t * g\|}{t} = \lim_{n\to\infty} \frac{\|n g\|}{n} = 0.
$$
\end{dfn}

We call an element $a$ of $G$ \textit{bounded} provided $\lin \{a\}$ is a (metrically) bounded subset of $G$. Let us denote
by $G_{bd}$ the set of all bounded elements of $G$. Let $\EeE(G) = \LlL(G) \cap G_{fin} \cap G_{bd}$. So, $\EeE(G)$
is a PV subgroup of $G$ consisting of all finite rank elements which are both Lipschitz and bounded. It is easily seen that
$G$ is of class $\OOo_{00}$ provided $\EeE(G)$ is dense. We want to prove the `converse' of this statement, namely

\begin{thm}{PV00}
A subnormed topological PV group $G$ is of class $\OOo_{00}$ iff it may be enlarged to a subnormed PV group $\widetilde{G}$
such that $\EeE(\widetilde{G})$ is dense in $\widetilde{G}$. Moreover, if $G \in \Gg_r(N) \cap \OOo_{00}$, the above
$\widetilde{G}$ may be chosen so that $\widetilde{G} \in \Gg_r(N)$.
\end{thm}

\THM{PV00} is an almost immediate consequence of

\begin{lem}{PV00}
Let $a \in \CcC(G) \setminus \{0\}$ be such that $a \in G_{0*} \cap G_{*0}$. For every $\epsi > 0$ there is a subnormed PV
group $(\bar{G},+,*,\|\cdot\|)$ enlarging $G$ and $b \in \EeE(\bar{G})$ such that $\|a - b\| \leqsl \epsi$. What is more,
if $\rank(a) < \infty$, then $\rank(b) = \rank(a)$; and $\|\cdot\|$ is bounded by $1$ provided so is the subnorm of $G$.
\end{lem}
\begin{proof}
We mimic the proofs of \LEM{N-PV} and \LEM{00-PV}. Let $\|\cdot\|_G$ denote the subnorm of $G$. If $a$ has finite rank,
let $N = \rank(a)$. Otherwise let $N \geqsl 2$ be as in \eqref{eqn:aux80}. Put $\kappa = \nabla$. Take $T > 1$ such that
for each $t \geqsl T$,
\begin{equation}\label{eqn:aux95}
N \|t * a\|_G \leqsl \epsi t.
\end{equation}
If $\rank(a) < \infty$, repeat the proof of \LEM{N-PV} (below we use the same notation as there) with $\|\cdot\|$
replaced by
\begin{multline*}
\|(g,f)\| = \inf \{\|g - \psi_G(u)\|_G + \epsi \|u\|_{\kappa} + A (\|\psi_H(u) - f\|_{\delta} \wedge T)\dd\\
u \in L_0[\ZZZ]\}.
\end{multline*}
If $\rank(a) = \infty$, repeat the proof of \LEM{00-PV} with
\begin{multline*}
\|(g,f)\| = \inf\{\|g - \sum_{j=1}^m s_j * (l_j a)\|_G + \epsi \sum_{j=1}^m \kappa(s_j)\\ + A \Bigl[\|\sum_{j=1}^m s_j *
(l_j \widehat{b}) - f\|_{\delta} \wedge T\Bigr]\dd\\ m \geqsl 0,\ s_1,\ldots,s_m \geqsl 0,\ l_1,\ldots,l_m \in \{-1,1\}\}.
\end{multline*}
In order to show that $\|(g,0)\| \geqsl \|g\|_G$ for $g \in G$, distinguish between two cases: when $\|\psi_H(u)\|_{\delta}
\leqsl T$ (respectively $\|\sum_{j=1}^m s_j * (l_j \widehat{b})\|_{\delta} \leqsl T$) and when the latter is false.
In the first case just copy the original proof of suitable lemma. Here we only show how to derive the second case.\par
First assume $\rank(a) < \infty$. Write $u = \sum_{k=1}^m s_k * \widehat{l_k}$ with $0 = s_0 < \ldots < s_m$ and
$l_1,\ldots,l_m \in \ZZZ \setminus \{0\}$. Let $z \in \{0,\ldots,m\}$ be such that $s_z \leqsl T$ and $s_{z+1} > T$. Put
$u' = \sum_{k=0}^z s_k * \widehat{l_k} \in L_0[\ZZZ]$. Notice that $\|\psi_H(u')\|_{\delta} \leqsl T$, $\|\psi_G(u)\|_G
\leqsl \|\psi_G(u')\|_G + \sum_{k=z+1}^m \|s_k * (l_k a)\|_G$ and, by \eqref{eqn:aux95}, $\|s_k * (l_k a)\|_G \leqsl
\epsi s_k = \epsi \kappa(s_k)$ for $k > z$. So,
\begin{multline*}
\|\psi_G(u)\|_G \leqsl \|\psi_G(u')\|_G + \epsi \sum_{k=z+1}^m \kappa(s_k) \leqsl\\ \leqsl \epsi \sum_{k=1}^z \kappa(s_k)
+ A \|\psi_H(u')\|_{\delta} + \epsi \sum_{k=z+1}^m \kappa(s_k) \leqsl \epsi \sum_{k=1}^m \kappa(s_k) + AT =\\
= \epsi \|u\|_{\kappa} + A (\|\psi_H(u)\|_{\delta} \wedge T)
\end{multline*}
which gives the assertion.\par
When $\rank(a) = \infty$, argue in a similar way. Assuming that $0 = s_0 < \ldots < s_m$ and taking $z \in \{0,\ldots,m\}$
for which $s_z \leqsl T$ and $s_{z+1} > T$, observe that $\|\sum_{k=1}^z s_k * (l_k \widehat{b})\|_{\delta} \leqsl T$ and
(by \LEM{00-PV})
$$
\|\sum_{k=1}^z s_k * (l_k a)\|_G \leqsl \epsi \sum_{k=1}^z \kappa(s_k)
+ A \|\sum_{k=1}^z s_k * (l_k \widehat{b})\|_{\delta}.
$$
Further, it follows from \eqref{eqn:aux95} that $\|s_k * (l_k a)\|_G \leqsl \epsi \kappa(s_k)$ for $k > z$ (recall that
$l_k \in \{-1,1\}$). Hence
\begin{multline*}
\|\sum_{k=1}^m s_k * (l_k a)\|_G \leqsl \|\sum_{k=1}^z s_k * (l_k a)\|_G + \epsi \sum_{k=z+1}^m \kappa(s_k) \leqsl\\
\leqsl \epsi \sum_{k=1}^m \kappa(s_k) + AT \leqsl \epsi \sum_{k=1}^m \kappa(s_k) + A \Bigl[\|\sum_{k=1}^m s_k *
(l_k \widehat{b})\|_{\delta} \wedge T\Bigr]
\end{multline*}
which is equivalent to $\|(g,0)\| \geqsl \|g\|_G$.\par
Finally, replace $\|\cdot\|$ by $\|\cdot\| \wedge 1$ if $\|\cdot\|_G \leqsl 1$. The details are left for the reader.
\end{proof}

\THM{main1} and the existence of the Gurari\v{\i} Banach space suggests to adapt these ideas in the realm of subnormed
topological pseudovector (Abelian) groups. It may be done in a few ways. One of them is proposed below.\par
We fix $r \in \{1,\infty\}$, $N \in \ZZZ_+ \setminus \{0\}$. For simplicity, denote by $\PpP\VvV\Gg_r(N)$ the class of all
subnormed \textbf{topological} PV groups which belong, as valued groups, to $\Gg_r(N)$ and are of class $\OOo_{00}$.
In particular, $\PpP\VvV\Gg_{\infty}(0)$ consists of all separable subnormed topological PV groups of class $\OOo_{00}$.
Let us call a subnormed PV group $H \in \PpP\VvV\Gg_r(N)$ \textit{of class $\EeE_r(N)_{fin}$} if $H = \EeE(H)$ and $H$ is
\textit{finitely generated}, that is, $H = \lin F$ for some finite subset of $H$. Observe that the subnorm of a PV group
of class $\EeE_r(N)_{fin}$ is bounded.\par
`Gurari\v{\i}-like' space in category of subnormed topological PV groups of class $\PpP\VvV\Gg_r(N)$ may be defined
as follows.

\begin{dfn}{Gurarii-like}
A subnormed PV group $(G,+,*,\|\cdot\|_G)$ is said to be \textit{$\EeE_r(N)$-Gurari\v{\i}} iff the following two
conditions are satisfied:
\begin{enumerate}[(GPV1)]
\item $(G,+,\|\cdot\|_G) \in \Gg_r(N)$, $G$ is complete and the set $\EeE(G)$ is dense in $G$,
\item whenever $(H,+,*,\|\cdot\|_H) \in \EeE_r(N)_{fin}$, $K$ is a finitely generated PV subgroup of $H$ and $\varphi\dd K
   \to G$ is an isometric linear homomorphism, for every $\epsi \in (0,1)$ there exists an $\epsi$-almost isometric linear
   homomorphism $\psi\dd H \to G$ such that $\|\varphi(x) - \psi(x)\|_G \leqsl \epsi$ for every $x \in K$.
\end{enumerate}
\end{dfn}

Notice that, thanks to (GPV1), every $\EeE_r(N)$-Gurari\v{\i} PV group is of class $\PpP\VvV\Gg_r(N)$ (by \PRO{cont} and
\THM{PV00}). Below we list other properties of them.

\begin{pro}{Gurarii}
Every $\EeE_r(N)$-Gurari\v{\i} pseudovector group is isometrically group isomorphic to $\GGG_r(N)$.
\end{pro}
\begin{proof}
Let $G$ be a subnormed $\EeE_r(N)$-Gurari\v{\i} PV group. Thanks to \PRO{completion}, it suffices to check that
$G_0 = \EeE(G)$ satisfies conditions (QG1) and (QG2) for $Q = \RRR_+$. But this easily follows from \THM{g-PV} and (GPV2).
(Indeed, note that the PV group $\bar{G}$ obtained from \THM{g-PV} may be constructed so that $\bar{G} \in
\EeE_r(N)_{fin}$.)
\end{proof}

We shall now show that every $\EeE_r(N)$-Gurari\v{\i} PV group satisfies the counterpart of (UEP). Precisely, that
in (GPV2) one may put $\epsi = 0$. We can provide this by repeating the arguments of Sections~2 and~3. The crucial point
here is that the PV groups of class $\EeE_r(N)_{fin}$ have bounded subnorms. We start with

\begin{lem}{APV2}
Let $(D_j,+,*,\|\cdot\|_j) \in \EeE_r(N)_{fin}$ ($j=1,2$) and $D_0$ be a finitely generated PV subgroup of $D_1$. Let $u\dd
D_0 \to D_2$ and $v\dd D_1 \to D_2$ be an isometric and, respectively, an $\epsi$-almost isometric linear homomorphism
(where $\epsi \in (0,1)$) such that $\|u - v\bigr|_{D_0}\|_{\infty} \leqsl \epsi$. Then there are a subnormed PV group
$(G,+,*,\|\cdot\|) \in \EeE_r(N)_{fin}$ and isometric linear homomorphisms $\psi_j\dd D_j \to G$ ($j=1,2$) such that
$\psi_1\bigr|_{D_0} = \psi_2 \circ u$ and $\|\psi_1 - \psi_2 \circ v\|_{\infty} \leqsl A\epsi$ where $A = 1 +
\diam(D_1,\|\cdot\|_1)$.
\end{lem}
\begin{proof}
Define $D$, $w_1$, $w_2$ and a semivalue $\lambda$ on $D$ in exactly the same way as in the proof of \LEM{A2}. Observe that
$D$ is a PV group, $\lambda$ is a semisubnorm and $w_1$ and $w_2$ are linear homomorphisms. To this end, let $G$ be
the quotient subnormed PV group $D / \lambda^{-1}(\{0\})$ and $\psi_1$ and $\psi_2$ be the homomorphisms naturally induced
by $w_1$ and $w_2$.
\end{proof}

Now repeating the proofs of \LEM{QG4} and \THM{UEP} we obtain the next two results.

\begin{lem}{APV3}
Let $G$ be an $\EeE_r(N)$-Gurari\v{\i} PV group, $(H,+,*,\|\cdot\|_H) \in \EeE_r(N)_{fin}$ be a subnormed PV group and $K$
its finitely generated PV subgroup. Further, let $\varphi\dd K \to G$ and $\psi\dd H \to G$ be, respectively, an isometric
and an $\epsi$-almost isometric linear homomorphism (where $\epsi \in (0,1)$) such that $\|\psi\bigr|_K
- \varphi\|_{\infty} \leqsl \epsi$. For every $\delta \in (0,\epsi)$ there is a $\delta$-almost isometric linear
homomorphism $\psi'\dd H \to G$ such that $\|\psi'\bigr|_K - \varphi\|_{\infty} \leqsl \delta$ and $\|\psi
- \psi'\|_{\infty} \leqsl C \epsi$ where $C = 3 + 2 \diam (H,\|\cdot\|_H)$.
\end{lem}

\begin{thm}{PVUEP}
Let $G$ be an $\EeE_r(N)$-Gurari\v{\i} PV group, let $H \in \EeE_r(N)_{fin}$ and $K$ be a finitely generated PV subgroup
of $H$. Every isometric linear homomorphism of $K$ into $G$ is extendable to an isometric linear homomorphism of $H$
into $G$.
\end{thm}

\begin{cor}{uniq}
Each two subnormed $\EeE_r(N)$-Gurari\v{\i} PV groups are isometrically isomorphic as subnormed PV groups.
\end{cor}

Thanks to \COR{uniq}, we may reserve a special symbol for a $\EeE_r(N)$-Gurari\v{\i} PV group. We shall denote it
(if it only exists) by $\PPP\VVV\GGG_r(N)$.\par
It follows from \THM{PV00} and \THM{PVUEP} that

\begin{pro}{univ}
Every subnormed PV group of class $\PpP\VvV\Gg_r(N)$ is embeddable into $\PPP\VVV\GGG_r(N)$ by means
of an isometric linear homomorphism.
\end{pro}

We believe $\PPP\VVV\GGG_r(N)$ exists for every $r$ and $N$. The existence of it for $N = 0$ would have an interesting
application in functional analysis (cf. \PRO{metriz}).

\begin{pro}{TVS}
If there exists $\PPP\VVV\GGG_r(0)$, there is a separable complete subnormed topological vector space $V$ of class
$\PpP\VvV\Gg_r(0)$ which is universal for all subnormed topological vector spaces of this class. Precisely, every subnormed
topological vector space of class $\OOo_{00}$ whose subnorm is bounded by $r$ admits an isometric linear homomorphism into
$V$. Moreover, every separable metrizable topological vector space admits a homeomorphic linear embedding into $V$.
\end{pro}
\begin{proof}
Put
$$
V = \{x \in \PPP\VVV\GGG_r(0)\dd\ (t + s) * x = t * x + s * y \textup{ for every } t,s \in \RRR_+\}.
$$
Equip $V$ with the subnorm inherited from the one of $\PPP\VVV\GGG_r(0)$. We leave this as a simple exercise that $V$
is a vector space (over $\RRR$) when the multiplication by reals is defined by $t v = t * v$ and $(-t) v = - (t * v)$
for $t \in \RRR_+$ and $v \in V$; and that the assertion of the proposition is fulfilled (use \PRO{metriz} and \PRO{univ}).
\end{proof}

Then next result corresponds to \PRO{embed}.

\begin{pro}{N,M}
If $M | N$, then $$\PPP\VVV\GGG_r(N,M) := \{x \in \PPP\VVV\GGG_r(N)\dd\ M \cdot x = 0\}$$ is isometrically isomorphic
as a subnormed PV group to $\PPP\VVV\GGG_r(M)$.
\end{pro}
\begin{proof}
For simplicity put $G = \PPP\VVV\GGG_r(N,M)$. We only need to check that the set $\EeE(G)$ is dense if $G$. Let $a \in G$,
$\epsi > 0$ and let $\|\cdot\|$ stand for the subnorm of $\PPP\VVV\GGG_r(N)$. By (GPV1), there is $b \in
\EeE(\PPP\VVV\GGG_r(N))$ such that $\|a - b\| \leqsl \epsi$. At the same time, according to \THM{PV00}, there is
a subnormed PV group $(G',+,*,\|\cdot\|')$ of class $\PpP\VvV\Gg_r(M)$ which enlarges $\lin \{a\}$ and such that $\EeE(G')$
is dense in $G'$. In particular, there is $a' \in \EeE(G')$ with $\|a - a'\|' \leqsl \epsi$. Now using a standard argument
of amalgamation, we see that there is a subnormed PV group $(D,+,*,\|\cdot\|_D)$ and isometric linear homomorphisms
$\psi\dd \lin \{a,b\} \to D$ and $\psi'\dd G' \to D$ which coincide on $\lin \{a\}$ (we omit the details). Put $H = \lin
\{\psi(b),\psi'(a')\} \subset D$ and $K = \lin \{\psi'(b)\} \subset H$. Notice that $H \in \EeE_r(N)_{fin}$ and apply
(GPV2) to get an $\epsi$-almost isometric linear homomorphism $\varphi\dd H \to \PPP\VVV\GGG_r(N)$ for which
$\|b - \varphi(a')\| \leqsl \epsi$. To this end, observe that $\varphi(a') \in \EeE(G)$ and $\|a - \varphi(a')\| \leqsl
2\epsi$.
\end{proof}

\begin{cor}{exist}
$\PPP\VVV\GGG_r(N)$ exists for each $N \geqsl 2$ provided $\PPP\VVV\GGG_r(0)$ exists.
\end{cor}

\begin{exm}{many}
Let $(H,+,*,\|\cdot\|)$ be a separable subnormed topological PV group. Then $(H,+,*,\|\cdot\|^{\alpha}) \in
\PpP\VvV\Gg_{\infty}(0)$ for every $\alpha \in (0,1)$. So, as in case of valued groups, every separable subnormed
topological PV group may be `approximated' by members of class $\PpP\VvV\Gg_{\infty}(0)$.
\end{exm}

\SECT{Concluding remarks}

In \PRO{no3} we have proved that the Urysohn universal space $\UUU$ admits no structure of an Abelian valued group
of exponent $3$. At the same time, as it was shown in \cite{pn1}, $\UUU$ admits a unique structure of a valued group
of exponent $2$. Taking into account these two remarks, we pose the following
\begin{quote}
\textbf{Conjecture.} \textit{For $r \in \{1,\infty\}$, $\GGG_r(2)$ is a unique (up to isometric group isomorphism) valued
Abelian group of finite exponent and of diameter $r$ which is Urysohn as a metric space.}
\end{quote}

If the above conjecture was false, it would be interesting to know for which exponents suitable groups exist and for which
exponents they are unique.\par
Cameron and Vershik \cite{c-v} have shown that the Urysohn universal metric space admits a structure of a monothetic valued
group. Taking this into account, the following may also be interesting.
\vspace{0.3cm}

\textbf{Question 1.} Are the topological groups $\GGG_1(0)$ and $\GGG_{\infty}(0)$ monothetic?
\vspace{0.3cm}

There is a fascinating phenomenom related to universal disposition property. Namely, the Urysohn universal space $\UUU$
is a unique (separable complete) metric space with universal disposition property for finite metric spaces;
$\GGG_{\infty}(0)$ is a unique (separable complete with dense set of finite rank elements) valued Abelian group with
universal disposition property for finite valued Abelian groups and $\PPP\VVV\GGG_{\infty}(0)$ (if only exists) is a unique
(separable complete with dense set of finite rank bounded Lipschitz elements) subnormed topological PV group with universal
disposition property for finitely generated bounded Lipschitz subnormed PV groups. It follows from \THM{Ur}
and \PRO{Gurarii} that all these three metric spaces are isometric and both the above valued groups are isometrically group
isomorphic. So, in fact we just enrich the structure of a single space. It seems to us important to answer the following
question with which we end the paper.
\vspace{0.3cm}

\textbf{Question 2.} Do there exist the PV groups $\PPP\VVV\GGG_1(0)$ and $\PPP\VVV\GGG_{\infty}(0)$~?

\end{document}